\documentclass[11pt]{article}
\usepackage[utf8]{inputenc}
\usepackage[english]{babel}
\usepackage[a4paper,margin=0.9in]{geometry}
\usepackage{amsfonts,amssymb,amsmath,amsthm,hyperref,graphicx,enumerate,wasysym,xcolor,fancyvrb}
\usepackage{titlesec}
\usepackage{mathtools,mathrsfs}
\mathtoolsset{showonlyrefs}
\usepackage[nottoc,numbib]{tocbibind}
\usepackage{chngcntr}
\usepackage[safeinputenc,maxbibnames=99]{biblatex}
\usepackage{csquotes}

\usepackage[toc,page]{appendix}

\title{Heteroclinic traveling waves of 2D parabolic Allen-Cahn systems}
\author{Ramon Oliver-Bonafoux\thanks{Sorbonne Université, Laboratoire Jacques-Louis Lions. 4 Place Jussieu, 75005 Paris (France). \linebreak Email: \texttt{ramon.oliver\_bonafoux@upmc.fr}}}

\numberwithin{equation}{section}

\def\R{{\mathbb R}}

\def\N{{\mathbb N}}

\def\CA{\mathcal{A}}

\def\CC{\mathcal{C}}

\def\CE{\mathcal{E}}

\def\CV{\mathcal{V}}

\def\CF{\mathcal{F}}

\def\CW{\mathcal{W}}

\def\FU{\mathfrak{U}}

\def\FP{\mathfrak{P}}
\def\FM{\mathfrak{M}}

\def\scrH{\mathscr{H}}
\def\scrL{\mathscr{L}}

\def\scrF{\mathscr{F}}

\def\sfC{\mathsf{C}}
\def\sfT{\mathsf{T}}

\def\Fu{\mathfrak{u}}

\def\Fm{\mathfrak{m}}

\def\Fw{\mathfrak{w}}
\def\Fq{\mathfrak{q}}

\def\Fd{\mathfrak{d}}
\def\Fe{\mathfrak{e}}
\def\Fr{\mathfrak{r}}
\def\Fi{\mathfrak{i}}

\def\bV{\mathbf{V}}

\def\bU{\mathbf{U}}
\def\bE{\mathbf{E}}
\def\bW{\mathbf{W}}

\def\be{\mathbf{e}}
\def\bm{\mathbf{m}}

\def\bv{\mathbf{v}}

\def\eps{\varepsilon}
\def\loc{\mathrm{loc}}
\def\supp{\mathrm{supp}}

\def\dist{\mathrm{dist}}

\titleclass{\subsubsubsection}{straight}[\subsection]

\newcounter{subsubsubsection}[subsubsection]
\renewcommand\thesubsubsubsection{\thesubsubsection.\arabic{subsubsubsection}}
\renewcommand\theparagraph{\thesubsubsubsection.\arabic{paragraph}} 

\titleformat{\subsubsubsection}
  {\normalfont\normalsize\bfseries}{\thesubsubsubsection}{1em}{}
\titlespacing*{\subsubsubsection}
{0pt}{3.25ex plus 1ex minus .2ex}{1.5ex plus .2ex}

\makeatletter
\renewcommand\paragraph{\@startsection{paragraph}{5}{\z@}%
  {3.25ex \@plus1ex \@minus.2ex}%
  {-1em}%
  {\normalfont\normalsize\bfseries}}
\renewcommand\subparagraph{\@startsection{subparagraph}{6}{\parindent}%
  {3.25ex \@plus1ex \@minus .2ex}%
  {-1em}%
  {\normalfont\normalsize\bfseries}}
\def\toclevel@subsubsubsection{4}
\def\toclevel@paragraph{5}
\def\toclevel@paragraph{6}
\def\l@subsubsubsection{\@dottedtocline{4}{7em}{4em}}
\def\l@paragraph{\@dottedtocline{5}{10em}{5em}}
\def\l@subparagraph{\@dottedtocline{6}{14em}{6em}}
\makeatother

\newtheorem{theorem}{Theorem}
\newtheorem{definition}{Definition}
\newtheorem{proposition}{Proposition}
\newtheorem{lemma}{Lemma}
\newtheorem{corollary}{Corollary}

\theoremstyle{definition}
\newtheorem{remark}{Remark}

\newtheorem{asu}{}

\newtheorem{hyp}{}

\numberwithin{proposition}{section}
\numberwithin{lemma}{section}
\numberwithin{remark}{section}
\numberwithin{definition}{section}
\numberwithin{example}{section}
\numberwithin{corollary}{section}
\numberwithin{figure}{section}

\begin{document}
\maketitle
\begin{abstract}
In this paper we show the existence of traveling waves $w:  [0,+\infty) \times \R^2 \to \R^k$ ($k \geq 2$) for the parabolic Allen-Cahn system
\begin{equation}
\partial_t w - \Delta w = -\nabla_u V(w) \mbox{ in } [0,+\infty) \times \R^2,
\end{equation}
satisfying some \textit{heteroclinic} conditions at infinity. The potential $V$ is a non-negative and smooth multi-well potential, which means that its null set is finite and contains at least two elements. The traveling wave $w$ propagates along the horizontal axis according to a speed $c^\star>0$ and a profile $\FU$. The profile $\FU$ joins as $x_1 \to \pm \infty$ (in a suitable sense) two locally minimizing 1D heteroclinics which have \textit{different} energies and the speed $c^\star$ satisfies certain uniqueness properties. The proof of variational and, in particular, it requires the assumption of an upper bound, depending on $V$, on the difference between the energies of the 1D heteroclinics.
\end{abstract}

\section{Introduction}
Consider the parabolic system of equations
\begin{equation}\label{parabolic-allencahn}
\partial_t w - \Delta w = -\nabla_u V(w) \mbox{ in } [0,+\infty) \times \R^2,
\end{equation}
where $V: \R^k \to \R$ is a smooth, non-negative, multi-well potential (see assumptions \ref{asu-sigma}, \ref{asu-infinity}, \ref{asu-zeros} later) and $w: [0,+\infty) \times \R^2 \to \R^k$, with $k \geq 2$. We seek for \textit{traveling wave} solutions to \eqref{parabolic-allencahn}. That is, we impose on $w$
\begin{equation}
\forall (t,x_1,x_2) \in [0,+\infty) \times \R^2, \hspace{2mm} w(t,x_1,x_2)=\FU(x_1-c^\star t,x_2),
\end{equation}
where $\FU: \R^2 \to \R^k$ is the \textit{profile} of the wave and $c^\star>0$ is the \textit{speed} of propagation of the wave, which occurs in the $x_1$-direction. Both the profile and the speed are the unknowns of the problem. Replacing in \eqref{parabolic-allencahn}, we find that the profile $\FU$ and $c^\star$ must satisfy the elliptic system
\begin{equation}\label{profile}
-c^\star\partial_{x_1}\FU-\Delta \FU=-\nabla_uV(\FU) \mbox{ in } \R^2.
\end{equation}

The system \eqref{parabolic-allencahn} can be seen as a reaction-diffusion system. Since the early works, motivated by questions from population dynamics, of Fisher \cite{fisher} and Kolmogorov, Petrovsky and Piskunov \cite{kpp}, devoted to a scalar reaction-diffusion equation in one space dimension known today as the Fisher-KPP equation, traveling and stationary waves are known to play a major role in the dynamics of reaction-diffusion problems. For instance, in \cite{fife-mcleod77,fife-mcleod81}, Fife and McLeod proved stability results for the equations considered in \cite{fisher,kpp}. Regarding higher dimensional problems (but always in the scalar case), existence results for traveling waves were obtained by Aronson and Weinberger \cite{aronson-weinberger} for equations with $\R^N$ as space domain and by Berestycki, Larrouturou and Lions \cite{berestycki-larrouturou-lions}, Berestycki and Nirenberg \cite{berestycki-nirenberg} for unbounded cylinders of the type $\R \times \omega$, with $\omega \subset \R^{N-1}$ a bounded domain. We also mention that asymptotic stability results (for a suitable class of perturbations) for traveling waves in the scalar Allen-Cahn equation in $\R^N$ were obtained by Matano, Nara and Taniguchi \cite{matano-nara-taniguchi}.

All the papers mentioned above are devoted to scalar equations and they rely on the application of the maximum principle and its related tools. As it is well-known, the maximum principle does not apply in general to systems of equations, meaning that other techniques are needed in order to study the existence of traveling waves (and their properties in case they exist) for systems. Different, more general, approaches had been taken in order to circumvent the lack of the maximum principle when dealing with parabolic systems. We refer to the books by Smoller \cite{smoller} and Volpert, Volpert and Volpert \cite{volpertx3}. One of these approaches consists on the use of variational methods. In the context of reaction-diffusion equations, this approach seems to appear for the first time in Heinze's PhD thesis \cite{heinze} (even though the existence of a variational framework for reaction diffusion problems was known since \cite{fife-mcleod77,fife-mcleod81}) and subsequently carried on also by Muratov \cite{muratov}, Lucia, Muratov and Novaga \cite{lucia-muratov-novaga}, Alikakos and Katzourakis \cite{alikakos-katzourakis} (see also Alikakos, Fusco and Smyrnelis \cite{alikakos-fusco-smyrnelis}), Risler \cite{risler,risler2021-1,risler2021-2} and, more recently, by Chen, Chien and Huang \cite{chen-chien-huang}. In the latter, the authors consider a parabolic Allen-Cahn system in a two dimensional strip $\R\times (-l,l)$ and find traveling waves which join a well and an approximation of an heteroclinic orbit in $(-l,l)$, for a class of symmetric triple-well potentials. Lastly, we mention that variational methods have also been applied to scalar reaction-diffusion equations, see for instance Bouhours and Nadin \cite{bouhours-nadin} for the case of heterogeneous equations as well as Lucia, Muratov and Novaga \cite{lucia-muratov-novaga04}. In this paper we shall also take a variational approach for dealing with the following question:

\textit{Question: } Assuming that there exist two \textit{heteroclinic orbits}, joining two fixed wells, with \textit{different} energy (defined in \eqref{1D_energy}) levels, does there exist a solution $(c,\FU)$ to \eqref{profile} such that $\FU$ \textit{joins} the two heteroclinic orbits at infinity, uniformly in $x_1$?

Heteroclinic orbits are curves $\Fq: \R \to \R^k$ which solve the equation
\begin{equation}
\Fq''=\nabla_u V(\Fq) \mbox{ in } \R
\end{equation}
and join two \textit{different} wells of $\Sigma$ at $\pm\infty$. Moreover, one asks that the \textit{1D energy} (i. e., the functional associated to the previous equation, see \eqref{1D_energy}) is finite. We show that, under the proper assumptions, the question we posed has an affirmative answer. Our motivation comes from two different sides:
\begin{enumerate}
\item Stationary heteroclinic-type solutions of \eqref{parabolic-allencahn} have been known to exist in several situations for a long time. Indeed, for a class of symmetric potentials, Alama, Bronsard and Gui in \cite{alama-bronsard-gui} showed the existence of a stationary wave (that is, a solution to \eqref{profile} with $c=0$) in the situation such that two heteroclinics with \textit{equal} energy levels exist and are global minimizers of the 1D energy. Their analysis was later extended to potentials without symmetry in several papers, which in some cases obtained similar results by means of different techniques. See Fusco \cite{fusco}, Monteil and Santambrogio \cite{monteil-santambrogio}, Schatzman \cite{schatzman}, Smyrnelis \cite{smyrnelis}. A key observation is that this problem can be seen as a heteroclinic orbit problem for a potential (the 1D energy, see \eqref{1D_energy}) defined in the infinite-dimensional space $L^2(\R,\R^k)$. Therefore, it is natural to aim a solving a connecting orbit problem for potentials defined in, say, Hilbert spaces and then deduce the original problem as a particular case. This is the approach taken in \cite{monteil-santambrogio} (in the metric space setting) and in \cite{smyrnelis} (in the Hilbert space setting).
\item Alikakos and Katzourakis \cite{alikakos-katzourakis}, showed the existence of traveling waves for a class of 1D parabolic systems of gradient type. Essentially, they assume that the potential possesses two local minima (one of them global) at \textit{different} levels. Hence, their potential is not of multi-well type in general. The profile of the traveling waves connects the two local minima at infinity and the determination of the speed becomes also part of the problem.
\end{enumerate}
The results of this paper follow by suitably merging the ideas of the previous items. More precisely, we formulate and provide solutions for a \textit{heteroclinic traveling problem} as that in \cite{alikakos-katzourakis} for potentials defined in an abstract Hilbert space. Then, we recover as a particular case the existence of a traveling wave solution for \eqref{parabolic-allencahn} with heteroclinic behavior at infinity,
\subsection*{Acknowledgments}
I wish to thank my PhD advisor Fabrice Bethuel for bringing this problem into my attention and for many useful comments and remarks during the elaboration of this paper.
\includegraphics[scale=0.3]{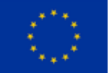}
This program has received funding from the European Union's Horizon 2020 research and innovation programme under the Marie Skłodowska-Curie grant agreement No 754362.
\tableofcontents
\section{The main results: Statements and discussions}\label{section_statements}

We now state the results of this paper. In Theorem \ref{THEOREM_main_TW}, which is the main result, existence of a traveling wave solution with speed $c^\star$ and profile $\FU$ is established as well as the uniqueness (in some sense) of $c^\star$ and the  $L^2$ exponential convergence of $\FU$ at the limit $x_1 \to +\infty$. For proving such a result, we use the bound assumption \ref{asu_perturbation}. In Theorem \ref{THEOREM_strong_bc}, we show that under the additional assumption \ref{asu_convergence}) the condition at infinity as $x_1 \to -\infty$ can be strengthened with respect to that given in Theorem \ref{THEOREM_main_TW} and in particular we show that the solution converges at $-\infty$ at an exponential rate. In Theorem \ref{THEOREM_uniform_convergence}, we show that under the previous assumptions we have uniform convergence of the solution in the $x_1$ and the $x_2$ direction.  Assumption \ref{asu_convergence} is also used for proving Theorem \ref{THEOREM_carac_speed}, which gives further properties on the speed $c^\star$. We conclude this section by describing the outline and main ideas of our proofs (subsection \ref{subs_proofs}) as well as giving examples of potentials that verify the assumptions of this paper (subsection \ref{subs_examples}).
\subsection{Basic assumptions and definitions}
Before stating the results, we recall some standard assumptions, definitions and results and we introduce some notation.  The multi-well potentials $V$ considered in this paper satisfy the following:
\begin{asu}\label{asu-sigma}
$V \in \CC^2_{\loc}(\R^k)$ and  $V \geq 0$ in $\R^k$. Moreover, $V(u) = 0$ if and only if $u \in \Sigma$, where, for some $l \geq 2$
\begin{equation}\label{DEF-Sigma}
\Sigma:= \{\sigma_1,\ldots ,\sigma_l\}.
\end{equation}
\end{asu}
\begin{asu}\label{asu-infinity}
There exist $\alpha_0, \beta_0, R_0 > 0$ such that for all $u \in \R^k$ with $\lvert u \rvert \geq R_0$ it holds $ \langle \nabla_u V(u), u \rangle \geq \alpha_0 \lvert u \rvert^2$ and $V(u) \geq \beta_0$.
\end{asu}
\begin{asu}\label{asu-zeros}
For all $\sigma \in \Sigma$, the matrix $D^2V(\sigma)$ is positive definite.
\end{asu}
As we advanced before, one considers the 1D energy functional
\begin{equation}\label{1D_energy}
E(q) := \int_\R e(q)(t) dt := \int_\R \left[ \frac{1}{2}\lvert q'(t) \rvert^2 + V(q(t)) \right]dt, \hspace{2mm} q \in H^1_{\loc}(\R,\R^k).
\end{equation}
Given a pair of wells $(\sigma_i,\sigma_j) \in \Sigma^2$, as done for instance in Rabinowitz \cite{rabinowitz93} we define
\begin{equation}\label{set-heteroclinic}
X (\sigma_i,\sigma_j):= \left\{ q \in H^1_{\mathrm{loc}}(\R,\R^k): E(q) < +\infty \mbox{ and } \lim_{t \to - \infty} q(t) = \sigma_i, \lim_{t \to +\infty}q(t)=\sigma_j\right\},
\end{equation} 
the set of curves in $\R^k$ connecting $\sigma_i$ and $\sigma_j$. The space $X(\sigma_i,\sigma_j)$ is a metric space when it is endowed with the $L^2$ and the $H^1$ distances, since $q-\tilde{q} \in H^1(\R,\R^k)$ whenever $q$ and $\tilde{q}$ belong to $X(\sigma_i,\sigma_j)$. If $\Fq$ is a critical point of the energy $E$ in $X(\sigma_i,\sigma_j)$, we say that $\Fq$ is an \textit{homoclinic orbit} when $\sigma_i=\sigma_j$ and that $\Fq$ is an \textit{heteroclinic orbit} when $\sigma_i \not = \sigma_j$. Define as well the corresponding infimum value
\begin{equation}\label{mij}
\Fm_{\sigma_i\sigma_j}:= \inf \{ E(q): q \in X(\sigma_i,\sigma_j)\}.
\end{equation}
If $\sigma^-$ and $\sigma^+$ are two distinct wells in $\Sigma$, it turns out that $\Fm_{\sigma^-\sigma^+}$ is not attained in general. We need to add the following assumption:
\begin{asu}\label{asu-wells}
We have that
\begin{equation}
\forall \sigma \in \Sigma \setminus \{ \sigma^-,\sigma^+ \}, \hspace{2mm} \Fm_{\sigma^-\sigma^+} < \Fm_{\sigma^-\sigma}+\Fm_{\sigma\sigma^+}.
\end{equation}
\end{asu}
Notice that one can always find a pair $(\sigma^-,\sigma^+) \in \Sigma^2$ such that \ref{asu-wells} holds. Assuming that \ref{asu-sigma}, \ref{asu-infinity}, \ref{asu-zeros} and \ref{asu-wells} hold, it is well known that there exists a minimizer of $E$ in $X(\sigma^-,\sigma^+)$. Moreover, we have the compactness of minimizing sequences as follows: For any $(q_n)_{n \in \N}$ in $X(\sigma^-,\sigma^+)$ such that $E(q_n) \to \Fm$, there exists $(\tau_n)_{n \in \N}$ in $\R$ and $\Fq \in X(\sigma^-,\sigma^+)$ such that $E(\Fq)=\Fm$ and, up to subsequences
\begin{equation}\label{MS_COMPACTNESS}
\lVert q_n(\cdot+\tau_n)-\Fq \rVert_{H^1(\R,\R^k)} \to 0 \mbox{ as } n \to +\infty.
\end{equation}
As we said before, this result is well-known. The earlier references are Bolotin  \cite{bolotin}, Bolotin and Kozlov \cite{bolotin-kozlov}, Bertotti and Montecchiari \cite{bertotti-montecchiari} and Rabinowitz \cite{rabinowitz89,rabinowitz92}, sometimes in a slightly different setting. Proofs and applications of the compactness property \eqref{MS_COMPACTNESS} are also given in Alama, Bronsard and Gui \cite{alama-bronsard-gui}, Alama et al \cite{alama-bronsard} and Schatzman \cite{schatzman}.

We fix the two wells $\sigma^-$ and $\sigma^+$ for the rest of the paper as well as $\Fm:=\Fm_{\sigma^-\sigma^+}$. According to the previous discussion, we have that the set
\begin{equation}\label{CF}
\CF:=\{ \Fq: \Fq \in X(\sigma^-,\sigma^+) \mbox{ and } E(\Fq)=\Fm\},
\end{equation}
is not empty. We term the elements of $\CF$ as \textit{globally minimizing heteroclinics} between $\sigma^-$ and $\sigma^+$. The term \textit{heteroclinics} comes from the fact that $\sigma^-$ and $\sigma^+$ are different. An important fact is that, due to the translation invariance of $E$ and $X(\sigma^-,\sigma^+)$, we have that if $\Fq \in \CF$, then for all $\tau \in \R$ it holds $\Fq(\cdot+\tau) \in \CF$. 
\subsection{Existence}
The assumptions \ref{asu-sigma}, \ref{asu-infinity}, \ref{asu-zeros} and \ref{asu-wells} stated before are classical. In order to obtain our results, we shall supplement them with the following one, which is more specific to the setting of this paper:
\begin{asu}\label{asu_unbalanced}
Assume that \ref{asu-sigma}, \ref{asu-infinity}, \ref{asu-zeros} and \ref{asu-wells} hold for the potential $V$. We keep the previous notations. We assume the following:
\begin{enumerate}
\item It holds that $\CF^-:=\CF= \{ \Fq^-(\cdot+\tau): \tau \in \R \}$ for some $\Fq^- \in X(\sigma^-,\sigma^+)$, where $\CF$ was defined in \eqref{CF}. We also set $\Fm^-:= \Fm$.
\item There exists $\Fm^+>\Fm^-$ and $\Fq^+ \in X(\sigma^-,\sigma^+)$ such that $E(\Fq^+)=\Fm^+$ and $\Fq^+$ is a local minimizer of $E$ with respect to the $H^1$ norm. We denote $\CF^+:= \{ \Fq^+(\cdot+\tau): \tau \in \R \}$.
\item We have the spectral nondegeneracy assumption due to Schatzman (\cite{schatzman}): For all $q \in X(\sigma^-,\sigma^+)$, let $A(q)$ be the unbounded linear operator in $L^2(\R,\R^k)$ with domain $H^2(\R,\R^k)$ defined as
\begin{equation}
A(q): v  \to -v''+D^2V(q)v,
\end{equation}
then, it holds that for any $\Fq \in \CF^- \cup \CF^+$ we have $\mathrm{Ker}(A(\Fq))=\{ \Fq' \}$. The fact that $\Fq' \in H^2(\R,\R^k)$ follows from the identity $\Fq'''=D^2V(\Fq)\Fq'$.
\end{enumerate}
\end{asu}
Notice that if we had $\Fm^+=\Fm^-$ we would be in the framework of Alama, Bronsard and Gui \cite{alama-bronsard-gui}, for which the 2D solution connecting $\Fq^-$ and $\Fq^+$ is stationary. Essentially, conditions 1. and 2. in \ref{asu_unbalanced} imply that $\Fq^-$ is a globally minimizing heteroclinic and $\Fq^+$ is a locally (but not globally) minimizing heteroclinic. Regarding assumption 3., introduced in \cite{schatzman}, it must be seen as a generalized non-degeneracy assumption for the minima. They are still degenerate critical points because every critical point of $E$ is degenerate due to the invariance by translations. Nevertheless, the assumption 3. implies that they are \textit{non-degenerate up to invariance by translations}. As shown in \cite{schatzman}, such a condition is generic in the sense that given a potential satisfying 1. and 2. one can always find a potential which verifies 1. 2. (with the same minimizers) and 3. and it is arbitrarily close to the given potential\footnote{One can think about the analogy between this property and the classical results for Morse functions (i. e., functions without degenerate critical points), which state that such type of functions are generic.}.  The most important consequence of this assumption, as proven in \cite{schatzman}, is the existence of two constants $\rho_0^+>0$ and $\rho_0^-$ such that
\begin{align}\label{rho_1}
\forall q \in L^2_{\loc}(\R,\R^k), &\hspace{2mm} \dist_{L^2(\R,\R^k)}(q,\CF^\pm) \leq \rho_0^\pm \\ &\Rightarrow \exists! \tau^\pm(q) \in \R: \lVert q-\Fq^\pm(\cdot + \tau^\pm(q)) \rVert_{L^2(\R,\R^k)}=\dist_{L^2(\R,\R^k)}(q,\CF^\pm)
\end{align}
and for some constant $\beta^\pm$ we have
\begin{align}\label{rho_2}
\forall q \in X(\sigma^-,\sigma^+), & \hspace{2mm} \dist_{H^1(\R,\R^k)}(q,\CF^\pm) \leq \rho_0^\pm\Rightarrow \dist_{H^1(\R,\R^k)}(q,\CF^\pm)^2 \leq \beta^\pm(E(q)-\Fm^\pm).
\end{align}
In fact, in \cite{schatzman} this is only proven for global minimizers but the proof readily extends to local ones as well. Notice that \eqref{rho_1} and \eqref{rho_2} state that the energy is \textit{quadratic} around $\CF^-$ and $\CF^+$, which is the infinite-dimensional analogue of \ref{asu-zeros}, taking into account the degeneracy generated by the group of translations. We will define for $r>0$ the sets
\begin{equation}\label{CFr}
\CF^\pm_{r}:= \{ q \in  L^2_{\loc}(\R,\R^k): \dist_{L^2(\R,\R^k)}(q,\CF^\pm) \leq r \},
\end{equation}
so that we can assume without loss of generality that $\CF^+_{\rho_0^+} \cap \CF^-_{\rho^-_0}=\emptyset$. See Figure \ref{Figure_unbalanced} for an explanatory design of \ref{asu_unbalanced}. Let us now assume that $\Fm^+-\Fm^-$ is bounded above as follows:

\begin{figure}[h!]
\centering
\includegraphics[scale=0.65]{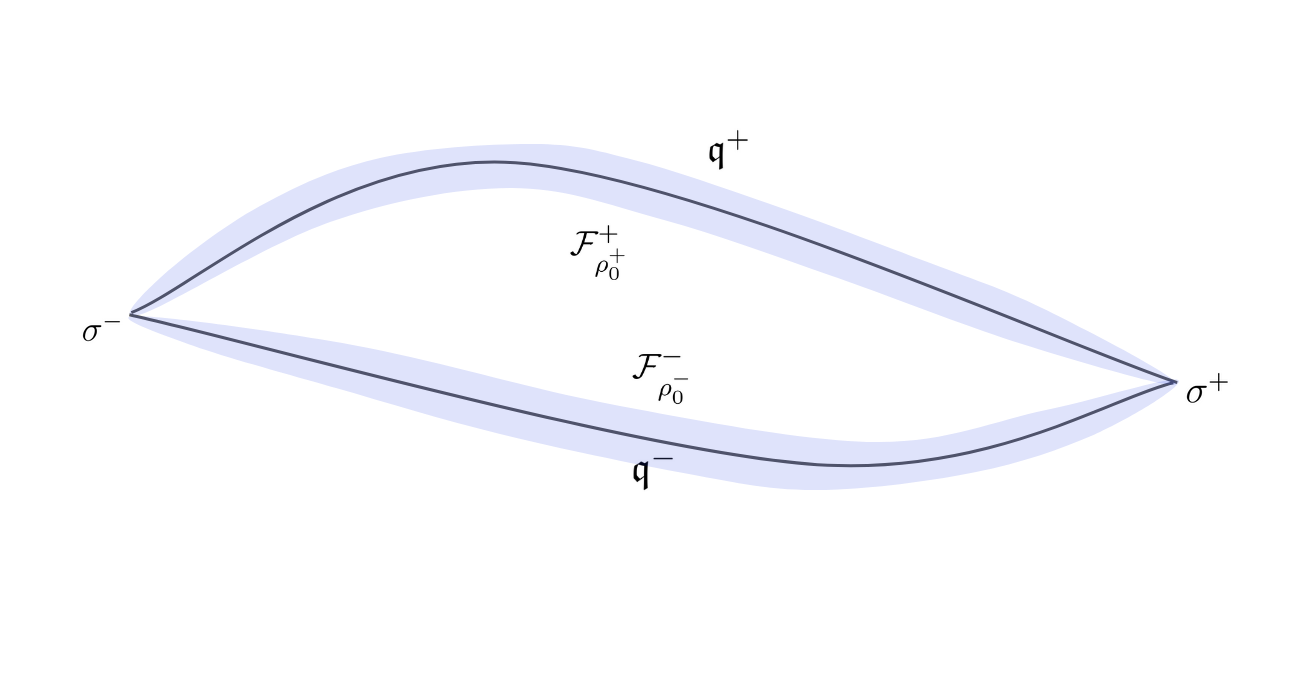}
\caption{Situation described by \ref{asu_unbalanced}. The curves correspond to the traces of $\Fq^-$ and $\Fq^+$ as indicated. The shadowed regions correspond to the traces of the functions in $\CF^-_{\rho_0^-}$ and $\CF^+_{\rho_0^+}$, which are a neighborhood of $\CF^-$ and $\CF^+$ respectively.}\label{Figure_unbalanced}
\end{figure}
 
\begin{asu}\label{asu_perturbation}
Assume that \ref{asu_unbalanced} holds and, moreover,
\begin{equation}
0<\Fm^+ - \Fm^- < E_{\max},
\end{equation}
where $E_{\max}$ will be defined later in \eqref{Fb_max}. Moreover, assume that
\begin{equation}\label{perturbation_ls}
\{ q \in X(\sigma^-,\sigma^+): E(q) < \Fm^+ \} \subset \CF^{-}_{\rho^-_0/2},
\end{equation}
with $ \CF^{-}_{\rho^-_0/2}$ as in \eqref{CFr}.
\end{asu}
See also Figure \ref{Figure_perturbation}. Essentially, \ref{asu_perturbation} requires that $\Fm^--\Fm^+$ is not too large and the bound is given by a constant $E_{\max}$ that can be computed through the constants produced in \eqref{rho_1} and \eqref{rho_2} as a consequence of \ref{asu_unbalanced}. If \ref{asu_perturbation} holds, then we are able to answer the question that we posed at the beginning of the paper in a positive way. More precisely, recall the equation of the profile:
\begin{equation}\label{profile_recall}
-c\partial_{x_1}\FU-\Delta\FU=-\nabla_u V(\FU) \mbox{ in } \R^2
\end{equation}
and consider the conditions at infinity
\begin{equation}\label{weak_bc-}
\exists L^- \in \R, \forall x_1\leq L^-, \hspace{2mm} \FU(x_1,\cdot) \in \CF^{-}_{\rho^-_0/2},
\end{equation}
\begin{equation}\label{weak_bc+}
\exists L^+ \in \R, \forall x_1\geq L^+, \hspace{2mm} \FU(x_1,\cdot) \in \CF^{+}_{\rho^+_0/2}
\end{equation}
As stated before, our proof is variational, which implies that the profile $\FU$ can be characterized as a critical point of a functional. The variational framework is as follows: assume that \ref{asu_perturbation} holds and set
\begin{align}\label{space_S}
S:= \{ U \in H^1_{\loc}(\R,L^2(\R,\R^k)): \exists L \geq 1, & \forall x_1 \geq L, \hspace{2mm} U(x_1,\cdot) \in \CF^+_{\rho_0^+/2} \\
& \forall x_1 \leq -L, \hspace{2mm} U(x_1,\cdot) \in \CF^-_{\rho^-_0/2} \}.
\end{align}
For $U \in S$ and $c>0$ we define the energy
\begin{equation}
E_{2,c}(U):= \int_\R \left(\int_\R \frac{\lvert \partial_{x_1} U(x_1,x_2) \rvert^2}{2}dx_2+(E(U(x_1,\cdot))-\Fm^+)\right)e^{cx_1}dx_1.
\end{equation}
Formally, critical points of $E_{2,c}$ give rise to solutions of \eqref{profile_recall}. If $U \in S$, we can define the translated function $U^\tau:= U(\cdot+\tau,\cdot)$ for $\tau \in \R$. Then, for all $c>0$ we have
\begin{equation}
E_{2,c}(U^\tau)=e^{-c\tau}E_{2,c}(U)
\end{equation}
which implies that
\begin{equation}
\forall c>0, \hspace{2mm} \inf_{U \in S}E_{2,c}(U) \in \{-\infty,0\}.
\end{equation}
We have by now introduced the notations which allow us to state the main result of this paper:
\begin{theorem}[Main Theorem]\label{THEOREM_main_TW}
Assume that \ref{asu_perturbation} holds. Then, we have:
\begin{enumerate}
\item \textbf{Existence}. There exist $c^\star>0$ and $\FU \in \CC^{2,\alpha}(\R^2,\R^k) \cap S$, $\alpha \in (0,1)$, which fulfill \eqref{profile_recall}. The profile $\FU$ satisfies the conditions at infinity \eqref{weak_bc-} and \eqref{weak_bc+} as well as the variational characterization
\begin{equation}\label{FU_min}
E_{2,c^\star}(\FU)=0=\inf_{U \in S}E_{2,c^\star}(U).
\end{equation}
\item \textbf{Uniqueness of the speed}. The speed $c^\star$ is unique in the following sense: Assume that $\overline{c^\star}>0$ is such that
\begin{equation}
\inf_{U \in S}E_{2,\overline{c^\star}}(U) = 0
\end{equation}
and that $\overline{\FU} \in S$ is such that $(\overline{c^\star},\overline{\FU})$ solves \eqref{profile_recall} and $E_{2,\overline{c^\star}}(\overline{\FU})<+\infty$. Then, $\overline{c^\star}=c^\star$. 
\item \textbf{Exponential convergence}. The convergence of $\FU$ at $+\infty$ is exponential with respect to the $L^2$-norm. More precisely, there exists $\FM^+>0$ and $\tau^+ \in \R$ such that for all $x_1 \in \R$
\begin{equation}\label{exp_convergence_main}
\lVert \FU(x_1,\cdot)-\Fq^+(\cdot+\tau^+) \rVert_{L^2(\R,\R^k)} \leq \FM^+ e^{-c^\star t}.
\end{equation}
\end{enumerate}
\end{theorem}
\begin{remark}
The existence part of Theorem \ref{THEOREM_main_TW} states that there exists a solution $(c^\star,\FU)$ such that $\FU$ is a global minimizer of $E_{c^\star}$ in $S$. We also have that the speed $c^\star$ is unique for some class of solutions, namely for finite energy solutions and speeds for which the corresponding energy is bounded below in $S$. In particular, $c^\star$ is unique among the class of globally minimzing profiles. In other words, if $c>0$ is such that the infimum of $E_c$ in $S$ is attained, then $c=c^\star$. This is analogous to what it was shown in Alikakos and Katzourakis \cite{alikakos-katzourakis}. As explained in the introduction, the main drawback of our approach is the existence assumption \ref{asu_perturbation}. In particular, the definition of the upper bound $E_{\max}$ is technical and it is possible that in several situations it could be small. Nevertheless, in subsection \ref{subs_examples} we show that there exists examples of potentials for which \ref{asu_perturbation} holds.
\end{remark}
\begin{figure}[h!]
\centering
\includegraphics[scale=0.35]{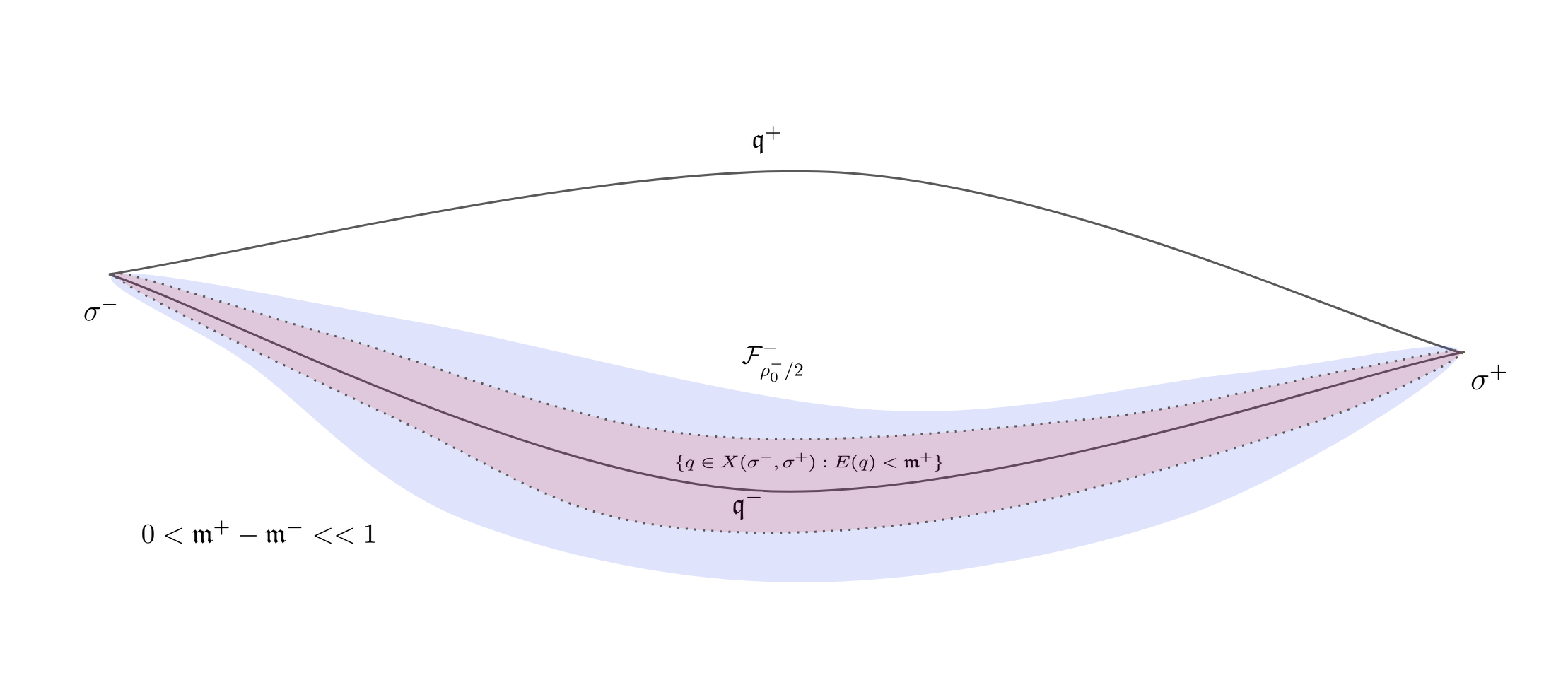}
\caption{Representation of \ref{asu_perturbation}. While the larger shadowed region corresponds to $\CF^-_{\rho_0^-/2}$, the smaller one which is contained inside represents the set $\{ q \in X(\sigma^-,\sigma^+): E(q) < \Fm^+\}$. Moreover, the value $\Fm^+-\Fm^-$ must be smaller than $E_{\max}$, defined in \eqref{Fb_max}.}\label{Figure_perturbation}
\end{figure}
\subsection{Conditions at infinity}
The solutions given by Theorem \ref{THEOREM_main_TW} satisfy the conditions at infinity \eqref{weak_bc-} and \eqref{weak_bc+}. As we can see, condition \eqref{weak_bc-} is more imprecise than expected, as it only states that $\FU(x_1,\cdot)$ is not far from $\CF^-$ with respect to the $L^2$ distance when $x_1$ is close enough to $-\infty$. In particular, we cannot ensure that $\FU$ is really heteroclinic, in the sense of connecting two stable states as $x_1 \to \pm \infty$. Therefore, it is reasonable to wonder if we can establish a behavior of the type
\begin{equation}\label{strong_bc-}
\exists \tau^- \in \R, \hspace{2mm} \lVert \FU(x_1,\cdot)-\Fq^-(\cdot+\tau^-) \rVert_{L^2(\R,\R^k)} \to 0 \mbox{ as } x_1 \to -\infty,
\end{equation}
so that $\FU$ is an actual heteroclinic. While \eqref{strong_bc-} was established in \cite{alikakos-katzourakis}, their argument does not seem to apply to the infinite-dimensional setting, which means that new ideas are needed. We have been able to show that \eqref{strong_bc-} holds under the following additional assumption:
\begin{asu}\label{asu_convergence}
We have that assumption \ref{asu_perturbation} holds and, additionally:
\begin{equation}\label{convergence_1}
\Fm^+-\Fm^- < \frac{( \mu^-\Fd_0)^2}{2},
\end{equation}
where the constants $\Fd_0$ and $\mu^-$ are defined later in \eqref{Fd0} and \eqref{mu-} respectively.
\end{asu}
Assumption \ref{asu_convergence} is not too restrictive (at least with respect to the assumptions we already have), since an upper bound on $\Fm^+-\Fm^-$ is already imposed in \ref{asu_perturbation}, meaning that at worst one only needs to lower it. The definition of the constant $\Fd_0$ is essentially technical depends only on the distance between the sets $\CF^-$ and $\CF^+$, while $\mu^-$ depends only on local information around $\CF^-$. Anyway, the result given by \ref{asu_convergence} writes as follows:
\begin{theorem}\label{THEOREM_strong_bc}
Assume that \ref{asu_perturbation} and \ref{asu_convergence} hold. Let $(c^\star,\FU)$ be the solution given by Theorem \ref{THEOREM_main_TW}. Then, $\FU$ satisfies the stronger condition \eqref{strong_bc-}. Moreover, it holds that $c^\star < \mu^-$, $\mu^-$ to be defined later in \eqref{mu-}, and there exists $\FM^->0$ such that for all $x_1 \in \R$
\begin{equation}\label{exp_convergence_-}
\lVert \FU(x_1,\cdot)- \Fq^-(\cdot+\tau^-) \rVert_{L^2(\R,\R^k)} \leq \FM^- e^{(\mu^--c^\star)x_1}
\end{equation}
\end{theorem}
The natural question is whether \eqref{exp_convergence_main} in Theorem \ref{THEOREM_main_TW} and \eqref{exp_convergence_-} in Theorem \ref{THEOREM_strong_bc} can be improved. In particular, whether the $L^2$-norm can be replaced by the $H^1$-norm. We conjecture that the answer to this question is positive, but we do not have a proof of this fact. However, as one can check in Smyrnelis \cite{smyrnelis} and Fusco \cite{fusco}, such a fact holds for the balanced 2D heteroclinic solution. They obtain these properties by combining standard elliptic estimates with some properties which are intrinsic to minimal solutions of the elliptic system \eqref{profile} with $c^\star=0$. See the results of section 4 in Alikakos, Fusco and Smyrnelis \cite{alikakos-fusco-smyrnelis}, mainly based on Alikakos and Fusco \cite{alikakos-fusco-15}. The main obstacle is that even if one was able to extend their analysis to the case $c^\star>0$, a crucial hypothesis of in their results is that solutions are minimal with respect to compactly supported perturbations, but the solution of Theorem \ref{THEOREM_main_TW} is only \textit{locally} minimizing (due to the fact that $\Fq^+$ is a local minimizer of the 1D energy). Therefore, we leave this question open. Nevertheless, besides the $L^2$-convergence rates \eqref{exp_convergence_main} and \eqref{exp_convergence_-} on can prove uniform convergence both in the $x_1$ and the $x_2$ direction:
\begin{theorem}\label{THEOREM_uniform_convergence}
Assume that \ref{asu_perturbation} holds. Let $(c^\star,\FU)$ be the solution given by Theorem \ref{THEOREM_main_TW}. Then, we have that
\begin{equation}\label{uniform_convergence_1}
\lim_{x_1 \to +\infty}\lVert \FU(x_1,\cdot)-\Fq^+(\cdot+\tau^+) \rVert_{L^\infty(\R,\R^k)}=0
\end{equation}
and for all $L \in \R$ we have
\begin{equation}\label{uniform_convergence_2}
\lim_{x_2 \to \pm \infty} \lVert \FU(\cdot,x_2)-\sigma^\pm \rVert_{L^\infty([L,+\infty),\R^k)}=0.
\end{equation}
If, moreover, \ref{asu_convergence} holds, then we have
\begin{equation}\label{uniform_convergence_3}
\lim_{x_1 \to -\infty}\lVert \FU(x_1,\cdot)-\Fq^-(\cdot+\tau^-) \rVert_{L^\infty(\R,\R^k)}=0
\end{equation}
and \eqref{uniform_convergence_2} can be improved into
\begin{equation}\label{uniform_convergence_4}
\lim_{x_2 \to \pm \infty}\lVert \FU(\cdot,x_2)-\sigma^\pm \rVert_{L^\infty(\R,\R^k)}=0.
\end{equation}
\end{theorem}
\subsection{Min-max characterization of the speed}
We provide here a min-max characterization of the speed $c^\star$ and other related properties which are summarized in Theorem \ref{THEOREM_carac_speed}. The idea of providing a variational characterization for the speed of traveling waves in reaction-diffusion systems can be traced back to Heinze \cite{heinze}, Heinze, Papanicolau and Stevens \cite{heinze-papanicolau-stevens} and it was used later in several other papers \cite{alikakos-katzourakis,bouhours-nadin,lucia-muratov-novaga,lucia-muratov-novaga04,muratov}.
\begin{theorem}\label{THEOREM_carac_speed}
Assume that \ref{asu_perturbation} and \ref{asu_convergence} hold. Let $(c^\star,\FU)$ be the solution given by Theorem \ref{THEOREM_main_TW}. Then for any $\tilde{\FU} \in S_i$ such that
\begin{equation}
E_{2,c^\star}(\tilde{\FU})=0
\end{equation}
we have that $(c^\star,\tilde{\FU})$ solves \eqref{profile_recall} and
\begin{equation}\label{speed_formula}
c^\star= \frac{\Fm^+-\Fm^-}{\int_{\R^2} \lvert \partial_{x_1}\tilde{\FU}(x_1,x_2) \rvert^2dx_2dx_1}.
\end{equation}
In particular, the quantity $\int_{\R^2}\lvert \partial_{x_1}\tilde{\FU}(x_1,x_2) \rvert^2 dx_2dx_1$ is well-defined and constant among the set of minimizers of $E_{2,c}$ in $S$. Moreover, it holds
\begin{equation}\label{speed_variational}
c^\star=\sup\{ c>0: \inf_{U \in S}E_{2,c}(U)=-\infty\}=\inf\{c>0: \inf_{U \in S}E_{2,c}(U)=0\}
\end{equation}
and we have the bound
\begin{equation}
c^\star \leq \frac{\sqrt{2(\Fm^+-\Fm^-)}}{\Fd_0}< \min\left\{\frac{\sqrt{2E_{\max}}}{\Fd_0}, \mu^-\right\}
\end{equation}
where $\Fd_0$, $\mu^-$ and  $E_{\max}$ will be defined later in \eqref{Fd0}, \eqref{mu-} and \eqref{Fb_max} respectively and the second inequality follows from the bound on $\Fm^+-\Fm^-$ given by \ref{asu_perturbation} and \ref{asu_convergence}.
\end{theorem}
\begin{remark}
Notice that the conditions at infinity imply that any $U \in S$ is such that
\begin{equation}
\int_{\R^2} \frac{\lvert \partial_{x_1}U(x_1,x_2) \rvert^2}{2}dx_2dx_1>0
\end{equation}
\end{remark}
As it can be seen, Theorem \ref{THEOREM_carac_speed} shows that the speed $c^\star$ is characterized by the explicit formula \eqref{speed_formula}, which nevertheless requires knowledge on a profile $\tilde{\FU}$. However, one also has the variational characterization \eqref{speed_variational}, which does not involve any information on the profiles. Indeed, one only needs to be able to compute the infimum of the energies with $c>0$ as a parameter. Moreover, notice that combining \eqref{speed_variational} with the uniqueness part of Theorem \ref{THEOREM_main_TW}, we obtain that if $\overline{c}>c^\star$ and $(\overline{c},\overline{\FU})$, with $\overline{\FU} \in S$, solves \eqref{profile_recall}, then $E_{2,\overline{c}}(\overline{\FU})=+\infty$, which is actually a contradiction. On the contrary, if we take $\overline{c}<c^\star$, then \eqref{speed_variational} implies that $\inf_{U \in S}E_{2,\overline{c}}(U)=-\infty$, meaning that Theorem \ref{THEOREM_main_TW} does not apply and nothing else can be said.
\subsection{Definition of the upper bounds}\label{subsection_constants}
We will now define some important numerical constants which are necessary in order to formulate assumptions \ref{asu_perturbation} and \ref{asu_convergence}. Assume first that \ref{asu_unbalanced} holds. Let $\rho_0^\pm$ as in \eqref{rho_1} and \eqref{rho_2}. Recall that we chose $\rho_0^+$ and $\rho_0^-$ such that
\begin{equation}
\CF^+_{\rho_0^+} \cap \CF^-_{\rho_0^-}  = \emptyset
\end{equation}
and, since those two sets (see the definition in \eqref{CFr}) are $L^2$-closed due to the local compactness of the sets $\CF^-$ and $\CF^+$, we have that
\begin{equation}\label{Fd0}
\Fd_0:= \dist_{L^2(\R,\R^k)}(\CF^+_{\rho_0^+/2}, \CF^-_{\rho_0^-/2})
\end{equation}
is positive. Therefore, as we advanced before, one can see that the constant $\Fd_0$ depends only on the distance between the two families of minimizing heteroclincs. Next, under \ref{asu_unbalanced}, recall the constants $\beta^\pm$ from \eqref{rho_2}. Set
\begin{equation}\label{overline_beta}
\overline{\beta^\pm}:= \frac{1}{2}(\beta^\pm)^2((\beta^\pm)^2+(\beta^\pm+1)^2)>0
\end{equation}
and, subsequently
\begin{equation}\label{mu-}
\mu^-:= \frac{1}{\beta^-+\overline{\beta^-}}>0
\end{equation}
which is the constant appearing in \ref{asu_convergence}. Of course, the nature of the definition given in \eqref{mu-} obeys to technical considerations. But $\mu^-$ should be thought as a constant depending only on the local behavior of the energy around $\CF^-$ and, in particular, independent on the behavior of the energy near $\CF^+$. Now let for $r \in (0, \rho^\pm_0]$
\begin{equation}\label{Fe}
\Fe^\pm_r := \inf \{ E(q): q \in X(\sigma^-,\sigma^+), \dist_{L^2(\R,\R^k)}(q,\CF^\pm) \in [r,\rho^\pm_0]\}.
\end{equation}
Known results which follow from compactness of minimizing sequences (see for instance Schatzman \cite{schatzman}) imply that $\Fe^\pm_r>0$. Moreover, we also have that for $r \in (0,\rho^\pm_0]$ there exists $\nu^\pm(r)>0$ such that
\begin{equation}\label{nu_r}
\forall q \in \CF^\pm_{\rho_0^\pm/2}, \hspace{2mm} E(q)-\Fm^\pm \leq \nu^\pm(r) \Rightarrow \dist_{H^1(\R,\R^k)}(q, \CF^\pm) \leq r.
\end{equation}
This leads to define the constants
\begin{equation}
\delta_0^-:= \min \left\{ \sqrt{e^{-1}\frac{\rho_0^-}{4}\sqrt{2(\Fe^-_{\rho_0^-/4}-\Fm^-)}},\frac{\rho_0^-}{4} \right\}>0,
\end{equation}
\begin{equation}
\Fr^-:=\frac{\rho_0^-}{\beta^-+1}>0
\end{equation}
and
\begin{equation}\label{Fb_max}
E_{\max}:= \frac{1}{(\beta^-)^2(\beta^-+1)}\min \left\{ \frac{(\delta_0^-)^2}{4},\Fe^-_{\delta_0^-}-\Fm^-,\nu^-(\Fr^-),\nu^-(\delta_0^-) \right\}>0
\end{equation}
which is the constant appearing in \ref{asu_perturbation}. Again, the definition on $E_{\max}$ is essentially due to technical reasons, but it must be thought as a constant which only depends on local information around $\CF^-$.
\subsection{Methods and ideas of the proofs}\label{subs_proofs}
The main result of this paper is Theorem \ref{THEOREM_main_TW}, which establishes the existence of a solution $(c^\star,\FU)$, with the profile $\FU$ satisfying the heteroclinic asymptotic conditions \eqref{weak_bc-}, \eqref{weak_bc+}. We also prove an exponential rate of convergence for the profile at $+\infty$ (with respect to the $L^2$-norm). We finally show that the speed $c^\star$ has some uniqueness properties. Important properties on the profile and the speed, as well as improvements on the results under additional assumptions, are also established in Theorems \ref{THEOREM_strong_bc}, \ref{THEOREM_uniform_convergence} and \ref{THEOREM_carac_speed}.

As already stated, the proof of our results follows by bringing together two different lines of research, see items 1. and 2. in the introduction. More precisely, in the spirit of \cite{monteil-santambrogio,schatzman}, we adapt the result of Alikakos and Katzourakis \cite{alikakos-katzourakis} (actually, we rather follow more closely the simplified version given in Alikakos, Fusco and Smyrnelis \cite{alikakos-fusco-smyrnelis}) to potentials defined in an abstract, possibly infinite-dimensional, Hilbert space and possessing two local minima at different levels. This abstract setting is established in section \ref{sect_abstract} and the main abstract results are Theorems \ref{THEOREM-ABSTRACT}, \ref{THEOREM_abstract_bc} and \ref{THEOREM_abstract_speed}. The proof of these results is found in section \ref{section_abs}. Assumption \ref{asu_unbalanced} guarantees that our main results (Theorems \ref{THEOREM_main_TW}-\ref{THEOREM_carac_speed}) are a particular case of the abstract results. Naturally, the advantage of proving the results in an abstract framework is that one can apply them to several problems different than the original one. In our case, the results in this paper apply to the 1D system
\begin{equation}
\partial_t w-\partial_x^2w=-\nabla_uW(w) \mbox{ in } [0,+\infty) \times \R,
\end{equation}
where $W$ is a smooth potential bounded below possessing two local and non-degenerate minima at different levels. As said before, this is formally the system considered in \cite{alikakos-katzourakis}, but the results of this paper allow to somewhat relax the non-degeneracy assumption used in \cite{alikakos-katzourakis}. More details, as well as other extensions, are given in our companion paper \cite{oliver-bonafoux-tw-bis}.

Generalizing the result from \cite{alikakos-katzourakis} for curves taking values in a more general, possibly infinite-dimensional, Hilbert space raises several additional difficulties. A detailed outline of our proof is given in subsection \ref{subs_scheme}, but let us here try to motivate the main difficulties of the problem we are facing.

As pointed out before, the approach in \cite{alikakos-katzourakis} is variational. A family of weighted energy functionals (essentially those introduced in Fife and McLeod \cite{fife-mcleod77,fife-mcleod81}) depending on a speed parameter $c>0$ is considered. In order to make the functionals well defined in the space of curves that connect the minima, the global minimum of the potential must be negative and the local one must be zero, which is always true up to an additive constant. As a consequence, one deals with an energy density which changes sign, which is in contrast with the equal depth (balanced) case, for which the energy is always non-negative. Recall that finite energy 1D connecting heteroclinics between two wells at the same level must be stationary. Another difficulty of the heteroclinic traveling wave existence problem comes from the fact that not only the profile but also the speed of the wave is an unknown as well. 

The method used in \cite{alikakos-katzourakis} is an adaptation of that introduced in Alikakos and Fusco \cite{alikakos-fusco} for the equal depth case. This method consists on considering families of solutions with prescribed behavior outside an interval of length $2T$ (namely, they are forced to stay close to the respective minimum) and minimizing the weighted energy functionals seeing the speed as parameter. Since compactness is restored due to the constrains, the problem has a solution for each $c>0$ and $T  \geq 1$. These ideas can be adapted to our setting without major difficulty. The next step consists on determining the solution speed $c^\star$ and then showing that for $c^\star$ and a suitable $T$ the corresponding constrained minimizer does not meet the constraints, meaning that it is an actual solution. Since the energy functionals change sign, one needs to show that the constrained minimizers do not oscillate between positive and negative regions of the energy (which would produce compensations) inside arbitrarily long intervals as $T$ goes to infinity. In order to show that, the authors in \cite{alikakos-katzourakis} assume that the local minima are isolated and that in the negative region of the functional one has strict radial monotonicity with respect to the global minimum. Subsequently, they use this property in combination with the ODE system and minimality arguments in order to exclude oscillations. 

For several reasons, the previous idea does not seem available in our setting without substantial modifications. Despite the fact that, as we show in \cite{oliver-bonafoux-tw-bis}, one can adapt the assumption of \cite{alikakos-katzourakis} for potentials in infinite-dimensional spaces with possibly degenerate minima, in this case we have trouble showing that our original problem can be put as a particular case of the abstract one. In other words, it does not seem reasonable to expect that such adaptation of the radial monotonicity assumption of \cite{alikakos-katzourakis} would be met in our original problem. Indeed, one would need to prove some kind of radial monotonicity for the energy $E$ (see \eqref{1D_energy}), in some suitable subset. We think that this might be too restrictive and we cannot prove it even for simple explicit examples. The difficulty comes from the fact that, while in the finite-dimensional case one can directly modify the potential, the level sets of $E$ depend on a rather indirect way on the potential $V$ and they are infinite-dimensional manifolds. Therefore, the most important difficulty of our problem is to replace the radial monotonicity assumption of \cite{alikakos-katzourakis} by another one which can be met in our situation and which allows to obtain a similar type of conclusion (namely, exclude oscillatory behavior for the constrained minimizers in arbitrarily large intervals). We have been able to provide one assumption, \ref{asu_perturbation}, which plays this role. It consists on imposing an upper bound on the difference between the energy levels. This upper bound is (the abstract version of) the constant $\CE_{\max}$, defined in \eqref{Fb_max}, subsection \ref{subsection_constants}. It enables us to exclude oscillations on the minimizers because the (renormalized) energy is positive outside the region in which the solution is constrained. Once oscillations are excluded, we conclude the proof as in \cite{alikakos-katzourakis}. 

The main drawback of our proof is that the computation of the upper bound in \ref{asu_perturbation} is not straightforward and it obeys technical considerations, as the definition of $\CE_{\max}$ \eqref{Fb_max} in subsection \ref{subsection_constants} shows. In particular, we cannot exclude the possibility that $\CE_{\max}$ is \textit{small}. However, in subsection \ref{subs_examples} we give a general method in order to obtain potentials for which the corresponding energy functional satisfies the bound assumption \ref{asu_perturbation}. Essentially, one considers a potential for which two different globally minimizing heteroclinics exist (which implies $k \geq 2$) and then modifies it in a suitable manner. It would be also interesting to know whether our assumption \ref{asu_perturbation} is only technical or rather there is some kind of obstruction for existence when the difference between the energy of the heteroclinics is too large. We think that the answer to this question is possibly related to the loss of compactness for the 1D energy functional (see \cite{oliver-bonafoux} and the references therein) and hence we conjecture that \ref{asu_perturbation} is not only technical (although it is likely non-optimal) and that in its absence some counterexamples might be found. It is also reasonable to conjecture that the removal of \ref{asu_perturbation} would imply the existence of traveling waves with more complicated behavior at infinity, for example approaching chains of connecting orbits (heteroclinic or homoclinic) stable in some suitable sense.

At the final stage of the proof, one needs to ensure that the solution obtained presents the suitable heteroclinic behavior at infinity. Moreover, we want to obtain more refined convergence results, see Theorems \ref{THEOREM_main_TW} and \ref{THEOREM_strong_bc}. This asymptotic analysis is delicate, as finite energy functions do not necessarily converge at all $-\infty$ and they converge exponentially at $+\infty$ but only with respect to the $L^2$-norm. Moreover, one needs to distinguish between $L^2$ convergence (which is weaker and does not imply convergence of the energy) and $H^1$ convergence. We deal with all these difficulties by using minimality of the solution (inspired for instance by \cite{smyrnelis}), which allows us to obtain the convergence at $-\infty$, exponentially with respect to the $L^2$-norm. Assumption \ref{asu_convergence} is needed in order to obtain the convergence at $-\infty$. Moreover, working in the main setting, we obtain Theorem \ref{THEOREM_uniform_convergence}, which shows that convergence is not only $L^2$ but also $L^\infty$, and not only according to $x_1$ but also $x_2$, the limit in this case being the wells $\sigma^\pm$.

We will use extensively the fact that the energy has \textit{good} properties at least on a neighborhood of the 1D minimizing heteroclinics. Essentially, those are the results and assumptions made by Schatzman \cite{schatzman}, which in our case are given in \ref{asu_unbalanced} and the discussion that follows. To our knowledge, these properties have not been shown to hold for the 2D heteroclinic solutions of \cite{alama-bronsard-gui,schatzman} and one could expect that some of them do not hold. This represents, in our opinion, a (momentary) obstruction to establishing the existence of 3D heteroclinic traveling waves connecting two 2D heteroclinics. Moreover, notice that the fact that $V$ is multi-well implies that 1D heteroclinic traveling waves do not exist in general, unless one imposes the existence of a local minimum at a level higher than 0 and some other properties are verified.
\subsection{Examples of potentials verifying the assumptions}\label{subs_examples}
In order to conclude this section, we exhibit a rather general and elementary method in order to produce examples of potentials for which the assumptions we make in this paper are satisfied. As we advanced before, the idea is to modify a given multi-well potential $V_0 : \R^k \to \R$ satisfying \ref{asu-sigma}, \ref{asu-infinity} and \ref{asu-zeros} such that the associated energy possesses two minimizing heteroclinics (up to translations) for two given wells $\sigma^-$, $\sigma^+$ in a finite set $\Sigma$. We also assume that the strict triangle's inequality \ref{asu-wells} is met for $V_0$ with respect to $(\sigma^-,\sigma^+)$. Furthermore, we assume that the generic Schatzman's spectral assumption (\cite{schatzman}) is satisfied for those heteroclinics, meaning that the constants defined in subsection \ref{subsection_constants} (with the obvious modifications) also make sense here. That is, one can think of any potential $V_0$ satisfying the assumptions of Schatzman's paper \cite{schatzman}. For the reader's convenience, we shall give here some explicit examples of such potentials which we found on the literature. 

The first of the examples we give was found by Antonopoulos and Smyrnelis, see \textit{Remark 3.6} in \cite{antonopoulos-smyrnelis}. Consider the case $k=2$. Let $V_{GL}$ be the Ginzburg-Landau potential
\begin{equation}\label{V_GL}
V_{GL}: u=(u_1,u_2) \in \R^2 \to \frac{(1-\lvert u \rvert^2)^2}{4} \in \R
\end{equation}
and consider the corresponding energy
\begin{equation}\label{E_GL}
E_{GL}(q):= \int_\R\left[ \frac{\lvert q'(t) \rvert^2}{2}+\frac{(1-\lvert q(t) \rvert^2)^2}{4} \right]dt, \hspace{2mm} q \in H^1_{\loc}(\R,\R^2).
\end{equation}
The idea is to perturb $V_{GL}$ in order to obtain a double-well potential with zero set $\{ (-1,0),(1,0) \}$ and symmetric with respect to the axis $\{ u_2=0 \}$. Such a potential will possess two heteroclinics provided that any curve with trace in $\{ u_2=0\}$ can be beaten by a competitor with a trace that is not contained in this set. Notice that for all $(u_1,0) \in \R \times \{0\}$ we have that $V_{GL}(u_1,0)=(1-u_1^2)^2/4$, which is the standard scalar double-well potential. As it is well known, the (unique) heteroclinic for such potential is given by the odd function $\Fq_{AC}:t \in \R \to \tanh(t/\sqrt{2}) \in \R$. Therefore, each curve $q=(q_1,q_2)$ in $H^1_{\loc}(\R,\R^2)$ with $q_2=0$ and $\lim_{t \to \pm \infty}q_1(t)=\pm 1$ verifies
\begin{equation}\label{Energy_q_AC}
E_{GL}(q) \geq E_{GL}(\Fq_{AC}) = \frac{2\sqrt{2}}{3}.
\end{equation}
For $T>0$, define 
\begin{equation}
q_T(t):=\begin{cases}
(-1,0) &\mbox{ if } t \leq -T-1,\\
((t+T)+(t+T+1)\Fq_{AC}(-T),0) &\mbox{ if } -T-1 \leq t \leq -T, \\
 -\Fq_{AC}(T)(\cos(\pi(t+T)/(2T)),\sin(\pi(t+T)/(2T))) &\mbox{ if } -T \leq t \leq T,\\
 ((t-T)-(t-T-1)\Fq_{AC}(T),0) &\mbox{ if } T \leq t \leq T+1,\\
 (1,0) &\mbox{ if } T+1 \leq t.
\end{cases}
\end{equation}
A modification of the computations made in \cite{antonopoulos-smyrnelis} shows that
\begin{equation}
\lim_{T \to +\infty}E_{GL}(q_T) = 0
\end{equation}
meaning that by \eqref{Energy_q_AC} there exists $\overline{T}>0$ such that $E_{GL}(q_{\overline{T}}) < E_{GL}(\Fq_{AC})$. Then, given $\eps:= (1-\lvert \Fq_{AC}(\overline{T}) \rvert^2)/4$, consider $\phi \in \CC^\infty(\R,[0,+\infty))$ such that
\begin{equation}
\phi(t)=\begin{cases}
0 &\mbox{ if } t \leq \lvert \Fq_{AC}(\overline{T}) \rvert^2+\eps\\
1 &\mbox{ if } t \geq 1-\eps
\end{cases}
\end{equation}
and define $\tilde{V}_0: u=(u_1,u_2) \in \R^2 \to V_{GL}(u)+u_2^2 \phi(\lvert u \rvert^2)$. Let $\tilde{E}_0$ be the corresponding energy. Notice that $\tilde{V}_0$ is a double-well potential verifying \ref{asu-sigma}, \ref{asu-infinity} and \ref{asu-zeros}. By definition, we have that if $q=(q_1,q_2) \in H^1([-R,R],\R^2)$ is such that $q_2=0$, then $\tilde{E}_0(q)=E_{GL}(q)$. Moreover, we also have $\tilde{E}_0(q_{\overline{T}}) =E_{GL}(q_{\overline{T}})<E_{GL}(\Fq_{AC})$. As a consequence, the minimizer $\Fq=(\Fq_1,\Fq_2)$ of $\tilde{E}_0$ in the class of curves in $H^1_{\loc}(\R,\R^k)$ which tend to $(\pm 1,0)$ at $\pm \infty$ satisfies $\Fq_2 \not =0$, which means that $\hat{\Fq}:=(\Fq_1,-\Fq_2)$ is also a minimizer due to the symmetry of $\tilde{V}_0$ and $\hat{\Fq}$ is not a translation of $\Fq$. Therefore, $\tilde{V}_0$ possesses two geometrically distinct globally minimizing heteroclinics. In order to find our example of potential, we need that such heteroclinics are non-degenerate in the sense asked by Schatzman in \cite{schatzman}, see our 3. in \ref{asu_unbalanced}. However, as shown in her \textit{Theorem 4.3} such assumption is generic, i. e., we can find $V_0$ arbitrarily close to $\tilde{V}_0$ which is still a double-well potential with wells $(-1,0)$, $(1,0)$ and with $\Fq$ and $\hat{\Fq}$ non-degenerate globally minimizing heteroclinics which satisfy the spectral assumptions.

Another example, this time in dimension $k=3$, is provided by Zuñiga and Sternberg \cite{zuniga-sternberg}. They consider the potential
\begin{equation}
\tilde{V}_0 : u=(u_1,u_2,u_3) \to u_1^2(1-u_1^2)^2+\left(u_2^2-\frac{1}{2}(1-u_1^2)^2 \right)^2+\left( u_3^2-\frac{1}{2}(1-u_1^2)^2 \right)^2 \in \R,
\end{equation}
which vanishes exactly on the points
\begin{align}
&(-1,0,0), (1,0,0),\\ &\left(0,\frac{1}{\sqrt{2}},\frac{1}{\sqrt{2}}\right), \left(0,-\frac{1}{\sqrt{2}},\frac{1}{\sqrt{2}}\right), \left(0,\frac{1}{\sqrt{2}},-\frac{1}{\sqrt{2}}\right), \left(0,-\frac{1}{\sqrt{2}},-\frac{1}{\sqrt{2}}\right).
\end{align}
By explicit computations, they show that the potential $\tilde{V}_0$ satisfies \ref{asu-sigma}, \ref{asu-infinity}, \ref{asu-zeros} and \ref{asu-wells} with $\sigma^\pm:=(\pm 1,0,0)$ and, moreover, that the infimum of the corresponding energy $\tilde{E_0}$ in $X(\sigma^-,\sigma^+)$ is not attained by a curve with trace contained in $\{ u_2=u_3=0\}$. Using the reflections $(0,u_2,0) \to (0,-u_2,0)$ and $(0,0,u_3) \to (0,0,-u_3)$, one deduces the multiplicity up to translations of the globally minimizing heteroclinics for $\tilde{E}_0$ in $X(\sigma^-,\sigma^+)$. As above, one can obtain $V_0$ arbitrarily close to $\tilde{V}_0$ such that the globally minimizing heteroclinics satisfy the spectral assumption.

Let us now return to the initial problem, and let $V_0$ be any potential satisfying the previous assumptions. In order to obtain a potential which satisfies the requirements of our setting, the idea is to make arbitrarily small smooth perturbations of $V_0$ around the trace of one of the heteroclinics, in such a way that its energy increases but a locally minimizing heteroclinic still exists (at least for small perturbations), which must necessarily have larger energy. One then chooses a perturbation which is not too large so that the upper bound on the difference of the energies is met. The idea is pictured in Figure \ref{FIGURE_chi}. We now show how to rigorously implement this idea. Let $\Fq^-$ and $\Fq^+$ in $X(\sigma^-,\sigma^+)$ be different up to translations and such that
\begin{equation}
E_0(\Fq^-)=E_0(\Fq^+)=\Fm_0:= \inf_{q \in X(\sigma^-,\sigma^+)}E_0(q),
\end{equation}
where, for $q \in H^1_\loc(\R,\R^k)$
\begin{equation}
E_0(q):= \int_\R \left[\frac{\lvert q'(t) \rvert^2}{2}+ V_0(q(t)) \right]dt.
\end{equation}
Recall that there exist $\rho_0^\pm$ such that
\begin{equation}
\forall q \in X(\sigma^-,\sigma^+), \hspace{2mm} \dist_{H^1(\R,\R^k)}(q,\CF^\pm) \leq \rho_0^\pm\Rightarrow \dist_{H^1(\R,\R^k)}(q,\CF^\pm)^2 \leq \beta^\pm(E_0(q)-\Fm_0)
\end{equation}
where
\begin{equation}
\CF^\pm:= \{ \Fq^\pm(\cdot+\tau): \tau \in \R \}.
\end{equation}
Let $t_0 \in \R$ be such that $\dist(\Fq^+(t_0),\Sigma)=\max_{t \in \R}\dist(\Fq^+(t),\Sigma)$ for some $\Fq^+ \in \CF^+$ and set $u_0:=\Fq^+(t_0)$. Let
\begin{equation}
r:=\min\{\rho_0^+/2,\dist(\Fq^+(t_0),\Sigma)/2\}>0.
\end{equation}
Define $\chi \in \CC^\infty_c(\R^k)$ be such that $0 \leq \chi \leq 1$, $\chi=1$ on $B(u_0,r)$ and $\supp(\chi)\subset B(u_0,2r)$. For each $\delta >0$, consider the potential $V_\delta:= V+\delta\chi \geq 0$. Define
\begin{equation}
E_\delta(q):= \int_\R \left[ \frac{\lvert q'(t) \rvert^2}{2}+ V_\delta(q(t))\right]dt
\end{equation}

Notice that, by the choice of $\chi$, $V_{\delta}$ vanishes exactly in $\Sigma$. Let now be $q \in X(\sigma^-,\sigma^+)$ such that $\dist_{H^1(\R,\R^k)}(q,\CF^+) \leq \rho_0^+/2$. We have that
\begin{equation}\label{ineq_delta}
\Fm_0+\frac{1}{\beta^+}\dist_{H^1(\R,\R^k)}(q,\CF^+)^2 \leq E_0(q) < E_0(q)+\delta\int_\R \chi(q) =E_\delta(q)
\end{equation}
and notice that for $q \in \CF^+$ we have that $E_\delta(q) =\Fm_0+\delta\Fi$ with
\begin{equation}
\Fi:=\int_\R \chi(\Fq^+(t))dt >0.
\end{equation}
A contradiction argument shows that
\begin{equation}
\Fm^+_\delta:= \inf\{ E_\delta(q): \dist_{H^1(\R,\R^k)}(q,\CF^+) \leq \rho_0^+/2\} > \Fm_0
\end{equation}
and we have $\Fm^+_{\delta} \leq E_{\delta}(\Fq^+)=\Fm_0+\delta \Fi$. Since the cut-off function is supported away from $\Sigma$, we can show by the usual concentration-compactness arguments that there exists $\Fq^+_{\delta}\in X(\sigma^-,\sigma^+)$ such that $\dist_{H^1(\R,\R^k)}(\Fq^+_{\delta},\CF^+) \leq \rho_0^+/2$ and $E_\delta(\Fq^+_{\delta})=\Fm^+_{\delta}$. If we show that $\dist_{H^1(\R,\R^k)}(\Fq^+_{\delta},\CF^+) < \rho_0^+/2$, then the constraints of the minimization problem are not saturated and $\Fq^+_{\delta}$ is an actual critical point. Notice that if $q \in X(\sigma^-,\sigma^+)$ is such that $\dist_{H^1(\R,\R^k)}(q,\CF^+)=\rho_0^+/2$, then by \eqref{ineq_delta} we obtain $E_0(q) \geq \Fm_0 +(\rho_0^+)^2/(4\beta^+)>\Fm_0$. Then, if we take $\delta < \delta_1$ with
\begin{equation}
\delta_1:= \frac{(\rho_0^+)^2}{4\beta^+\Fi}>0,
\end{equation}
it holds $E_{\delta}(q) > E_0(q) \geq \Fm_0+\delta\Fi \geq   \Fm^+_{\delta}$, so that $q$ cannot be a minimum. Therefore, for such $\delta$ items 1. and 2. in \ref{asu_unbalanced} are satisfied for $E_{\delta}$ with minimizing heteroclinics $\Fq^-$ and $\Fq^+_\delta$, with the obvious modifications on the notations. Regarding item 3., which is the spectral assumption of Schatzman \cite{schatzman}, it is a generic assumption, meaning that, arguing as it is done in her \textit{Theorem 4.3}, we find that $V_\delta$ can be modified with an arbitrary small perturbation away from the traces of $\Fq^+_\delta$ and $\Fq^-$ so that 3. holds. As a consequence, we can assume that \ref{asu_unbalanced} holds for all $\delta \in (0,\delta_1)$. Regarding \ref{asu_perturbation},  compute the constant $E_{\max}$ as in \eqref{Fb_max}, which by the choice of $r$ and $\chi$ does not depend on $\delta$, and set
\begin{equation}
\delta_2:= \frac{E_{\max}}{\Fi}>0,
\end{equation}
so that for all $\delta \in (0,\delta_2)$ we have $\Fm^+_\delta-\Fm_0<E_{\max}$. Define now $\CF_{\rho_0^-/2}^-$ as in \eqref{CFr}. The choice of $r$ and $\chi$ implies that $\CF_{\rho_0^-/2}^-$ does not depend on $\delta$, meaning that we can find $\delta_3$ such that for all $\delta \in (0,\delta_3)$ it holds
\begin{equation}
\{ q \in X(\sigma^-,\sigma^+): E_\delta(q) < \Fm_\delta^+ \} \subset \CF_{\rho^-_0/2}^-,
\end{equation}
meaning that \ref{asu_perturbation} holds for $E_{\delta}$ provided that $\delta \in (0,\delta_{\max})$ with $\delta_{\max}:=\min\{\delta_1,\delta_2,\delta_3\}>0$. As a consequence, we have found a family $\{ V_{\delta} \}_{\delta \in (0,\delta_{\max})}$ of potentials which are in the framework of Theorem \ref{THEOREM_main_TW} and we obtain a heteroclinic traveling wave with limits $\Fq^-$ and $\Fq^+$. Moreover, recall that if we put when $\delta=0$ we recover the \textit{classical} potentials considered in \cite{alama-bronsard-gui,fusco,monteil-santambrogio,schatzman,smyrnelis}, meaning that in this setting one can prove convergence results of the traveling waves toward stationary waves as $\delta \to 0^+$. Moreover, we see that is possible to decrease the value of $\delta$ even more so that the convergence assumption \ref{asu_convergence} holds and Theorems \ref{THEOREM_strong_bc} and \ref{THEOREM_carac_speed} also apply.
\begin{figure}
\centering
\includegraphics[scale=0.7]{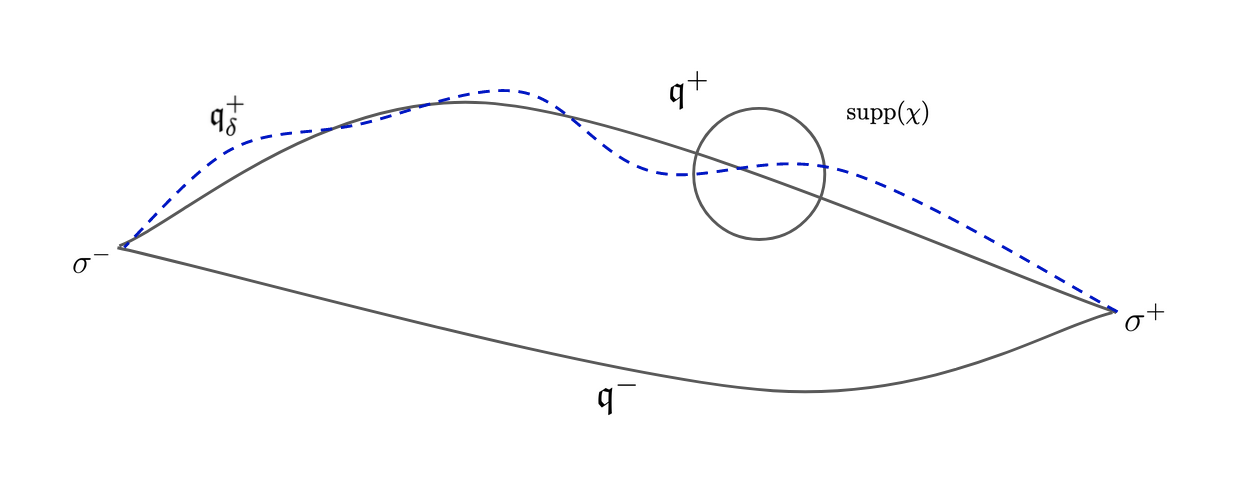}
\caption{Representation of the cut-off function $\chi$ used in order to produce the family of perturbed functionals $V_{\delta}$. We also draw the corresponding local minimizer $\Fq^+_\delta$ (discontinuous curve).}\label{FIGURE_chi}
\end{figure}

\section{Discussion on the previous literature and open problems}

\subsection{Related reaction-diffusion models and the question of stability}\label{subs_previous}

As we said in the introduction, the problem of existence of traveling waves for reaction-diffusion systems as well as their qualitative properties has been widely studied since the early works of Fisher \cite{fisher}, Kolmogorov, Petrovsky and Piskunov \cite{kpp} regarding the equation today known as the \textit{Fisher-KPP equation}. From the modeling perspective, while the aim of these authors was to describe of the dynamics of a given population, reaction-diffusion systems have also been proposed as models in other domains of the natural and social sciences. For example, applications in chemistry were given by Zeldovich \cite{zeldovich} and Kanel \cite{kanel} (see also Berestycki, Nicolaenko and Scheurer \cite{berestycki-nicolaenko-scheurer}) and the same Allen-Cahn model that we consider in this paper was proposed by Allen and Cahn \cite{allen-cahn}, following Cahn and Hilliard \cite{cahn-hilliard}, for describing phase transition problems in material physics.  It is also worth mentioning that for the most classical studies for traveling waves in reaction-diffusion problems the profile tends at infinity to two (possibly equal) constant stable states. However, other type of stable states (in particular, non constant) can be considered as conditions at infinity (as we do here). Moreover, the notion of traveling wave can be generalized in order to contain and describe similar structures. We refer to the papers by Berestycki and Hamel \cite{berestycki-hamel-07,berestycki-hamel-12} and the references therein.

Going back to the model Fisher-KPP equation, it can be written as follows:
\begin{equation}
\partial_t w- \partial_{x}^2w = f(w), \mbox{ in } [0,+\infty) \times \R
\end{equation}
where $f: \R \to \R$ is such that $f(0)=0$, $f(1)=0$, $f>0$ in $(0,1)$, $f<0$ in $(-\infty,0)$ and $f(u)<f'(0)u$ for $u>0$. Traveling waves for this equation are solutions of the type
\begin{equation}
w(t,x)=U(x-ct)
\end{equation}
with $c>0$ and $U: \R \to \R$ satisfies
\begin{equation}
\lim_{x \to -\infty}U(x)=0 \mbox{ and } \lim_{x \to +\infty}U(x)=1.
\end{equation}
From the point of view of modelling, traveling waves intend, for instance, to describe the invasion from a stable state to another one. An important feature of the Fisher-KPP equation is the existence of an important speed parameter, $c_{KPP}>0$, usually called the \textit{invasion speed} which can be explicitly computed as follows:
\begin{equation}
c_{KPP}:= 2\sqrt{f'(0)}.
\end{equation}
The previous problem, and related ones, is studied by means of the maximum principle and comparison results. Using these methods, one proves existence and uniqueness of a traveling wave with fixed speed $c>0$ if and only if $c \geq c_{KPP}$. This seems to be an important contrast with respect to the model that we consider here. Indeed, recall that our Theorem \ref{THEOREM_main_TW} states that the threshold speed $c^\star$ is, in particular, unique among the class of profiles which are minimizers in our variational setting. In fact, the same phenomenon is observed in earlier papers which also establish existence of traveling waves for reaction-diffusion systems by a variational procedure:  Muratov \cite{muratov}, Lucia, Muratov and Novaga \cite{lucia-muratov-novaga}, Alikakos and Katzourakis \cite{alikakos-katzourakis}, Chen, Chien and Huang \cite{chen-chien-huang}. Nevertheless, we point out that our results (the same as the ones we cite) do not exclude the possibility of other type of traveling wave solutions with speed different than $c^\star$. In particular, there could exist traveling waves with heteroclinic profiles which are obtained from a different variational setting than ours, or even from non-variational methods.

The analysis in \cite{fisher,kpp} was substantially extended in subsequent works. Fife and McLeod \cite{fife-mcleod77,fife-mcleod81} established stability properties for traveling waves in the Fisher-KPP equation. Generalizations bringing into consideration higher-dimensional equations (in space) were also made. For instance, Aronson and Weinberger \cite{aronson-weinberger} (see also Hamel and Nadirashvili \cite{hamel-nadirashvili} and the references therein) considered the case of $\R^N$ as space domain. We also mention the work of Berestycki, Larrouturou and Lions \cite{berestycki-larrouturou-lions} Berestycki and Nirenberg \cite{berestycki-nirenberg} for the case of a cylinder $\R \times \omega$, with $\omega\subset \R^{N-1}$ a bounded domain. For the case of periodic domains, see Berestycki and Hamel \cite{berestycki-hamel-02}, Berestycki, Hamel and Nadirashvili \cite{berestycki-hamel-nadirashvili-05}. The case of more general domains is adressed in Berestycki, Hamel and Nadirashvili \cite{berestycki-hamel-nadirashvili-10}. For the non-local problem see Berestycki et. al. \cite{berestycki-09}.

The family of non-linear functions $f$ which are admissible for the Fisher-KPP model does not contain non-linearities of Allen-Cahn type. Indeed, such non-linearities are written as $f=-V_{AC}'$ where $V_{AC}$ is a non-negative double-well potential, the prototypical case being
\begin{equation}
V_{AC}: u \in \R \to \frac{(1-u^2)^2}{4} \in \R,
\end{equation}
which does not satisfy the assumptions for required the equations of Fisher-KPP type, written above. For the scalar Allen-Cahn equation, we mention the result due to Matano, Nara and Taniguchi regarding the stability for traveling waves with $\R^N$ as space domain. In this case, the waves propagate according to one direction and connect the stable states $\pm 1$ at infinity. The case of traveling waves that connect one stable state with one unstable, non-constant periodic 1D solution was studied by Hamel and Roquejoffre \cite{hamel-roquejoffre}. The non-local case was adressed in Bates et. al. \cite{bates97}.

Many results are available for Allen-Cahn \textit{systems}, but mostly in one space dimension. In this case, the (negative) gradient flow structure implies that for initial data of finite energy which connects two different wells, the corresponding solution at long time the solution should look as a chain of glued 1D connecting orbits. More generally, if the initial condition connects at infinity two local minima at \textit{possibly different} levels (and a suitable weighted energy is finite), then traveling waves should also appear in the asymptotic pattern. Proofs of these facts, even in a more general framework, can be found in Risler \cite{risler,risler2021-1,risler2021-2}. Moreover, one can aim at obtaining quantitative results which describe more precisely the previous qualitative behavior and also introduce the problem of considering the system as a singular perturbation. That is, one considers a coefficient $\eps^{-2}$ multiplying the non-linear term and passes to the limit $\eps \to 0^+$. In this direction, it has been shown that the \textit{fronts} (that is, the regions in which the solution is far from the set of wells) of the solution of the gradient flow problem move at slow motion. The first rigorous proofs of this fact was given by Carr and Pego \cite{carr-pego89,carr-pego90}, Fusco and Hale \cite{fusco-hale}, for the Allen-Cahn equation. This analysis was later extended to multi-well systems by Bethuel, Orlandi and Smets \cite{bethuel-orlandi-smets} for multi-well systems. Bethuel and Smets \cite{bethuel-smets17,bethuel-smets19} obtained results regarding the motion law and the long-time interaction between stationary solutions in the multi-well \textit{scalar} case, allowing also for degenerate wells, but for the moment their work has not been extended to systems.

Regarding Allen-Cahn systems in higher dimensions, besides the classical articles regarding the stationary wave and this paper, we are only aware of the recent work of Chen, Chien and Huang \cite{chen-chien-huang}. In the latter, the authors consider the strip $\R \times (-l,l)$ as space domain and, for a class of symmetric triple-well potentials on the plane (similar to that from Bronsard, Gui and Schatzman \cite{bronsard-gui-schatzman}), show the existence of traveling wave solutions connecting at infinity a well and an approximation in $(-l,l)$ of a globally minimizing heteroclinic, which they assume to be unique. Their proof follows by a suitable application of the variational device of Muratov \cite{muratov}, which differs from that by Alikakos and Katzourakis \cite{alikakos-katzourakis} mainly on the fact that a different constrained minimization problem is considered.

A discussion concerning the mathematical methods used for addressing these problems is in order. As it is well known, while the maximum principle and the comparison theorems play a key role in the study of most \textit{scalar} reaction-diffusion equations (such as the Fisher-KPP equation), those tools are no longer available for systems except in some particular classes, for instance when dealing with the so-called \textit{monotone systems}, see Volpert, Volpert and Volpert \cite{volpertx3}. As a consequence, for more general classes of systems one needs other (more general) tools. Several approaches were developed, for example the use of Leray-Schauder degree \cite{volpertx3} or Conley theory as discussed in Smoller \cite{smoller}. We refer to the reader to the sources given in \cite{smoller,volpertx3}. While the gradient structure of some reaction-diffusion equations enables the application of variational methods (see the already cited references \cite{alikakos-fusco,bouhours-nadin,chen-chien-huang,heinze,lucia-muratov-novaga04,lucia-muratov-novaga,muratov,risler,risler2021-1,risler2021-2}), these methods have not been extensively used in this context. This is some kind of contrast with respect to the case of dispersive equations where, since the seminal work of Cazenave and Lions \cite{cazenave-lions}, a large amount of results regarding the existence and orbital stability of traveling waves and solitons has been produced. For instance, this has been done for Gross-Pitaevskii equations and systems, which are in some sense the dispersive counterpart of the parabolic problems of Allen-Cahn type. See Bethuel, Gravejat and Saut \cite{bethuel-gravejat-saut}, Bethuel et. al. \cite{bethuel-gravejat-saut-smets} for the orbital stability of traveling waves for the 1D Gross-Pitaevskii equation, Maris \cite{maris} for the Gross-Pitaevskii equation in $\R^N$, $N \geq 3$. The orbital stability of stationary waves (of heteroclinic type) for two coupled Gross-Pitaevskii equations was proven by Alama et. al. \cite{alama-bronsard}.

In order to conclude this section, we mention that the question of the local stability for the parabolic system of this paper is wide open. Even for the minimizing stationary wave obtained by Alama, Bronsard and Gui \cite{alama-bronsard-gui}, stability properties have not been studied to our knowledge. Of course, the question is also open for the traveling wave solutions that we obtain here.
\subsection{The heteroclinic stationary wave for 2D Allen-Cahn systems}
As pointed out before, the profile of the traveling wave solution that we obtain of this paper behaves at infinity as the stationary waves obtained in \cite{alama-bronsard-gui}. We briefly recall here how the existence of these solutions is shown, which we hope will make the links with our problem clearer.  Consider the elliptic system
\begin{equation}\label{stationary}
-\Delta \FU= \nabla_u V(\FU) \mbox{ in } \R^2,
\end{equation}
which corresponds to stationary solutions of \eqref{parabolic-allencahn}. The main result obtained in \cite{alama-bronsard-gui} states that if one assumes the existence of two distinct globally minimizing heteroclinics up to translations, $\Fq_-$ and $\Fq_+$, and adds a symmetry assumption on the potential, there exists a solution $\FU$ to \eqref{stationary} satisfying the conditions at infinity
\begin{equation}\label{stationary_limits}
\begin{cases}
\FU(x_1,x_2) \to \sigma^\pm & \mbox{ as } x_2 \to \pm \infty, \mbox{ uniformly in } x_1,\\
\FU(x_1,x_2)  \to \Fq_\pm(\cdot+\tau^\pm) &\mbox{ as } x_1 \to \pm \infty, \mbox{ uniformly in } x_2,
\end{cases}
\end{equation}
for some translation parameters $(\tau^-,\tau^+)\in \R^2$ (in the symmetric case $\tau^-=\tau^+=0$). The symmetry assumption was later removed by Schatzman in \cite{schatzman}, so that the parameters $(\tau^-,\tau^+)$ are part of the solution as well. While the existence proofs in \cite{alama-bronsard-gui} and \cite{schatzman} are (roughly speaking) addressed by addressing directly a functional associated to \eqref{stationary}, a more general approach was carried out successfully in a more recent paper by Monteil and Santambrogio \cite{monteil-santambrogio}, later by Smyrnelis in \cite{smyrnelis}. Their approach follows from the key observation (for more details see Alessio and Montecchiari \cite{alessio-montecchiari} and the references therein) that $\FU$ can be seen as a curve in the space
\begin{align}
Y(\Fq_-,\Fq_+):=\left\{ U: H^1_{\loc}(\R,(X(\sigma^-,\sigma^+),\right.&\left.\lVert \cdot \rVert_{L^2(\R,\R^k)})) :\right.\\ &\left.\lim_{x_1 \to \pm \infty}\inf_{\tau \in \R}\lVert U(x_1)-\Fq_\pm(\cdot+\tau) \rVert_{L^2(\R,\R^k)}=0 \right\}.
\end{align}
The previous leads to consider the functional
\begin{equation}\label{E2}
E_2: U \in Y(\Fq_-,\Fq_+) \to E_2(U):=\int_{\R} \left[\frac{\lVert U'(x_1) \rVert_{L^2(\R,\R^k)}^2}{2}+\CV(U(x_1))\right]dx_1 \in \R,
\end{equation}
where $\CV:X(\sigma^-,\sigma^+) \to \R$ is the normalized energy
\begin{equation}
\CV: q \in X(\sigma^-,\sigma^+) \to E(q)-\Fm \in \R,
\end{equation}
so that $\CV$ is a non-negative functional in $X(\sigma^-,\sigma^+)$ with zero set equal to $\CF$. Therefore, $\CV$ can be thought as a multi-well potential (modulo translations) in an infinite dimensional space.  Identifying $U$ with a function in $H^1_{\loc}(\R^2,\R^k)$ in the obvious way, we rewrite from \eqref{E2}
\begin{align}
E_2(U)= \int_\R \left[\int_\R\right.& \frac{\lvert \partial_{x_1} U(x_1,x_2) \rvert^2}{2} dx_2 \\ &\left.+\left[ \int_\R \left( \frac{\lvert\partial_{x_2} U(x_1,x_2)\rvert^2}{2}+V(U(x_1,x_2)) \right)dx_2-\Fm\right] \right]dx_1,
\end{align}
so, formally, critical points $\FU$ of $E_2$ in $Y(\Fq_-,\Fq_+)$ are solutions to \eqref{stationary} satisfying the conditions at infinity 
\begin{equation}
\lim_{x_1 \to \pm \infty} \inf_{\tau \in\R}\lVert \FU(x_1)-\Fq(\cdot+\tau) \rVert_{L^2(\R,\R^k)}=0
\end{equation}
In particular, the solution found in \cite{alama-bronsard-gui} is a global minimizer of $E_2$ in $Y(\Fq_-,\Fq_+)$. For the case of global minimizers, one shows that the stronger condition \eqref{stationary_limits} holds, which we suspect might not be true for other critical points. From this starting point, the authors in \cite{monteil-santambrogio} generalize the one-dimensional problem \eqref{mij} seeing it as a problem of finding geodesics in a metric space (more general than $\R^k$), which can be chosen to be equal to $Y(\Fq_-,\Fq_+)$. Subsequently, they are able to deduce the results of \cite{alama-bronsard-gui} as particular cases. These type of ideas inspired us for proving the results of this paper, which we do in the framework of abstract Hilbert spaces similar to that in \cite{smyrnelis}.

\subsection{Traveling waves for 1D parabolic systems of gradient type}
For the reader's convenience, we provide some more details on the result which was proven by Alikakos and Katzourakis \cite{alikakos-katzourakis} and how it links to our problem. Consider the 1D parabolic system
\begin{equation}\label{1D_Parabolic}
\partial_t \Fw - \partial^2_{x}\Fw=-\nabla_{u} W(\Fw) \mbox{ in } [0,+\infty) \times \R.
\end{equation}
Here $W: \R^k \to \R$ is an \textit{unbalanced} double-well potential.  The existence of traveling wave solutions for \eqref{1D_Parabolic} has been adressed by several authors, see for instance Risler \cite{risler}, Lucia, Muratov, Novaga \cite{lucia-muratov-novaga} as well as Alikakos and Katzourakis \cite{alikakos-katzourakis}. As we mentioned earlier, in this paper we look closely to the proof given in \cite{alikakos-katzourakis} (see also the book by Alikakos, Fusco and Smyrnelis \cite{alikakos-fusco-smyrnelis}). In order to be more precise, one looks for a pair $(c^\star,\Fu)$ such that the function
\begin{equation}
\Fw: (t,x) \in  [0,+\infty) \times \R \to \Fu(x-ct) \in \R^k
\end{equation}
solves \eqref{1D_Parabolic} and the profile $\Fu$ joins at infinity two local minimzers of $W$ at \textit{different} levels. The solution $\Fw$ is then a traveling wave solution. More precisely, the profile $\Fu$ solves the system
\begin{equation}\label{TW-FD}
-c^\star \Fu'-\Fu''=-\nabla_u W(\Fu) \mbox{ in } \R,
\end{equation}
and it satisfies at infinity
\begin{equation}
\lim_{t \to \pm \infty}\Fu(t)=a^\pm
\end{equation} 
where $a^- \in \R^k$ is a global minimum of $W$ with $W(a^-) < 0$ and $a^+$ is a local minimum of $W$ and $W(a^+)=0$. Moreover, in \cite{alikakos-katzourakis} it is also shown that the speed $c^\star$ is unique (a property to which our Proposition \ref{PROPOSITION-uniqueness} is analogous), while the profile does not need to be. 

The approach in \cite{alikakos-katzourakis,alikakos-fusco-smyrnelis} is variational and uses some previous ideas from Muratov \cite{muratov}. More precisely, they study the family of weighted functionals introduced by Fife and McLeod \cite{fife-mcleod77,fife-mcleod81}
\begin{equation}
E_c(q):= \int_\R \left( \frac{\lvert q'(t) \rvert^2 }{2}+W(q(t)) \right)e^{ct}dt,
\end{equation}
where $q$ belongs to a suitable subspace in $H^1_{\loc}(\R,\R^k)$ such that $\lim_{t \to \pm \infty}q(t) = a^\pm$. We formally check that critical points of $E_{c^\star}$ solve \eqref{TW-FD}. The strategy of the proof in \cite{alikakos-katzourakis} was introduced before in Alikakos and Fusco \cite{alikakos-fusco} and it can be summarized as follows: First, one solves a family of constrained minimization problems for $E_c$, where $c>0$ is at this point thought just as parameter. Once these problems have been solved one needs to find proper speed $c^\star$. Finally, one needs to ``remove" the constraints, that is, to show that for $c^\star$ one can find a constrained minimizer which does not saturate the constraints, meaning that it is an actual solution to \eqref{TW-FD}. These last two steps are accomplished by showing that constrained minimizers exhibit a suitable asymptotic behavior (more precisely, that they do not present a oscillatory behavior in arbitrarily large regions) . 

Therefore, the idea is to follow Monteil and Santambrogio \cite{monteil-santambrogio}, Smyrnelis \cite{smyrnelis} and adapt the result of Alikakos and Katzourakis for infinite-dimensional ODE systems, in which curves take values on an abstract Hilbert space. More precisely, we consider for $c>0$ the functional
\begin{equation}
E_{2,c}(U):= \int_\R \left(\int_\R \frac{\lvert \partial_{x_1} U(x_1,x_2) \rvert^2}{2}dx_2+(E(U(x_1,\cdot))-\Fm^+)\right)e^{cx_1}dx_1
\end{equation}
which we already introduced before. We then see $U$ (with the proper identifications) as a mapping $U: x_1 \in \R \to U(x_1,\cdot) \in L^2(\R,\R^k)$. For $v \in L^2(\R,\R^k)$, we set 
\begin{equation}\label{CW}
\CW(v):= \begin{cases}
E(U(x_1,\cdot))-\Fm^+ &\mbox{ if } v \in H^1(\R,\R^k),\\
+\infty & \mbox{ otherwise}.\\
\end{cases}
\end{equation}
Then, under assumption \ref{asu_unbalanced} we have that $\CW$ is an unbalanced double well potential in $L^2(\R,\R^k)$ and $E_{2,c}$ can be rewritten as
\begin{equation}\label{E2c}
E_{2,c}(U)= \int_{\R} \left[\frac{\lVert U'(x_1) \rVert_{L^2(\R,\R^k)}}{2}+ \CW(U(x_1))\right]e^{cx_1}dx_1,
\end{equation}
which is as $E_c$ but for curves taking values in $L^2(\R,\R^k)$ instead of $\R^k$. Therefore, the main issue here is to adapt the result of \cite{alikakos-katzourakis} for curves which take values in a possibly infinite-dimensional Hilbert space $\scrL$ (to be thought as $L^2(\R,\R^k)$) and possessing a proper subspace $\scrH$ (to be thought as $H^1(\R,\R^k)$) satisfying suitable properties with respect to $\CW$. A difficulty arises, since the minima of $\CW$ are non-isolated due to the invariance by translations of $E$. Nevertheless, this difficulty can be circumvented, and otherwise we could always restrict to potentials which are symmetric as Alama, Bronsard and Gui \cite{alama-bronsard-gui} and working in the resulting space of equivariant curves, in which invariance by translations disappears.  The major difficulty comes from the fact that in \cite{alikakos-katzourakis} the authors impose some non-degeneracy and radial monotonicity assumptions which prevent the constrained minimizers for exhibiting a degenerate, oscillatory behavior.  This assumption can be, in some sense, weakened in order to allow degenerate minima (we prove this in \cite{oliver-bonafoux-tw-bis}), but we cannot prove that $\CW$ fulfills them even for simple examples and we think it can be too restrictive. The reason is that the geometry of the level sets $\CW$ is difficult to understand, as it depends indirectly on $V$. For this reason, a new type of assumption, which in our case is \ref{asu_perturbation}, is needed to replace the one from \cite{alikakos-katzourakis}.

Finally, we point out that since our results are proven in an abstract setting, they also apply to 1D systems as \eqref{1D_Parabolic} by seeing this time $\scrL$ and $\scrH$ as $\R^k$. Moreover, in the companion paper \cite{oliver-bonafoux-tw-bis} we show that the abstract approach allows to modify the non-degeneracy assumptions on the minimizers used in \cite{alikakos-katzourakis} and consider classes of potentials with some kind of degenerate minima.

\subsection{Link with the singular limit problem}

The asymptotic behavior as $\eps \to 0$ for families of solutions $(u_\eps)_{\eps>0}$ of
\begin{equation}\label{parabolic_ac_eps}
\partial_t u_\eps - \Delta u_\eps = -\eps^{-2} \nabla_u V(u_\eps), \hspace{2mm} u_\eps: [0,T] \times \Omega \to \R^k
\end{equation}
has been extensively studied for bounded domains $\Omega \subset \R^N$ and $T>0$. Concerning the scalar case $k=1$, Ilmanen showed in \cite{ilmanen} that the equation above converges to \textit{Brakke's motion by mean curvature} as $\eps \to 0$. Regarding the vectorial case, while analogous results are established under several additional assumptions, little is proven regarding the general picture. We refer to Bronsard and Reitich \cite{bronsard-reitich} as well as the more recent Laux and Simon \cite{laux-simon} and the references therein. A state of the art regarding the elliptic problem can be found in Bethuel \cite{bethuel}. We will now briefly comment on how the results obtained in this paper can be linked to (and hopefully shed some light into) to asymptotic problem introduced above. For $\eps>0$, consider $(c^\star,\FU)$ the solution given by Theorem \ref{THEOREM_main_TW} and let
\begin{equation}\label{tw_eps}
w_\eps(t,x_1,x_2):=\FU_\eps(x_1-c_\eps t, x_2), \mbox{ for } (t,(x_1,x_2)) \in [0,+\infty) \times \R^2
\end{equation}
where for $(\tilde{x}_1,\tilde{x}_2) \in \R^2$
\begin{equation}\label{profile_eps}
\FU_\eps(\tilde{x}_1,\tilde{x}_2):= \FU(\eps^{-1}\tilde{x}_1,\eps^{-1}\tilde{x}_2)
\end{equation}
and
\begin{equation}\label{c_eps}
c_\eps:=\eps^{-1}c.
\end{equation}
Combining \eqref{tw_eps}, \eqref{profile_eps} and \eqref{c_eps} we have that for $(t,x_1,x_2) \in [0,+\infty) \times \R^2$
\begin{equation}\label{w_eps}
w_\eps(t,x_1,x_2)=\FU(\eps^{-1}x_1-\eps^{-2}c t,\eps^{-1}x_2)
\end{equation}
and recall that by Theorem \ref{THEOREM_main_TW} we have
\begin{equation}\label{profile_recall2}
-c\partial_{x_1}\FU-\Delta \FU=-\nabla_u V(\FU) \mbox{ in } \R^2
\end{equation}
which implies that for all $\eps>0$
\begin{equation}
\partial_t w_\eps-\Delta w_\eps= \eps^{-2} (-c\partial_{x_1}\FU-\Delta \FU)= -\eps^{-2}\nabla_u V(\FU)=-\eps^{-2}\nabla_u V(w_\eps) \mbox{ in }[0,+\infty) \times \R^2.
\end{equation}
Therefore, for any bounded domain $\Omega \subset \R^2$, $T>0$ and $\eps>0$, $w_\eps$ solves \eqref{parabolic_ac_eps}. That is, to sum up, $w_\eps$ is a traveling wave for the re-scaled potential $V_\eps:=\eps^{-2}V$, with profile $\FU_\eps$ as in \eqref{profile_eps} and with speed $c_\eps$ as in \eqref{c_eps}. Notice that $c_\eps \to +\infty$ as $\eps \to 0$. Regarding the asymptotics of $(w_\eps)_{\eps>0}$, let $\Omega=(-1,1)^2 \subset \R^2$ for simplification and consider a time interval $[0,T]$, $T>0$. Assume also that \ref{asu_convergence} holds, so that by Theorem \ref{THEOREM_strong_bc} we have
\begin{equation}\label{convergence_as}
\lim_{x_1 \to \pm \infty} \lVert \FU(x_1,\cdot)-\Fq^\pm(\cdot+\tau^\pm) \rVert_{H^1(\R,\R^k)}=0.
\end{equation}
In combination with \eqref{w_eps}, \eqref{convergence_as} implies that for all $x^+ \in (0,1)$
\begin{equation}
\lim_{\eps \to 0}\sup_{x_1 \in [x^+,1)}\lVert w_\eps(0,x_1,\cdot) - \Fq^+(\eps^{-1}(\cdot+\tau^+)) \rVert_{L^\infty((-1,1),\R^k)}=0
\end{equation}
and for all $x^- \in (-1,0)$
\begin{equation}
\lim_{\eps \to 0}\sup_{x_1 \in (-1,x^-]}\lVert w_\eps(0,x_1,\cdot) - \Fq^-(\eps^{-1}(\cdot+\tau^-)) \rVert_{L^\infty((-1,1),\R^k)}=0.
\end{equation}
Therefore, we find a phase transition on the line $\{(x_1,x_2) \in \R^2: x_1=0\}$, as it happens in the elliptic case with the rescaling of the stationary wave. On the contrary, for any $t>0$ we have
\begin{equation}
\lim_{\eps \to 0}\sup_{x_1 \in (-1,1)}\lVert w_\eps(t,x_1,\cdot) - \Fq^-(\eps^{-1}(\cdot+\tau^-)) \rVert_{L^\infty((-1,1),\R^k)}=0.
\end{equation}
That is, for positive time and small $\eps$, the rescaled solutions tend to look like the globally minimizing heteroclinic at the limit $x_1 \to\pm \infty$, in contrast with what is observed for $t=0$. In terms of the \textit{interfacial density}, the previous means that an initial condition with non-constant density gives a solution with constant density for $t>0$. This phenomenon is probably explained by some kind of parabolic regularization effects. To conclude this paragraph, notice that the considerations presented here are obtained by direct scaling computations. That is, they do not depend on the way $\FU$ is obtained. It is only required that $\FU$ solves \eqref{profile_recall2} with conditions \eqref{convergence_as}. In particular, the assumption \ref{asu_perturbation} is not relevant here and the same would apply to profiles obtained by different means under some other type of assumptions.

\section{The abstract setting}\label{sect_abstract}

\subsection{Main definitions and notations}
As we advanced in the introduction, instead of proving directly Theorems \ref{THEOREM_main_TW}, \ref{THEOREM_strong_bc} and \ref{THEOREM_carac_speed} , we will prove a set of more general results which will allow us to deduce the original ones as particular cases. In particular, we introduce an abstract setting similar to the ones considered in \cite{monteil-santambrogio} and specially \cite{smyrnelis}. The proof of the main abstract results, Theorems \ref{THEOREM-ABSTRACT}, \ref{THEOREM_abstract_bc} and \ref{THEOREM_abstract_speed} below, are thus the core of the paper. The passage between the abstract and the original setting is established in section \ref{section_proofs_final}, which in turn proves Theorems \ref{THEOREM_main_TW}, \ref{THEOREM_strong_bc} and \ref{THEOREM_carac_speed}. 

As we said before, the abstract results should be thought as an extension of the work by Alikakos and Katzourakis \cite{alikakos-katzourakis} to curves taking values in a more general Hilbert space and with minimum \textit{sets} instead of isolated minimum \textit{points}. In fact, we essentially perform an adaptation of their strategy of proof, which turns out to carry on to our setting. That is, our approach will consist on establishing existence of a pair $(c,\bU)$ in $(0,+\infty) \times X$ which fulfills
\begin{equation}\label{abstract_equation}
\bU''-D_{\scrL}\CE(\bU)=-c\bU' \mbox{ in } \R
\end{equation}
and satisfies the conditions at infinity
\begin{equation}\label{abstract_bc_weak-}
\exists T^-  \in \R: \forall t \leq T^-, \hspace{2mm} \bU(t) \in \scrF^-_{r_0^-/2},
\end{equation}
\begin{equation}\label{abstract_bc_weak+}
\exists T^+ \in \R: \forall t \geq T^+, \hspace{2mm} \bU(t) \in \scrF^+_{r_0^+/2}.
\end{equation}
Notice that this problem can also be thought as a heteroclinic connection problem on Hilbert spaces for a second order potential system with friction term. Such a problem could have its own interest besides the main application to the existence of traveling waves that we give here. Of course, analogous considerations can be also applied to the results in \cite{alikakos-katzourakis} as well as our companion paper \cite{oliver-bonafoux-tw-bis}.

The nature of the objects introduced above will be made precise along this paragraph. Let $\scrL$ be a Hilbert space with inner product $\langle \cdot,\cdot \rangle_{\scrL}$ and induced norm $\lVert \cdot \rVert_{\scrL}$. Let $\scrH \subset \scrL$ a Hilbert space with inner product $\langle \cdot,\cdot \rangle_{\scrH}$. In the original setting, $\scrL$ is $L^2(\R,\R^k)$ and $\scrH$ is $H^1(\R,\R^k)$, both endowed with their natural inner products. We will take $\CE: \scrL \to (-\infty,+\infty]$ an \textit{unbalanced} potential. In the setting of Theorem \ref{THEOREM_main_TW}, $\CE$ will \emph{essentially} coincide with $E-\Fm^+$ in $H^1(\R,\R^k)$ and with $+\infty$ elsewhere\footnote{this statement is not exact as the energy $E$ is not defined in $H^1(\R,\R^k)$, but on an affine space based on $H^1(\R,\R^k)$. However, we can trivially obtain a functional defined on $H^1(\R,\R^k)$ from $E$. See subsection \ref{subs_implication}.}. Here we just impose a set of abstract assumptions on $\CE$. Most of those assumptions follow bi combining ideas in \cite{alikakos-katzourakis} with ideas in Schatzman \cite{schatzman} and Smyrnelis \cite{smyrnelis}. We will begin by fixing two sets $\scrF^-$ and $\scrF^+$ in $\scrL$.  For $r>0$, we define
\begin{equation}\label{Fr}
\scrF^\pm_r:= \left\{ v \in \scrL: \inf_{\bv \in \scrF^\pm}\lVert v-\bv \rVert_{\scrL}\leq r \right\},
\end{equation}
and
\begin{equation}\label{FrH}
\scrF^\pm_{\scrH,r}:= \left\{ v \in \scrH: \inf_{\bv \in \scrF^\pm}\lVert v-\bv \rVert_{\scrH}\leq r \right\},
\end{equation}
that is, the closed balls in $\scrL$ and $\scrH$ respectively, with radius $r>0$ and center $\scrF^\pm$. The main assumption reads as follows:
\begin{hyp}\label{hyp_unbalanced}
The potential $\CE$ is weakly lower semicontinuous in $\scrL$. The sets $\scrF^-$ and $\scrF^+$ are closed in $\scrL$. There exists a constant $a<0$ such that
\begin{equation}
\forall v \in \scrL, \forall \bv^- \in \scrF^-, \hspace{2mm} \CE(v)  \geq \CE(\bv^-) = a
\end{equation}
and each $\bv^+ \in \scrF^+$ is a local minimizer satisfying $\CE(\bv^+) =0$. Moreover, there exist two positive constants $r_0^-$, $r_0^+$ such that $\scrF^+_{r_0^-} \cap \scrF^-_{r_0^+} = \emptyset$ (see \eqref{Fr}). There also exist $C^\pm>1$ such that
\begin{equation}\label{coercivityL}
\forall v \in \scrF^\pm_{r_0^\pm}, \hspace{2mm} (C^\pm)^{-1}\dist_{\scrL}(v,\scrF^\pm)^2 \leq \CE(v)-\min\{ \pm(-a),0\}.
\end{equation}
Moreover, for any $v \in \scrF_{r_0^\pm}^\pm$, there exists a unique $\bv^\pm(v) \in \scrF^\pm$ such that
\begin{equation}
\lVert v-\bv^\pm(v) \rVert_{\scrL}=\inf_{\bv^\pm \in \scrF^\pm}\lVert v-\bv^\pm \rVert_{\scrL}.
\end{equation}
Moreover, the projection maps
\begin{equation}\label{projection}
P^\pm: v \in \scrF^\pm_{r_0^\pm} \to \bv^\pm(v) \in \scrF^\pm
\end{equation}
are $C^2$ with respect to the $\scrL$-norm.
\end{hyp}
Hypothesis \ref{hyp_unbalanced} defines $\CE$ as an unbalanced double well potential with respect to $\scrF^-$ and $\scrF^+$ and gives local information of the minimizing sets. Compare with \ref{asu_unbalanced} and the remarks that follow. We have the following immediate consequence, which will be useful in the sequel:
\begin{lemma}\label{LEMMA_positivity}
Assume that \ref{hyp_unbalanced} holds. If we define for $r \in (0, r_0^\pm]$ we define
\begin{equation}\label{kappa_r}
\kappa^\pm_r:=\inf\{ \CE(v): \dist_\scrL(v,\scrF^\pm) \in [r,r_0^\pm] \}
\end{equation}
then we have $\kappa^\pm_r>\min\{ \pm(-a),0\}$. Moreover,
\begin{equation}\label{positivity}
\forall v \in \scrF^+_{r_0^+/2},\hspace{2mm} \CE(v) \geq 0.
\end{equation}
\end{lemma}
\begin{proof}
It follows directly from \eqref{coercivityL} in \ref{hyp_unbalanced}.
\end{proof}
 We now impose the following regarding the relationship between $\scrL$ and $\scrH$:
\begin{hyp}\label{hyp_smaller_space}
We have that $\scrH=\{ v \in \scrL: \CE(v)<+\infty\}$ and $\lVert \cdot \rVert_{\scrL} \leq \lVert \cdot \rVert_{\scrH}$. In particular, $\scrF^\pm \subset \scrH$. Moreover, $\CE$ is a $C^1$ functional on $(\scrH,\lVert \cdot \rVert_{\scrH})$ with differential $D\CE:v \in \scrH \to D\CE(v) \in \scrH'$, where $\scrH'$ is the (topological) dual of $\scrH$. Furthermore, there exists an even smaller space $\tilde{\scrH}$ with an inner product $\langle \cdot, \cdot \rangle_{\tilde{\scrH}}$ and associated norm $\lVert \cdot \rVert_{\tilde{\scrH}} \geq \lVert \cdot \rVert_{\scrH}$ such that we can find a continuous correspondence
\begin{equation}\label{DL}
D_{\scrL}\CE:v \in (\tilde{\scrH},\lVert \cdot \rVert_{\tilde{\scrH}}) \to D_{\scrL}\CE(v) \in (\scrL,\lVert \cdot \rVert_{\scrL})
\end{equation}
such that
\begin{equation}\label{abstract_ipp}
\forall v \in \tilde{\scrH}, \forall w \in \scrH,\hspace{2mm} D_{\scrL}\CE(v)(w)= D\CE(v)(w).
\end{equation}
\end{hyp}
Notice that in the context of Theorem \ref{THEOREM_main_TW} assumption \ref{hyp_smaller_space} is easily verified. The space $\tilde{\scrH}$ will be chosen $H^2(\R,\R^k)$ and \eqref{abstract_ipp} is no other that integration by parts. The notation $D_{\scrL}\CE$ is chosen to emphasize the formal $\scrL$-gradient flow structure of the corresponding abstract evolution equation.  We now continue by imposing a compactness assumption on $\scrF^\pm$:
\begin{hyp}\label{hyp_compactness}
$\scrL$-bounded subsets of $\scrF^\pm$ are compact with respect to $\scrH$-convergence.\footnote{hence, they are in particular compact with respect to $\scrL$-convergence}
\end{hyp}
Assumption \ref{hyp_compactness} readily implies the following:
\begin{lemma}\label{LEMMA_compactness}
Assume that \ref{hyp_unbalanced} and \ref{hyp_compactness} hold. Then, the sets $\scrF^\pm_{r_0^\pm/2}$ defined in \eqref{Fr} are closed in $\scrL$.
\end{lemma}

Assumption \ref{hyp_compactness} is necessary in order to establish the conditions at infinity. In the main context, it is a straightforward consequence of the compactness of the minimizing sequences. Subsequently, we impose the following:
\begin{hyp}\label{hyp_projections_inverse}
Assume that \ref{hyp_unbalanced} holds. For $\scrF^\pm$, one of the two following alternatives holds:
\begin{enumerate}
\item $\scrF^\pm$ is $\scrL$-bounded.
\item For all $(v,\bv^\pm) \in \scrF^\pm_{r_0^\pm} \times \scrF^\pm$, there exists an associated map $\hat{P}^\pm_{(v,\bv^\pm)}: \scrL \to \scrL$ such that
\begin{equation}\label{hatP}
P^\pm(\hat{P}^\pm_{(v,\bv^\pm)}(v))=\bv^\pm
\end{equation}
and
\begin{equation}\label{hatP_dist}
 \dist_{\scrL}(\hat{P}^\pm_{(v,\bv^\pm)}(v),\scrF^\pm)=\dist_{\scrL}(v,\scrF^\pm).
\end{equation}
Moreover, $\hat{P}^\pm_{(v,\bv^\pm)}: \scrL \to \scrL$ is differentiable and
\begin{equation}\label{D_inverse}
\forall (w_1,w_2) \in \scrL^2, \hspace{2mm} \lVert D(\hat{P}^\pm_{(v,\bv^\pm)})(w_1,w_2) \rVert_{\scrL}=\lVert w_2 \rVert_{\scrL}
\end{equation}
\begin{equation}\label{hatP_CE}
\CE(\hat{P}_{(v,\bv^\pm)}(v))=\CE(v).
\end{equation}
\end{enumerate}
\end{hyp}
Essentially, in 2. we impose that the projections $P^\pm$ from  \ref{hyp_unbalanced} are, in some sense, invertible. Again, this is straigtforward in the concrete setting, as the projections $P^\pm$ consist on performing a translation. We now impose an assumption for the sets $\scrF^\pm_{\scrH, r_0}$:
\begin{hyp}\label{hyp_projections_H}
For any $v \in \scrF_{\scrH,r_0^\pm}^\pm$, as defined in \eqref{FrH}, there exists a unique $\bv_\scrH^\pm(v) \in \scrF^\pm$ such that
\begin{equation}
\lVert v-\bv_\scrH^\pm(v) \rVert_{\scrL}=\inf_{\bv^\pm \in \scrF^\pm}\lVert v-\bv^\pm \rVert_{\scrL}.
\end{equation}
Moreover, the projection maps
\begin{equation}
P^\pm_\scrH: v \in \scrF^\pm_{\scrH,r_0^\pm} \to \bv_\scrH^\pm(v) \in \scrF^\pm
\end{equation}
are $C^1$ with respect to the $\scrH$-norm. Moreover, if $C^\pm>1$ is the constant from \ref{hyp_unbalanced}, we have
\begin{equation}\label{H_difference_projections}
\forall v \in \scrF^\pm_{\scrH,r_0^\pm}, \hspace{2mm} \lVert P^\pm(v)-P^\pm_{\scrH}(v) \rVert_{\scrH} \leq C^\pm \lVert v-P^\pm_{\scrH}(v) \rVert_{\scrH}.
\end{equation}
Furthermore, for each $r^\pm \in (0,r_0^\pm]$ there exist constants $\beta^\pm(r^\pm)>0$ such that in case that $v \in \scrF^\pm_{r_0^\pm}$ satisfies
\begin{equation}\label{H_local_bound}
\CE(v) \leq \min\{ \pm(-a),0\}+\beta^\pm(r^\pm),
\end{equation}
then $v \in \scrF_{\scrH,r}^\pm$. Finally, we have the following
\begin{equation}\label{H_local_est}
\forall v \in \scrF_{\scrH,r_0^\pm}^\pm, \hspace{2mm} (C^\pm)^{-2}\lVert v-P^\pm_\scrH(v) \rVert_{\scrH}^2 \leq \CE(v)- \min\{ \pm(-a),0\} \leq (C^\pm)^2 \lVert v-P^\pm_\scrH(v) \rVert_{\scrH}^2.
\end{equation}
\end{hyp}


Assumption \ref{hyp_projections_H} is made in order to ensure the suitable local properties around $\scrF^\pm$ in $\scrH$. In the main setting, those are known results which follow essentially from the spectral assumption by Schatzman \cite{schatzman}. Before introducing the last assumptions, we need some additional notation. For $U \in H^1_{\loc}(\R,\scrL)$ and $c>0$, we (formally) define
\begin{equation}\label{Ec}
\bE_c(U):= \int_\R \be_c(U)(t) dt:= \int_{\R} \left[ \frac{\lVert U'(t) \rVert_{\scrL}^2}{2}+\CE(U(t)) \right]e^{ct} dt.
\end{equation}
More generally, for $I \subset \R$ a non-empty interval and $U \in H^1_{\loc}(I,\scrL)$, put
\begin{equation}\label{Ec_interval}
\bE_c(U;I) := \int_I \be_c(U)(t) dt.
\end{equation}
Notice that the integrals defined in \eqref{Ec} and \eqref{Ec_interval} might not even make sense in general due to the fact that $\CE$ has a sign. Nevertheless, we can define the notion of \textit{local minimizer} of $\bE_c(\cdot;I)$ as follows:
\begin{definition}\label{def_local_minimizer}
Assume that \ref{hyp_unbalanced} and \ref{hyp_smaller_space} hold. Let $I \subset \R$ be a bounded, non-empty interval. Assume that  $U \in H^1_{\loc}(I,\scrL)$ is such that $E_c(U;I)$ is well-defined and finite. Assume also that there exists $C>0$ such that for any $\phi \in \CC^1_c(\mathrm{int}(I),(\scrH,\lVert \cdot \rVert_{\scrH}))$ such that
\begin{equation}
\max_{t \in I} \lVert \phi(t) \rVert_{\scrH}<C,
\end{equation}
the quantity $\bE_c(U+\phi;I)$ is well-defined and larger than $\bE_c(U;I)$. Then, we say that $U$ is a \emph{local minimizer} of $\bE_c(\cdot;I)$.
\end{definition}
We assume the following property for local minimizers:
\begin{hyp}\label{hyp_regularity}
Assume that \ref{hyp_unbalanced} and \ref{hyp_smaller_space} hold. There exists a map $\FP:\scrL \to \scrL$ such that
\begin{equation}\label{FP_1}
\forall v \in \scrL, \hspace{2mm} \CE(\FP(v)) \leq \CE(v) \mbox{ and } \CE(\FP(v))=\CE(v) \Leftrightarrow \FP(v)=v,
\end{equation}
\begin{equation}\label{FP_2}
\forall (v_1,v_2) \in \scrL^2, \hspace{2mm} \lVert \FP(v_1)-\FP(v_2) \rVert_{\scrL} \leq \lVert v_1-v_2 \rVert_{\scrL},
\end{equation}
and
\begin{equation}\label{FP_3}
\FP|_{\scrF^\pm}=\mathrm{Id}|_{\scrF^\pm}.
\end{equation}

 Let $I\subset \R$, possibly unbounded and non-empty. Let $c>0$. If $\bW \in H^1_{\loc}(I,\scrL)$ is a local minimizer of $\bE_c(\cdot;I)$ in the sense of Definition \ref{def_local_minimizer}, which, additionally, is such that for all $t \in I$, $\bW(t)=\FP(\bW(t))$,  then $\bW \in \CA(I)$ where for any open set $O\subset \R$, $\CA(O)$ is defined as
\begin{equation}\label{CA_def}
\CA(O):= \CC^2_\loc(O,\scrL) \cap \CC^1_\loc(O,(\scrH,\lVert \cdot \rVert_{\scrH})) \cap \CC^0_\loc(O,(\tilde{\scrH},\lVert \cdot \rVert_{\tilde{\scrH}}))
\end{equation}
and $\bW$ solves
\begin{equation}
\bW''-D_\scrL\CE(\bW)=-c\bW' \mbox{ in } I,
\end{equation}
where $D_\scrL\CE$ was introduced in \eqref{DL}. 
\end{hyp}
In the context of Theorem \ref{THEOREM_main_TW}, \ref{hyp_regularity} is a consequence of classical elliptic regularity results as well as properties on the energy functional. The purpose of the projection $\FP$ is technical, and in the main setting it will mean that constrained minimizers are bounded with respect to the $L^\infty$ norm.   Before stating the abstract result, we introduce the following constants (assuming that all the previous assumptions hold) which are obviously analogous with those introduced in subsection \ref{subsection_constants}:
\begin{equation}\label{intr_eta_0-}
\eta_0^-:= \min\left\{ \sqrt{e^{-1}\frac{r_0^-}{4}\sqrt{2(\kappa^-_{r_0^-/4}-a)}}, \frac{r_0^-}{4} \right\}>0,
\end{equation}
\begin{equation}\label{intr_r-}
\hat{r}^-:= \frac{r_0^-}{C^-+1}>0
\end{equation}
\begin{equation}\label{intr_eps_0-}
\CE_{\max}^-:= \frac{1}{(C^-)^2(C^-+1)}\min\left\{ \frac{(\eta_0^-)^2}{4},\kappa_{\eta_0^-}^--a,\beta^-(\hat{r}^-),\beta^-(\eta_0^-)\right\}>0,
\end{equation}
\begin{equation}\label{sfC}
\sfC^\pm:= \frac{1}{2}(C^\pm)^2((C^\pm)^2+(C^\pm+1)^2)>0,
\end{equation}
\begin{equation}\label{gamma-}
\gamma^-:= \frac{1}{C^-+\sfC^-}>0
\end{equation}
and
\begin{equation}\label{d0}
d_0:= \dist_{\scrL}(\scrF^+_{r_0^+/2},\scrF^-_{r_0^-/2})>0,
\end{equation}
where the constants $C^-, \beta^-(\hat{r}^-), \beta^-(\eta_0^-)$ are those from \ref{hyp_projections_H} and $\kappa_r^\pm$ for $r>0$ are defined in \eqref{kappa_r}. The fact that $d_0>0$ follows from Lemma \ref{LEMMA_compactness} and \ref{hyp_unbalanced}. We can finally state the following assumption:
\begin{hyp}\label{hyp_levelsets}
Assume that \ref{hyp_unbalanced} and \ref{hyp_smaller_space} hold.  Moreover,  assume that
\begin{equation}\label{intr_asu_a}
-a < \CE_{\max}^-
\end{equation}
and
\begin{equation}\label{negative_levelset}
\{ v \in \scrH: \CE(v) < 0 \} \subset \scrF^-_{r_0^-/2}.
\end{equation}
\end{hyp} 
Assumption \ref{hyp_levelsets} is essentially the abstract version of \ref{asu_perturbation}.  
\subsection{Statement of the abstract results}
Let us define the space
\begin{align}\label{X}
X:= \left\{ U \in H^1_{\loc}(\R,\scrL): \exists T \geq 1, \right. & 
\forall t \geq T, \hspace{2mm} \dist_{\scrL}(U(t),\scrF^+) \leq \frac{r_0^+}{2}, \\ &\left.
\forall t \leq -T, \hspace{2mm} \dist_{\scrL}(U(t),\scrF^-) \leq \frac{r_0^-}{2} \right\}.
\end{align}
The statement of the main abstract result is as follows:
\begin{theorem}[Main abstract result]\label{THEOREM-ABSTRACT}
Assume that \ref{hyp_compactness}, \ref{hyp_projections_inverse}, \ref{hyp_projections_H}, \ref{hyp_regularity} and \ref{hyp_levelsets} hold. Then, the following holds:
\begin{enumerate}
\item \textbf{Existence}. There exists $c^\star>0$ and $\bU \in \CA(\R) \cap X$, $\CA(\R)$ as in \eqref{CA_def} and $X$ as in \eqref{X}, such that $(c^\star,\bU)$ solves \eqref{abstract_equation} with conditions at infinity \eqref{abstract_bc_weak-}, \eqref{abstract_bc_weak+} and $\bU$ is a global minimizer of $\bE_c$ in $X$, that is, $\bE_c(\bU)=0$. Moreover, for all $t \in \R$, $\bU(t)=\FP(\bU(t))$, where $\FP$ is as in \ref{hyp_regularity}.
\item \textbf{Uniqueness of the speed}. The speed $c^\star$ is unique in the following sense: if $\overline{c^\star}>0$ is such that
\begin{equation}
\inf_{U \in X}\bE_{\overline{c^\star}}(U) =0
\end{equation}
and there exists $\overline{\bU} \in \CA(\R) \cap X$ such that $(\overline{c^\star},\overline{\bU})$ solves \eqref{abstract_equation} and $\bE_{\overline{c^\star}}(\overline{\bU})<+\infty$, then $\overline{c^\star}=c^\star$.
\item \textbf{Exponential convergence}. There exists a constant $M^+>0$ such that for all $t \in \R$ we have
\begin{equation}\label{exponential_abstract}
\lVert \bU(t)-\bv^+(\bU) \rVert_{\scrL} \leq M^+ e^{-ct},
\end{equation}
for some $\bv^+(\bU) \in \scrF^+$.
\end{enumerate}
\end{theorem}
\begin{remark}\label{REMARK-X}
Given the definition of $X$ in \eqref{X}, we have that for any $U \in X$ and $\tau \in \R$ it holds $U(\cdot+\tau) \in X$ and for any $c>0$ it holds $\bE_c(U(\cdot+\tau)) =e^{-c\tau}\bE_c(U)$. Such a thing implies
\begin{equation}
\forall c>0, \hspace{2mm} \inf_{U \in X}\bE_c(U) \in \{-\infty,0\}.
\end{equation}
Moreover, we see that in case $c>0$ is such that $\inf_{U \in X}\bE_c(U)=0$ one can find plenty of examples of minimizing sequences in $X$ which cannot ever reasonably produce a global minimizer. Indeed, consider any function $\tilde{U} \in X$ such that $E_c(\tilde{U})>0$ and then take the minimizing sequence $(\tilde{U}(\cdot+n))_{n \in \N}$.
\end{remark}
\begin{remark}
A more general statement can be given about the uniqueness of the speed, which in particular works for eventual non-minimizing solutions. See Proposition \ref{PROPOSITION-uniqueness}.
\end{remark}
Theorem \ref{THEOREM-ABSTRACT} will be shown to contain Theorem \ref{THEOREM_main_TW} in section \ref{section_proofs_final}. Notice that, as before, the conditions at infinity \eqref{abstract_bc_weak-} are rather weak (and not really of heteroclinic type), since we do not have convergence to an element of $\scrF^-$ as $t \to -\infty$. It is however clear that the conditions at infinity \eqref{abstract_bc_weak-}, \eqref{abstract_bc_weak+} are enough to ensure that the solution given by Theorem \ref{THEOREM-ABSTRACT} is not constant. In any case, we can impose an additional assumption in order to obtain stronger conditions at $- \infty$ on the solution:
\begin{hyp}\label{hyp_convergence}
Hypothesis \ref{hyp_levelsets} is fulfilled and, additionally:
\begin{equation}\label{a_convergence}
-a < \frac{(d_0\gamma^-)^2}{2},
\end{equation}
where $d_0$ and $\gamma^-$ were defined in \eqref{d0} and \eqref{gamma-} respectively.
\end{hyp}
Then we can show the following exponential convergence result
\begin{theorem}\label{THEOREM_abstract_bc}
Assume that  \ref{hyp_smaller_space}, \ref{hyp_compactness}, \ref{hyp_projections_inverse}, \ref{hyp_projections_H}, \ref{hyp_regularity}, \ref{hyp_levelsets} and \ref{hyp_convergence} hold. Then, if $(c^\star,\bU)$ is the solution given by Theorem \ref{THEOREM-ABSTRACT}, it holds that $\gamma^->c^\star$ ($\gamma^-$ as in \eqref{gamma-}) and there exists $\FM^->0$ such that for all $t \in \R$
\begin{equation}\label{abstract_bc_strong}
\lVert \bU(t)-\bv^-(\bU) \rVert_{\scrL} \leq M^- e^{(\gamma^--c^\star)t}
\end{equation}
for some $\bv^-(U) \in \scrF^-$.
\end{theorem}
Theorem \ref{THEOREM_abstract_bc} corresponds to Theorem \ref{THEOREM_strong_bc}. Finally, we will prove the following result:
\begin{theorem}\label{THEOREM_abstract_speed}
Assume that  \ref{hyp_smaller_space}, \ref{hyp_compactness}, \ref{hyp_projections_inverse}, \ref{hyp_projections_H}, \ref{hyp_regularity}, \ref{hyp_levelsets} and \ref{hyp_convergence} hold. Let $(c^\star,\bU)$ be the solution given by Theorem \ref{THEOREM-ABSTRACT}. Then, if $\tilde{\bU} \in \CA(\R) \cap X$ is such that
\begin{equation}
\bE_{c^\star}(\tilde{\bU}) =0
\end{equation}
then we have that $(c^\star,\tilde{\bU})$ solves \eqref{abstract_equation} and
\begin{equation}\label{abstract_speed_formula}
c^\star=\frac{-a}{\int_\R \lVert \tilde{\bU}'(t) \rVert_{\scrL}^2dt}.
\end{equation}
In particular, the quantity $\int_\R \lVert \tilde{\bU}'(t) \rVert_{\scrL}^2dt$ is finite. Moreover, we have
\begin{equation}\label{abstract_speed_variational}
c^\star=\sup\{c>0: \inf_{U \in X}\bE_c(U) = -\infty\}=\inf\{ c>0: \inf_{U \in X}\bE_c(U)=0\}
\end{equation}
as well as the bound
\begin{equation}\label{abstract_speed_bound}
c^\star \leq \frac{\sqrt{-2a}}{d_0}  <\min\left\{ \frac{\sqrt{2\CE_{\max}^-}}{d_0}, \gamma^- \right\}
\end{equation}
with $\CE_{\max}^-$ as in \eqref{intr_eps_0-}, $d_0$ as in \eqref{d0} and $\gamma^-$ as in \eqref{gamma-}. The second inequality follows from the bounds on $-a$ given by \ref{hyp_levelsets} and \ref{hyp_convergence}.
\end{theorem}
Theorem \ref{THEOREM_abstract_speed} corresponds to Theorem \ref{THEOREM_carac_speed}.

\section{Proof of the abstract results}\label{section_abs}
\subsection{Scheme of the proofs}\label{subs_scheme}
As pointed out several times, the structure of the proofs of our abstract results, Theorems \ref{THEOREM-ABSTRACT}, \ref{THEOREM_abstract_bc} and \ref{THEOREM_abstract_speed}, is analogous to that in Alikakos and Katzourakis \cite{alikakos-katzourakis}, which has its roots in Alikakos and Fusco \cite{alikakos-fusco}. In fact, most of their results also carry into the abstract setting with the suitable modifications. In fact, the structure of our proofs should be rather compared with \textit{subsection 2.6} in the book by Alikakos, Fusco and Smyrnelis \cite{alikakos-fusco-smyrnelis}, which slightly modifies and simplifies the argument in \cite{alikakos-katzourakis}.  We will also rely on some arguments provided in Smyrnelis \cite{smyrnelis}, when an analogous abstract approach is taken for the stationary problem. As usual, most of the intermediate results we prove hold under smaller subsets of assumptions (with respect to the set of all assumptions that we dropped in the previous section). Therefore, for the sake of clarity and generality, the necessary assumptions (and only these) that we use to prove a result are specified in its statement.

Despite the previous facts, and as pointed before, several important difficulties not present in \cite{alikakos-katzourakis} arise when one tries to tackle the same problem in the abstract setting we introduced in the previous section. One of those extra difficulties is due to the fact that, in our setting, we need deal with two different norms in the configuration space of the curves, $\scrL$ and $\scrH$ (to be thought as $L^2$ and $H^1$ respectively, for simplification) and that the potential $\CE$ is only lower semicontinuous with respect to $\scrL$-convergence. An additional difficulty comes from the fact that, due to the requirements of our original problem, we are not looking at curves that join two \textit{isolated minimum points}, but rather two \textit{isolated minimum sets}. This turns out to be an obstacle when one tries to adapt argument in \cite{alikakos-katzourakis}, even if one were to restrict to finite-dimensional configuration spaces. However, this difficulty is successfully dealt with using the precise knowledge about the \textit{projection mappings} (namely assumptions \ref{hyp_unbalanced}, \ref{hyp_projections_inverse} and \ref{hyp_projections_H}) is available. That is, one uses that, for a suitable neighborhood of the minimum sets, the projection onto the sets (both with the $\scrL$ and $\scrH$ norms) is well defined and enjoys some type of continuity and differentiability properties. This idea, in the Allen-Cahn systems setting, has to be be traced back to Schatzman \cite{schatzman}.

We will now briefly sketch the scheme of the proof of Theorem \ref{THEOREM-ABSTRACT}. Recall that, according to Remark \ref{REMARK-X}, direct minimization of $\bE_c$ in $X$ cannot yield solutions to the problem, the reason being the action of the group of translations. The spaces $X_T$, which were introduced in \cite{alikakos-katzourakis} (also in \cite{alikakos-fusco} for the equal-depth case) and will be precisely presented in \eqref{XT}, are defined in order to overcome this source of degeneracy, as they are no longer invariant by the action of the group of translations. See the design in Figure \ref{Figure_X_T_design}. As a consequence, compactness is restored and the corresponding minimization problem has a solution for all $c>0$ and $T \geq 1$. See Lemma \ref{LEMMA_mcT} later on. In general, minimizers in $X_T$ solve the profile equation on a (possibly proper) subset of $\R$ (see Lemma \ref{LEMMA_mcT_solution}), meaning that they are in general not solutions of \eqref{abstract_equation}. However, such constrained minimizers are in fact solutions of \eqref{abstract_equation} in the case they do not saturate the constraints. Therefore, the goal will be to show the existence of the speed $c^\star$ such that, for some $T \geq 1$, there exists a constrained minimizer in $X_T$ which does not saturate the constraints. For that purpose, a careful analysis of the behavior of the constrained minimizers is needed. Indeed, one needs a uniform bound (independent on $T$ and continuous on $c$) on the distance between the \textit{entry times}, i. e. the times in which the constrained minimizers enter $\scrF^\pm_{r_0^\pm/2}$. In the balanced case this follows from the fact that the energy density is bounded below by a positive constant outside  $\scrF^-_{r_0^-/2}\cup\scrF^+_{r_0^+/2}$ (see for instance Smyrnelis \cite{smyrnelis}). However this is no longer true for our unbalanced problem, which makes it more involved: If one does not have the positivity of the energy density, the constraint solutions can oscillate between the regions $\scrF^\pm_{r_0^\pm/2}$ (producing energy compensations) in larger and larger intervals as $T \to \infty$, so that no $T$-independent bound can be found. This is the main new difficulty with respect to the balanced setting, as one needs new ideas in order to obtain a uniform bound on the distance between the entry times. Our assumption \ref{hyp_levelsets} provides this control because the energy density of the constrained minimizers is bounded below by a positive constant in the interval given by the two entry times mentioned before, meaning that we can argue as in the balanced case. The precise result is Corollary \ref{COROLLARY_constrained_behavior}. This is the main step in which our proof differs with that in \cite{alikakos-katzourakis}.

The natural question is what happens if we remove \ref{hyp_levelsets}. A natural approach is to replace \ref{hyp_levelsets} by an assumption more closely related to the one used in \cite{alikakos-katzourakis} and \cite{alikakos-fusco-smyrnelis}. This would lead to introduce a convexity assumption on the level sets of $\CE$, as well as some sort of strict monotonicity on well-chosen segments. While this assumption can be worked out in the abstract setting and it is applicable for the finite-dimensional situation considered in \cite{alikakos-katzourakis} (as we show in the companion paper \cite{oliver-bonafoux-tw-bis}), we believe it to be too restrictive to be applied to our original problem.

In any case, after the uniform bound on the entry times of constrained minimizers is obtained, one needs to find the speed $c^\star$ as, until this point, the speed $c>0$ has been only considered as a parameter of the problem without any special role. Our arguments adapts without major difficulty from \cite{alikakos-fusco-smyrnelis} and it goes as follows: One introduces a set which classifies the speeds according to the value of the infimum of the corresponding energy on $X$ (which, due to the weight and the invariance by translations, is either $-\infty$ or $0$). Such a set is $\CC$, defined in \eqref{set_CC}. Subsequently, one shows (Lemma \ref{LEMMA_set_CC}) that $\CC$ is open, bounded, non-empty and that its positive limit points give rise to entire minimizing solutions of the equations (since for those points one can find corresponding constrained minimizers which do not saturate the constraints). The speed $c^\star$ is then defined as the supremum of $\CC$, which is in fact the unique positive limit point of the set, as shown in Corollary \ref{COROLLARY_CC}. At this point, the process of the proof of Theorem \ref{THEOREM-ABSTRACT} is completed. Later on, we show that the asymptotic behavior of the constrained solutions can be improved under an additional assumption, namely an upper bound on the speed. This is Proposition \ref{PROPOSITION_conv_sol_-}. Theorems \ref{THEOREM_abstract_bc} and \ref{THEOREM_abstract_speed} can be then proven.
\begin{figure}[h!]
\centering
\includegraphics[scale=0.4]{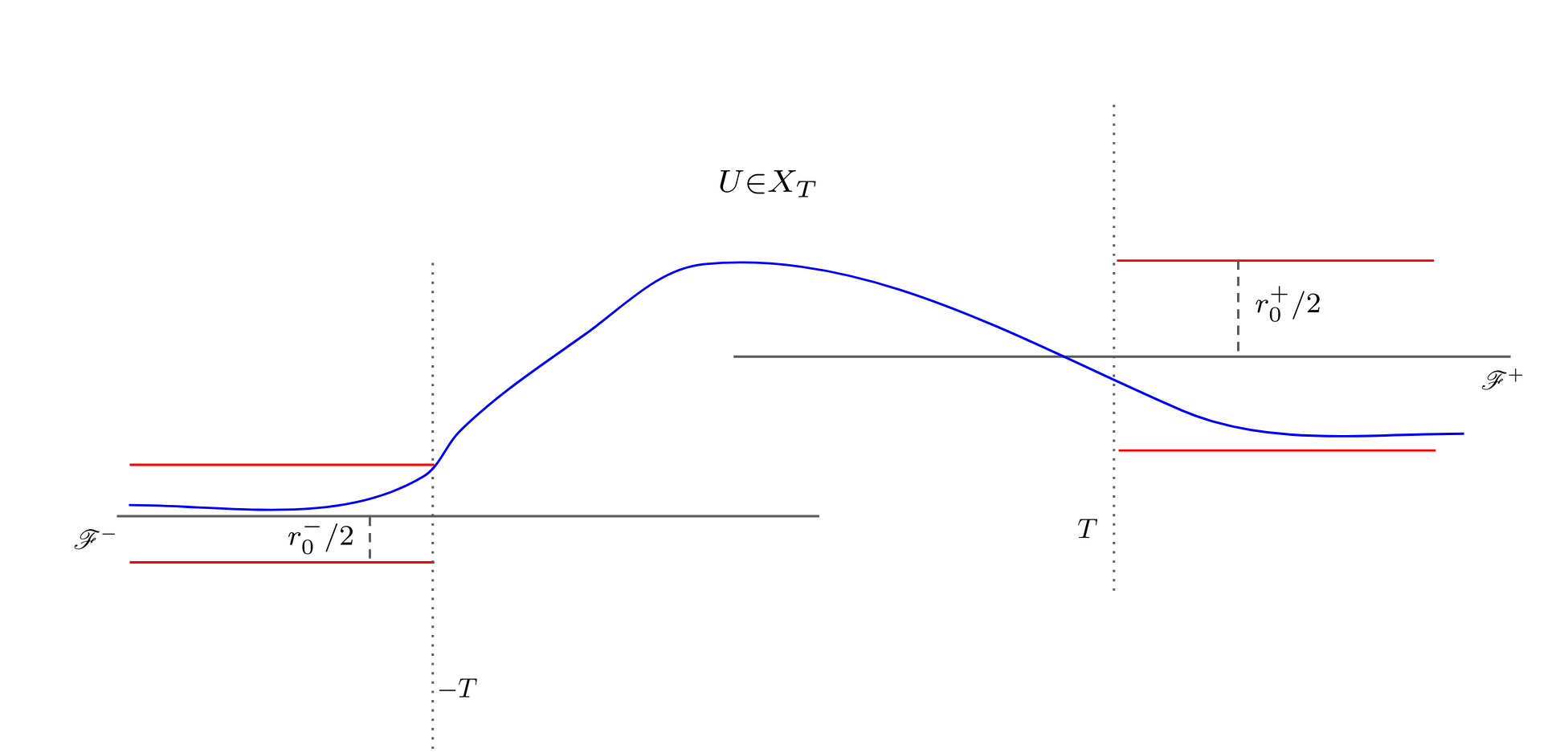}
\caption{One-dimensional representation of $X_T$. The blue line represents a function $U$ belonging to $X_T$. The red lines contain the points which are at $\scrL$-distance smaller than $r_0^\pm/2$ from $\scrF^\pm$.}\label{Figure_X_T_design}
\end{figure}
\subsection{Preliminaries}
Let $r_0^-$ and $r_0^+$ be the constants introduced in section \ref{sect_abstract} and $\scrF_{r_0^\pm/2}^\pm$ be the corresponding closed balls as in \eqref{Fr}. Assume that \ref{hyp_unbalanced} holds. For $T \geq 1$, we define the sets
\begin{equation}
X_{T}^-:= \left\{ U\in H^1_{\loc}(\R,\scrL): \forall t \leq -T, \hspace{1mm} U(t) \in \scrF^-_{r_0^-/2} \right\},
\end{equation}
\begin{equation}
X_{T}^+:= \left\{ U \in H^1_{\loc}(\R,\scrL): \forall t \geq T, \hspace{1mm}U(t) \in \scrF^+_{r_0^+/2} \right\}.
\end{equation}
Subsequently,  we set
\begin{equation}\label{XT}
X_{T}:=X_{T}^- \cap X_{T}^+.
\end{equation}
Recall the space $X$ introduced in \eqref{X}. Notice that
\begin{equation}
X=\bigcup_{T \geq 1}X_T.
\end{equation}
We have the following preliminary properties on the spaces $X_T$:
\begin{lemma}\label{LEMMA_well_defined}
Assume that \ref{hyp_unbalanced} holds.  Let $c>0$ and $T \geq 1$. For any $U \in X_T$, we have that
\begin{equation}\label{positivity_U}
\forall t \geq T, \hspace{2mm} \CE(U(t)) \geq 0.
\end{equation}
Moreover, the quantity $\bE_c(U)$ as introduced in \eqref{Ec} is well defined in $(-\infty,+\infty]$.
\end{lemma}
\begin{proof}
Let $U \in X_T$. Notice that for $t \geq T$, we have that $U(t) \in \scrF^+_{r_0/2}$. Therefore,  \eqref{positivity_U} follows directly from \eqref{positivity} in Lemma \ref{LEMMA_positivity}.

Let now $\CE^+(U) \geq 0$ and $\CE^-(U) \geq 0$ be, respectively, the non-negative and the non-positive part of $\CE(U)$, so that $\CE(U)=\CE^+(U)-\CE^{-}(U)$. We have that $\CE^-(U)$ is null on $[T,+\infty)$. That is
\begin{equation}
\int_{-\infty}^{+\infty} \CE^-(U(t)) e^{ct}dt = \int_{-\infty}^T \CE^-(U(t))e^{ct}dt \leq -\frac{a}{c}e^{cT}<+\infty,
\end{equation}
where $a$ is the minimum value from \ref{hyp_unbalanced}. Therefore, the negative part of the energy density $\be_c(U)$ (see \eqref{Ec}) belongs to $L^1(\R)$, which establishes the result.
\end{proof}
Lemma \ref{LEMMA_well_defined} shows that for any $T \geq 1$ and $c>0$, $\bE_c$ is well defined as an extended functional on $X_T$, at least if sufficient hypothesis are made. Moreover, it gives the following useful inequalities:
\begin{lemma}\label{LEMMA_integrability}
Assume that \ref{hyp_unbalanced} holds. Let $c>0$ and $T \geq 1$. For any $U \in X_T$, we have that
\begin{equation}\label{integrability_derivative}
\int_{\R} \frac{\lVert U'(t) \rVert_{\scrL}^2}{2}e^{ct}dt \leq \bE_c(U) -\frac{a}{c}e^{cT}
\end{equation}
and
\begin{equation}\label{integrability_energy}
\int_{\R} \lvert \CE(U(t)) \rvert e^{ct} dt \leq \bE_c(U)-\frac{a}{c}e^{cT}.
\end{equation}
Finally, we have that for all $t \in \R$,
\begin{equation}\label{integrability_TV}
\int_{t}^{+\infty} \lVert U'(s) \rVert_{\scrL}ds  \leq \left(\left(\bE_c(U) -\frac{a}{c}e^{cT}\right)\frac{e^{-ct}}{c}\right)^{\frac{1}{2}}
\end{equation}
\end{lemma}
\begin{proof}
Using \eqref{positivity_U} in Lemma \ref{LEMMA_well_defined}, we get that
\begin{equation}
\int_{\R} \frac{\lVert U'(t) \rVert_{\scrL}^2}{2}e^{ct}dt \leq \bE_c(U)-\int_{-\infty}^T \CE(U(t))e^{ct},
\end{equation}
which, by \ref{hyp_unbalanced}, implies that \eqref{integrability_derivative} holds. Inequality \eqref{integrability_energy} is obtained in the same fashion. Finally, we have that \eqref{integrability_TV} follows by combining \eqref{integrability_derivative} with Cauchy-Schwartz inequality.
\end{proof}
The previous results allow to prove the following convergence properties at $+\infty$ for finite energy functions in $X_T$:
\begin{lemma}\label{LEMMA_limit+}
Assume that \ref{hyp_unbalanced} and \ref{hyp_projections_H} hold. Let $c>0$ and $T \geq 1$. Take $U \in X_T$ such that $\bE_c(U) <+\infty$. Then, we have that there exists a subsequence $(t_n)_{n \in \N}$ in $\R$ such that $t_n \to +\infty$ as $n \to \infty$ and
\begin{equation}\label{limit+_energy0_weak}
\lim_{n \to \infty} \CE(U(t_n))e^{c t_n}=0.
\end{equation}
Moreover, there exists $\bv^+(U) \in \scrF^+$ such that for all $t \in \R$ it holds
\begin{equation}\label{limit+_function0_weak}
\lVert U(t) - \bv^+(U) \rVert_{\scrL}^2 \leq \left(\frac{\bE_c(U) -\frac{a}{c}e^{cT}}{c}\right)e^{-ct}.
\end{equation}
That is, $U$ tends to $\bv^+(U)$ at $+\infty$ with an exponential rate of convergence and with respect to the $\scrL$-norm.
\end{lemma}
\begin{proof}
We have by \eqref{integrability_energy} in Lemma \ref{LEMMA_integrability} that $t \in \R \to \CE(U(t))e^{ct} \in \R$ belongs to $L^1(\R)$ because $\bE_c(U)<+\infty$. Therefore, combining with \eqref{positivity_U} in Lemma \ref{LEMMA_positivity}, we obtain \eqref{limit+_energy0_weak}.

Subsequently, notice that \eqref{integrability_TV} in Lemma \ref{LEMMA_integrability} and the fact that $\bE_c(U)<+\infty$ gives the existence of $\bv^+(U) \in \scrF^+$ such that  $\lim_{t \to +\infty}\lVert U(t)-\bv^+(U) \rVert_{\scrL}=0$. Therefore, fix $t \in \R$ and notice that for any $\tilde{t}>t$ we have
\begin{equation}
\lVert U(\tilde{t})-U(t) \rVert_{\scrL} \leq \int_{t}^{\tilde{t}}\lVert U'(s) \rVert_{\scrL}ds \leq \int_{t}^{+\infty}\lVert U'(s) \rVert_{\scrL},
\end{equation}
which by \eqref{integrability_TV} in Lemma \ref{LEMMA_integrability} means that
\begin{equation}
\lVert U(\tilde{t})-U(t) \rVert_{\scrL}^2 \leq \left(\frac{\bE_c(U) -\frac{a}{c}e^{cT}}{c}\right)e^{-ct}.
\end{equation}
Therefore, passing to the limit $\tilde{t} \to +\infty$ we obtain \eqref{limit+_function0_weak}, also due to the fact that $U$ is continuous with respect to the $\scrL$-norm.

\end{proof}
\begin{remark}
Notice that \eqref{limit+_function0_weak} in Lemma \ref{LEMMA_limit+} does not imply convergence of $\CE(U)$ towards 0 at $+\infty$, due to the fact that $\CE$ is not continuous with respect to the $\scrL$ norm.
\end{remark}
\begin{remark}
Regarding the behavior at $-\infty$, notice that we can only say that if $U \in X_T$ is such that $\bE_c(U)<+\infty$, then $\CE(U)$ does not go to $+\infty$ faster than $e^{ct}$ at the limit $t \to -\infty$. That is, almost nothing can be said for generic finite energy solutions regarding their behavior at $-\infty$.
\end{remark}
\subsection{The infima of $\bE_c$ in $X_T$ are well defined}
Once we have defined the spaces $X_T$, we show that the corresponding infimum of $\bE_c$ is well defined as a real number for all $c>0$. Set
\begin{equation}\label{mcT}
\bm_{c,T}:= \inf_{U \in X_T}\bE_c(U) \in [-\infty,+\infty).
\end{equation}
We have the following:
\begin{lemma}\label{LEMMA_PSI}
Assume that \ref{hyp_unbalanced} and \ref{hyp_smaller_space} hold. Fix $\hat{\bv}^\pm \in \scrF^\pm$. Let $c>0$ and $T \geq 1$. For all $T \geq 1$ the function
\begin{equation}\label{Psi}
\Psi(t):= \begin{cases}
\hat{\bv}^- &\mbox{ if } t \leq -1,\\
\frac{1-t}{2}\hat{\bv}^-+\frac{t+1}{2}\hat{\bv}^+ &\mbox{ if } -1 \leq t \leq 1,\\
\hat{\bv}^+ &\mbox{ if } t \geq 1,
\end{cases}
\end{equation}
belongs to $X_T$. Moreover, for all $c>0$
\begin{equation}\label{Ec_PSI}
\bE_c(\Psi)<+\infty.
\end{equation}
Furthermore, we have
\begin{equation}\label{mcT_finite}
-\infty<\bm_{c,T}<+\infty.
\end{equation}
\end{lemma}
\begin{proof}
It is clear that $\Psi \in X_T $. We now show that \eqref{Ec_PSI} holds. Notice first that
\begin{equation}
\int_{-\infty}^1 \be_c(\Psi) = \int_{-\infty}^1  a e^{ct}dt=\frac{a}{c}e^c,
\end{equation}
where $a$ is the minimum value from \ref{hyp_unbalanced}. Subsequently, we have
\begin{equation}
\int_{1}^{+\infty}\be_c(\Psi)=0
\end{equation}
and
\begin{align}
\int_{-1}^1 \be_c(\Psi)&= \int_{-1}^1 \left[ \frac{\left\lVert \hat{\bv}^+-\hat{\bv}^- \right\rVert_\scrL^2}{8}+\CE\left( \frac{1-t}{2}\hat{\bv}^-+\frac{t+1}{2}\hat{\bv}^+ \right) \right]e^{ct}dt \\
& \leq \left[\frac{\left\lVert \hat{\bv}^+-\hat{\bv}^- \right\rVert_\scrL^2}{4}+2\max_{t \in [-1,1]}\CE\left( \frac{1-t}{2}\hat{\bv}^-+\frac{t+1}{2}\hat{\bv}^+ \right)\right]\frac{e^c-e^{-c}}{c}
\end{align}
and we have
\begin{equation}
\max_{t \in [-1,1]}\CE\left( \frac{1-t}{2}\hat{\bv}^-+\frac{t+1}{2}\hat{\bv}^+ \right)<+\infty,
\end{equation}
by \ref{hyp_smaller_space}. Therefore, we have obtained $E_c(\Psi) <+\infty$, which readily implies that $\bm_{c,T} <+\infty$. In order to establish \eqref{mcT_finite}, we still need to show that $\bm_{c,T}>-\infty$. For that purpose, let $U \in X_T$. By \eqref{positivity_U} in Lemma \ref{LEMMA_well_defined}, we have
\begin{equation}
\int_T^{+\infty}\be_c(U) \geq 0.
\end{equation}
We also have
\begin{equation}
\int_{-\infty}^T \be_c(U) \geq \int_{-\infty}^{T} a \be^{ct}dt = \frac{a}{c} e^{cT}.
\end{equation}
That is
\begin{equation}
\forall U \in X_T, \hspace{2mm} \bE_c(U) \geq \frac{a}{c}e^{cT}>-\infty,
\end{equation}
which means that $\bm_{c,T}>-\infty$. 
\end{proof}
The next goal will be to show that, under the proper assumptions, we have that for any $c>0$ and $T\geq 1$, the infimum values defined in \eqref{mcT} are attained. Such a fact is not hard to prove since the constraints that define the spaces $X_T$ allow to restore compactness. It relies on some properties that will be proven in the next subsection.
\subsection{General continuity and semi-continuity results}
We now provide some results which address continuity and semicontinuity properties of the energies $\bE_c$ in the spaces $X_T$. Such properties will allow us to show that the infimum values defined in \eqref{mcT} are attained under the proper assumptions. They will be also be useful in a more advanced stage of the proof, when the constrains will be removed. For now, we essentially adapt some results from \cite{alikakos-katzourakis} to our setting. 

Our first result is essentially \textit{Lemma 26} in \cite{alikakos-katzourakis}:

\begin{lemma}\label{LEMMA_ATU}
Assume that \ref{hyp_unbalanced} holds. Fix $T \geq 1$ and $U \in X_T$. Consider the set
\begin{equation}\label{ATU}
A_{T,U}:=\{ c>0: \bE_c(U) <+\infty\}.
\end{equation}
Then,  if $c \in A_{T,U}$, then $(0,c] \subset A_{T,U}$. Moreover, the correspondence
\begin{equation}
c \in A_{T,U} \to \bE_c(U) \in \R
\end{equation}
 is continuous.
\end{lemma}
\begin{proof}
Let $c \in A_{T,U}$. On the one hand, inequality \eqref{positivity_U} in Lemma \ref{LEMMA_well_defined} gives
\begin{equation}
0 \leq \int_{T}^{+\infty} \be_c(U(t))dt \leq \bE_c(U)-\int_{-\infty}^{T} \CE^-(U(t))e^{ct}dt \leq \bE_c(U)-ae^{cT}<+\infty,
\end{equation}
which implies that a. e. in $(T,+\infty)$
\begin{equation}\label{ATU_L1}
0 \leq \be_c(U(\cdot)) \in L^1((T,+\infty)).
\end{equation}
Therefore, if $c' \leq c$ we have that a. e. in $(T,+\infty)$
\begin{equation}\label{ATU_c'1}
0 \leq \be_{c'}(U(\cdot)) \leq \be_{c}(U(\cdot)) \in L^1(T,+\infty).
\end{equation}
On the other hand,
\begin{equation}
\int_{-\infty}^{T} \lvert \be_{c}(U(t)) \rvert dt \leq \bE_c(U)+2ae^{cT}<+\infty
\end{equation}
because $\be_c(U(\cdot))$ is non-negative a. e. in $[T,+\infty)$. The previous inequality shows
\begin{equation}
\lvert \be_c(U(\cdot)) \rvert \in L^1((-\infty,T)),
\end{equation}
and we have that a. e. in $(-\infty,T)$
\begin{equation}\label{ATU_c'2}
\lvert \be_{c'}(U(\cdot)) \rvert \leq \lvert \be_c(U(\cdot)) \rvert \in L^1((-\infty,T)).
\end{equation}
Combining \eqref{ATU_c'1} and \eqref{ATU_c'2}, we obtain that $\lvert e_c(U(\cdot)) \rvert \in L^1(\R)$, meaning that $c' \in A_{T,U}$. Hence, we have $(0,c] \subset A_{T,U}$ as we wanted to show.

Consider now a sequence $(c_n)_{n \in \N}$ in $A_{T,U}$ such that $c_n \to c_\infty \in A_{T,U}$. The sequence $(c_n)_{n \in \N}$ is convergent (so in particular it is bounded), meaning that in case it does not attain its $\sup$ we must have $c_\infty=\sup_{n \in \N}c_n$. Therefore, we can set
\begin{equation}
\hat{c}:= \begin{cases}
c_\infty &\mbox{ if } c_\infty = \sup_{n \in \N}c_n, \\
\max_{n \in \N}c_n &\mbox{ otherwise},
\end{cases}
\end{equation}
and we obviously have $\hat{c} \in A_{T,U}$. As a consequence, \eqref{ATU_c'1} and \eqref{ATU_c'2} imply that a. e. in $\R$
\begin{equation}
\forall n \in \N, \hspace{2mm} \lvert \be_{c_n}(U(\cdot)) \rvert \leq \lvert \be_{\hat{c}}(U(\cdot)) \rvert \in L^1(\R).
\end{equation}
Since we also have $\be_{c_n}(U(\cdot)) \to \be_{c_\infty}(U(\cdot))$ pointwise a. e. in $\R$, the Dominated Convergence Theorem gives the result.
\end{proof}
We now show a lower semicontinuity result, which in particular will imply the existence of the constrained solutions:
\begin{lemma}\label{LEMMA_LSC}
Assume that \ref{hyp_unbalanced}, \ref{hyp_compactness} and \ref{hyp_projections_inverse} hold. Let $T \geq 1$ be fixed. Let $(U_n^i)_{n \in \N}$ be a sequence in $X_T$ and $(c_n)_{n \in \N}$ a convergent sequence of positive real numbers such that
\begin{equation}\label{LSC_hyp}
\sup_{n \in \N}\bE_{c_n}(U_n^i) <+\infty.
\end{equation}
Then, there exists a sequence $(U_n)_{n \in \N}$ in $X_T$ and $U_\infty \in X_T$ such that up to extracting a subsequence in $(U_n,c_n)_{n \in \N}$ it holds
\begin{equation}\label{LSC_U_n}
\forall n \in \N, \hspace{2mm} \bE_{c_n}(U_n)=\bE_{c_n}(U_n^i),
\end{equation}
\begin{equation}\label{U_infty_conv1}
\forall t \in \R, \hspace{2mm} U_n(t) \rightharpoonup U_\infty(t) \mbox{ weakly in } \scrL
\end{equation}
\begin{equation}\label{U_infty_conv2}
U_n'e^{c_n \mathrm{Id}/2} \rightharpoonup U_\infty' e^{c_\infty \mathrm{Id}/2} \mbox{ weakly in } L^2(\R,\scrL)
\end{equation}
and
\begin{equation}\label{U_infty_LSC}
\bE_{c_\infty}(U_\infty) \leq \liminf_{n \to \infty}\bE_{c_n}(U_n),
\end{equation}
where $c_\infty:=\lim_{n \to \infty}c_n$.
\end{lemma}
\begin{proof}
Denote $M:=\sup_{n \in \N}E_{c_n}(U_n^i)$, which is finite by \eqref{LSC_hyp}. We will now use  \ref{hyp_projections_inverse}. We assume that 2. holds, the argument when 1. holds being similar and easier. Fix any $\bv^+ \in \scrF^\pm$ and for all $n \in \N$, set $v_n:= U_n^i(T) \in \scrF^+$. Define
\begin{equation}\label{LSC_U_n_def}
U_n: t \in \R \to \hat{P}_{(v_n,\bv^+)}(U^i_n(t)),
\end{equation}
where $\hat{P}_{\bv^+}$ is the differentiable operator introduced in \ref{hyp_projections_inverse}. We apply the properties summarized in 2. of \ref{hyp_projections_inverse}. Notice that for all $n \in \N$ we have $U_n \in X_T$ due to \eqref{hatP_dist}. The energy equality \eqref{LSC_U_n} follows from \eqref{D_inverse} and \eqref{hatP_CE}. Moreover, \eqref{hatP} implies that for all $n \in \N$
\begin{equation}
P^+(U_n(T))=P^+(\hat{P}_{(v_n,\bv^+)}(U^i_n(T)))=P^+(\hat{P}_{(v_n,\bv^+)}(v_n))=\bv^+,
\end{equation}
which in particular means
\begin{equation}\label{unif_bound_T}
\lVert U_n(T)-\bv^+ \rVert_{\scrL} \leq \frac{r_0^+}{2}.
\end{equation}
Notice now that $\CE(U(\cdot))$ is non-negative in $[T,+\infty)$ as $U \in X_T$ by \eqref{positivity_U} in Lemma \ref{LEMMA_well_defined}, therefore
\begin{align}\label{boundL2_Un}
\forall n \in \N, \hspace{2mm} \frac{1}{2}\int_\R \lVert U'_n(t) \rVert_\scrL^2 e^{c_n t} dt &\leq M-\int_\R \CE(U_n(t))e^{c_n t} dt\\ &\leq M-\int_{-\infty}^T \CE(U_n(t)) e^{c_n t}dt \leq \sup_{n \in \N}\left\{ M- \frac{a}{c_n}e^{c_nT} \right\}< +\infty.
\end{align} 
That is, we have that $(U_n'e^{c_n\mathrm{Id}/2})_{n \in \N}$ is uniformly bounded in $L^2(\R,\scrL)$. Therefore, there exists $\tilde{U} \in L^2(\R,\scrL)$ such that
\begin{equation}\label{LSC_Utilde}
U_n'e^{c_n\mathrm{Id}/2} \rightharpoonup \tilde{U} \mbox{ weakly in } L^2(\R,\scrL)
\end{equation}
up to subsequences. Such a thing implies
\begin{equation}\label{ineq_Un'}
\int_\R \lVert\tilde{U}(t) \rVert_{\scrL}^2 dt  \leq \liminf_{n \to \infty} \int_\R \lVert U'_n(t) \rVert_\scrL^2 e^{c_n t} dt.
\end{equation}
Now, notice that by \eqref{unif_bound_T} we have that $(U_n(T))_{n \in \N}$ is bounded in $\scrL$. Therefore, up to an extraction there exists $v_\infty \in \scrL$ such that
\begin{equation}\label{v_infinity}
U_n(T) \rightharpoonup v_\infty \mbox{ in } \scrL.
\end{equation}
As in \cite{smyrnelis}, we point out that
\begin{equation}\label{smyrnelis_obs}
\forall t \in \R, \forall n \in \N, \hspace{2mm} U_n(t)= U_n(T)+\int_{T}^t U_n'(s) ds.
\end{equation}
Now, notice that for all $t \in \R$ we have $\mathbf{1}_{(0,t)} e^{-c_n\mathrm{Id}/2} \to \mathbf{1}_{(0,t)}e^{-c_\infty \mathrm{Id}/2}$ in $L^\infty(\R)$, where $\mathbf{1}$ states for the indicator function of a set. Therefore, we obtain by \eqref{LSC_Utilde} and \eqref{v_infinity}
\begin{equation}
\forall t \in \R, \hspace{2mm} U_n(t) \rightharpoonup U_\infty(t):= v_\infty+\int_{T}^t \tilde{U}(s)e^{-c_\infty s/2}ds,
\end{equation}
which gives \eqref{U_infty_conv1}. Moreover, we have that $U_\infty \in H^1_{\loc}(\R,\scrL)$ and $U_\infty'=\tilde{U}e^{-c_\infty\mathrm{Id}/2}$, meaning by \eqref{LSC_Utilde} that \eqref{U_infty_conv2} also holds.

Recall now that $\CE$ is lower semicontinuous on $\scrL$ by \ref{hyp_unbalanced}, so that \eqref{U_infty_conv1} gives
\begin{equation}\label{simple_convergence_Un}
\forall t \in \R, \hspace{2mm} \CE(U_\infty(t)) \leq \liminf_{n \to \infty} \CE(U_n(t)).
\end{equation}

We need to show that $U_\infty \in X_T$ and to establish the inequality \eqref{U_infty_LSC}. 
\begin{itemize}
\item We begin by showing that $U_\infty \in X_T$. We need to show that for all $t \in [T,+\infty)$, it holds $U_\infty(t) \in \scrF^+_{r_0^+/2}$ and similarly for $(-\infty,-T]$. Fix $t \in [T,+\infty)$. We have that $U_n(t) \in \scrF^+_{r_0^+/2}$, so we can define the sequence $(\bv^+_n(t))_{n  \in \N}$ in $\scrF^+$ as $\bv^+_n(t):=P^+(U_n(t))$. We show that such a sequence is bounded. Indeed, we have
\begin{equation}
\forall n \in \N, \hspace{2mm} \lVert \bv_n^+(t) \rVert_{\scrL} \leq \frac{r_0^+}{2}+\lVert U_n(t) \rVert_{\scrL}
\end{equation}
and $(U_n(t))_{\scrL}$ converges weakly in $\scrL$, so in particular it is bounded. Therefore, up to an extraction we can assume that $\bv_n^+(t) \rightharpoonup \bv^+_{\infty}(t) \in \scrL$ and by \ref{hyp_compactness} we have $\bv^+_{\infty}(t) \in \scrF^\pm$. Using now the convergence properties we get the inequality
\begin{equation}
\lVert U_\infty(t)-\bv^+_\infty(t) \rVert_{\scrL} \leq \liminf_{n \to \infty} \lVert U_n(t)-\bv_n^+(t) \rVert_{\scrL} \leq \frac{r_0^+}{2},
\end{equation}
so that $U_n(t) \in \scrF^+_{r_0^+/2}$. An identical argument shows that for all $t \in (-\infty,-T]$ we have $U_n(t) \in \scrF^-_{r_0^-/2}$. Therefore, we have shown that $U_\infty \in X_T$.
\item Next, we prove \eqref{U_infty_LSC}.  We have
\begin{equation}
\sup_{n \in \N}\int_\R \CE(U_n(t))e^{c_nt}dt \leq M-\sup_{n \in \N} \int_\R \frac{\lVert U_n'(t) \rVert_\scrL^2}{2}e^{ct} dt <+\infty,
\end{equation}
by \eqref{boundL2_Un}. Hence, we can apply Fatou's Lemma to $(t \in \R \to \CE(U_n(t))e^{c_nt})_{n \in \N}$ (a sequence of functions uniformly bounded below by $a$) to show
\begin{equation}
\int_\R \liminf_{n \to +\infty} \CE(U_n(t))e^{c_nt}dt \leq \liminf_{n \to \infty} \int_\R \CE(U_n(t))e^{c_nt}dt,
\end{equation}
which, combined with \eqref{simple_convergence_Un} implies
\begin{equation}\label{ineq_VUn}
\int_\R \CE(U_\infty(t))e^{ct}dt \leq \liminf_{n \to \infty} \int_\R \CE(U_n(t))e^{c_nt}dt.
\end{equation}
Combining \eqref{ineq_Un'} and \eqref{ineq_VUn} we get
\begin{equation}
\bE_c(U_\infty) \leq \liminf_{n \to \infty} \int_\R \frac{\lVert U'_n(t) \rVert_\scrL^2}{2} e^{c_nt} dt + \liminf_{n \to \infty} \int_\R \CE(U_n(t)) e^{c_nt}dt,
\end{equation}
which, by superadditivity of the limit inferior gives \eqref{U_infty_LSC}.
\end{itemize}
\end{proof}

\subsection{Existence of an infimum for $\bE_c$ in $X_T$}

The goal now is to show that, for $T \geq 1$ and $c>0$ fixed, the infimum $\bm_{c,T}$ as defined in \eqref{mcT} is attained by a function in $X_T$. This will actually follow easily from Lemma \ref{LEMMA_LSC}.
\begin{lemma}\label{LEMMA_mcT}
Assume that \ref{hyp_unbalanced}, \ref{hyp_smaller_space}, \ref{hyp_compactness} and \ref{hyp_projections_inverse} hold. Let $c >0$, $T \geq 1$ and $\bm_{c,T}$ be as in \eqref{mcT}. Then, $\bm_{c,T}$ is attained for some $\bU_{c,T} \in X_T$. 
\end{lemma}
\begin{proof}
By \eqref{mcT_finite} in Lemma \ref{LEMMA_PSI}, we have that there exists a minimizing sequence $(U_n)_{n \in \N}$ for $\bE_c$ in $X_T$. We apply Lemma \ref{LEMMA_LSC} to $(U_n)_{n \in \N}$ and the sequence of speeds constantly equal to $c$. We obtain a function $\bU_{c,T} \in X_T$ such that
\begin{equation}
\bE_c(\bU_{c,T}) \leq \liminf_{n \to \infty}\bE_c(U_n)=\bm_{c,T},
\end{equation}
due to \eqref{U_infty_LSC}. Therefore, $\bm_{c,T}$ is attained by $\bU_{c,T}$ in $X_T$. 
\end{proof}

Subsequently, we show that assumption \ref{hyp_regularity} implies that the constrained minimizers are solutions of the equation in a certain set containing $(-T,T)$, with the proper regularity.

\begin{lemma}\label{LEMMA_mcT_solution}
Assume that \ref{hyp_regularity} holds. Let $c >0$, $T \geq 1$ and $\bm_{c,T}$ be as in \eqref{mcT}. Let $\bU_{c,T} \in X_T$ be such that $\bE_c(\bU_{c,T})=\bm_{c,T}$. Then, $\bU_{c,T} \in \CA((-T,T))$, $\CA((-T,T))$ as in \eqref{CA_def} and
\begin{equation}\label{mcT_equation}
\bU_{c,T}''-D_{\scrL}\CE(\bU_{c,T})=-c\bU_{c,T}' \mbox{ in } (-T,T).
\end{equation}
Moreover, if $t \geq T$ is such that
\begin{equation}\label{mcT_no_constrain_+}
\dist_{\scrL}(\bU_{c,T}(t),\scrF^+)< \frac{r_0^+}{2},
\end{equation}
then, there exists $\delta^+(t)>0$ such that $\bU_{c,T} \in \CA((t-\delta^+(t),t+\delta^+(t)))$ and
\begin{equation}\label{mcT_equation_delta+}
\bU_{c,T}''-D_{\scrL}\CE(\bU_{c,T})=-c\bU_{c,T}' \mbox{ in } (t-\delta^+(t),t+\delta^+(t)).
\end{equation}
Similarly, if $t \leq -T$ is such that
\begin{equation}\label{mcT_no_constrain_-}
\dist_{\scrL}(\bU_{c,T}(t),\scrF^-) < \frac{r_0^-}{2},
\end{equation}
then, there exists $\delta^-(t)>0$ such that $\bU_{c,T} \in \CA((t-\delta^-(t),t+\delta^-(t)))$ and
\begin{equation}\label{mcT_equation_delta-}
\bU_{c,T}''-D_{\scrL}\CE(\bU_{c,T})=-c\bU_{c,T}' \mbox{ in } (t-\delta^-(t),t+\delta^-(t)).
\end{equation}
\end{lemma}
\begin{proof}
We first show that
\begin{equation}\label{bounded_FP}
\forall t \in \R, \hspace{2mm} \FP(\bU_{c,T}(t))=\bU_{c.T}(t),
\end{equation}
where $\FP$ is the map from \ref{hyp_regularity}. We claim that the function
\begin{equation}
\bU_{c,T}^\FP: t \in \R \to \FP(\bU_{c,T}(t))
\end{equation}
belongs to $X_T$. Indeed, this follows from \eqref{FP_2})and \eqref{FP_3}. Property \eqref{FP_1} implies that
\begin{equation}\label{bounded_FP_ineq1}
\forall t \in \R, \hspace{2mm} \CE(\bU_{c,T}^\FP(t)) \leq \CE(\bU_{c,T}(t)).
\end{equation}
Take now $t \in \R$ and $s \in \R \setminus \{ t \}$. Property \eqref{FP_2} implies that
\begin{equation}
\left\lVert \frac{\bU_{c,T}^\FP(t)-\bU_{c,T}^\FP(s)}{t-s} \right\rVert_{\scrL} \leq \left\lVert \frac{\bU_{c,T}(t)-\bU_{c,T}(s)}{t-s} \right\rVert_{\scrL},
\end{equation}
which, by Lebesgue's differentiation Theorem implies that
\begin{equation}\label{bounded_FP_ineq2}
\mbox{for a. e. } t \in \R, \hspace{2mm} \lVert (\bU_{c,T}^\FP)'(t) \rVert_{\scrL} \leq \lVert \bU_{c,T}'(t) \rVert_{\scrL}.
\end{equation}
By contradiction, assume now that there exists $t \in \R$ such that $\bU_{c,T}(t) \not = \FP(\bU_{c,T}(t))=\bU_{c,T}^\FP(t)$. Property \eqref{FP_2} implies that $\FP$ is a $\scrL$-continuous map. Therefore, since $\bU_{c,T}$ is also $\scrL$-continuous, we must have that for some non-empty interval $I_t \ni t$, it holds
\begin{equation}
\forall s \in I_t, \hspace{2mm} \bU_{c,T}(s) \not = \FP(\bU_{c,T}(s))=\bU_{c,T}^\FP(s),
\end{equation}
so that, using \eqref{FP_1} we get
\begin{equation}
\forall s \in I_t, \hspace{2mm} \CE(\bU_{c,T}^\FP(s)) < \CE(\bU_{c,T}(s))
\end{equation}
so that, combining with \eqref{bounded_FP_ineq1} and \eqref{bounded_FP_ineq2} we obtain
\begin{equation}
\bE_c(\bU_{c,T}^\FP)< \bE_c(\bU_{c,T})=\bm_{c,T},
\end{equation}
which contradicts the definition of $\bm_{c,T}$ \eqref{mcT} since $\bU_{c,T}^\FP \in X_T$. Therefore, we have shown that \eqref{bounded_FP} holds. Next, notice that
\begin{equation}
\bE_c(\bU_{c,T};[-T,T]) \leq \bm_{c,T}-\frac{a}{c}e^{-cT}<+\infty
\end{equation}
and for any $\phi \in \CC^1_c((-T,T),(\scrH,\lVert \cdot \rVert_{\scrH}))$ we have $\bU_{c,T}+\phi \in X_T$, so that $\bE_c(\bU_{c,T}) \leq \bE_c(\bU_{c,T}+\phi)$. Therefore, the restriction of $\bU_{c,T}$ in $(-T,T)$ is a local minimizer of $\bE_c(\cdot,[-T,T])$ in the sense of Definition \ref{def_local_minimizer}. Since $\bU_{c,T}$ also verifies \eqref{bounded_FP}, we can apply the regularity assumption \ref{hyp_regularity}. Therefore, $\bU_{c,T} \in \CA((-T,T))$ and \eqref{mcT_equation} holds.  Assume now that there exists $t \geq T$ such that \eqref{mcT_no_constrain_+} holds. Then, there exists $\bv^+(t) \in \scrF^+$ such that
\begin{equation}
\lVert \bU_{c,T}(t)-\bv^+(t) \rVert_{\scrL}< \frac{r_0^+}{2}
\end{equation}
which, since $\bU_{c,T}$ is $\scrL$-continuous, implies that there exists $\delta^+(t)>0$ such that
\begin{equation}
\forall s \in (t-\delta^+(t),t+\delta^+(t)), \hspace{2mm} \lVert \bU_{c,T}(s)-\bv^+(t) \rVert_{\scrL}< \frac{r_0^+}{2}-d^+(t)
\end{equation}
where
\begin{equation}
d^+(t):= \frac{1}{2}\left( \frac{r_0^+}{2}-\lVert \bU_{c,T}(t)-\bv^+(t) \rVert_{\scrL} \right)>0.
\end{equation}
Therefore, if $\phi \in \CC^1_{c}((t-\delta^+(t),t+\delta^+(t)),(\scrH,\lVert \cdot \rVert_{\scrH}))$ is such that 
\begin{equation}
\max_{t \in [t-\delta^+(t),t+\delta^+(t)]} \lVert \phi(t) \rVert_{\scrH} \leq \frac{d^+(t)}{2}
\end{equation}
we have that
\begin{equation}
\forall s \in (t-\delta^+(t),t+\delta^+(t)), \hspace{2mm} \lVert \bU_{c,T}(s)+\phi(s)-\bv^+(t) \rVert_{\scrL}< \frac{r_0^+}{2}-\frac{d^+(t)}{2}
\end{equation}
so that $\bU_{c,T}+\phi \in X_T$. Meaning that $\bE_{c,T}(\bU_{c,T}) \leq \bE_{c,T}(\bU_{c,T}+\phi)$. Since $\phi$ is supported on $[t-\delta^+(t),t+\delta^+(t)]$, the previous implies that
\begin{equation}
\bE_{c,T}(\bU_{c,T};[t-\delta^+(t),t+\delta^+(t)]) \leq \bE_{c,T}(\bU_{c,T}+\phi;[t-\delta^+(t),t+\delta^+(t)]),
\end{equation}
so that $\bU_{c,T}$ is a local minimizer of $\bE_c(\cdot;[t-\delta^+(t),t+\delta^+(t)])$ in the sense of Definition \ref{def_local_minimizer}. Since \eqref{bounded_FP} also holds, we can apply \ref{hyp_regularity} and obtain that $\bU_{c,T} \in \CA((t-\delta^+(t),t+\delta^+(t)))$ and equation \eqref{mcT_equation_delta+} holds. If $t \leq -T$ is such that \eqref{mcT_no_constrain_-} holds, the same reasoning shows that for some $\delta^-(t)>0$, $\bU_{c,T}  \in \CA((t-\delta^-(t),t+\delta^-(t)))$ and \eqref{mcT_no_constrain_+} holds, which concludes the proof of the result.
\end{proof}

\subsection{The comparison result}
The goal of this subsection is to obtain relevant information on the behavior of the constrained minimizers. Such information is contained in Corollary \ref{COROLLARY_constrained_behavior} and it will allow us to remove the constraints later on. In order to carry on these arguments, assumption \ref{hyp_levelsets} will become necessary since it will show that our problem can be somehow dealt with as in the balanced one, which will allow us to argue in a fashion similar to Smyrnelis \cite{smyrnelis}. We begin by introducing some constants.  For $0<r\leq r_0^\pm$, recall the definition of $\kappa_r^\pm$ introduced in \eqref{kappa_r}, Lemma \ref{LEMMA_positivity}. We define
\begin{equation}\label{eta_0+}
\eta^+_0:= \min\left\{\sqrt{e^{-1}\frac{r_0^+}{4}\sqrt{2\kappa^+_{r_0^+/4}}}, \frac{r_0^+}{4} \right\}>0,
\end{equation}
\begin{equation}\label{r+}
\hat{r}^+:= \frac{r_0^+}{C^++1}>0,
\end{equation}
\begin{equation}\label{eps_0+}
\CE^+_{\max}:= \frac{1}{\left(C^+\right)^2(C^++1)}\min\left\{\frac{(\eta_0^+)^2}{4}, \kappa^+_{\eta_0^+},\beta^+(\hat{r}^+),\beta^+(\eta_0^+)\right\}>0,
\end{equation}
where the constants $C^\pm$, $\beta^\pm(\hat{r}^\pm), \beta(\eta_0^\pm)$ were introduced in \ref{hyp_projections_H}. Recall that in \eqref{intr_eta_0-}, \eqref{intr_r-}, \eqref{intr_eps_0-} we introduced the analogous constants
\begin{equation}\label{eta_0-}
\eta_0^-:= \min\left\{\sqrt{e^{-1}\frac{r_0^-}{4}\sqrt{2(\kappa^-_{r_0^-/4}-a)}}, \frac{r_0^-}{4} \right\}>0,
\end{equation}
\begin{equation}\label{r-}
\hat{r}^-:= \frac{r_0^-}{C^-+1}>0
\end{equation}
and
\begin{equation}\label{eps_0-}
\CE_{\max}^-:= \frac{1}{(C^-)^2(C^-+1)}\min\left\{ \frac{(\eta_0^-)^2}{4},\kappa_{\eta_0^-}^--a,\beta^-(\hat{r}^-),\beta^-(\eta_0^-)\right\}>0.
\end{equation}
For any $U \in X_T$, define
\begin{equation}\label{comparison_t-}
t^-(U,\CE_{\max}^-):= \sup\left\{ t \in \R: \CE(U(t))\leq a+\CE_{\max}^- \mbox{ and } \dist_{\scrL}(U(t),\scrF^-) \leq \frac{r_0^-}{2} 	\right\}
\end{equation}
and
\begin{equation}\label{comparison_t+}
t^+(U,\CE_{\max}^+):= \inf\left\{ t \in \R: \CE(U(t))\leq \CE_{\max}^+ \mbox{ and } \dist_{\scrL}(U(t),\scrF^+) \leq \frac{r_0^+}{2} \right\}.
\end{equation}
We have the following technical property:
\begin{lemma}\label{LEMMA_r}
Assume that \ref{hyp_unbalanced} and \ref{hyp_projections_H} hold. Let $\hat{r}^\pm>0$ be as in \eqref{r+}, \eqref{r-} and $\CE_{\max}^\pm$ be as in \eqref{eps_0+}, \eqref{eps_0-}. Then, if $v \in \scrF^\pm_{r_0^\pm}$ is such that
\begin{equation}\label{r_hyp}
\CE(v) \leq \min \{ \pm(-a),0 \} +\beta^\pm(\hat{r}^\pm),
\end{equation}
then
\begin{equation}\label{lambda_general}
\forall \lambda \in [0,1], \hspace{2mm} \CE(\lambda v+ (1-\lambda)P^\pm(v)) \leq \min\{ \pm(-a),0 \} +\sfC^\pm(\CE(v)- \min\{ \pm(-a),0\}).
\end{equation}
In particular, if
\begin{equation}\label{r_hyp_eps}
\CE(u) \leq \min\{\pm(-a),0\}+\CE_{\max}^\pm,
\end{equation}
then
\begin{equation}\label{lambda_energy}
\forall \lambda \in [0,1], \hspace{2mm} \CE(\lambda v+ (1-\lambda)P^\pm(v))\leq \min\{\pm(-a),0\}+\sfC^\pm \CE_{\max}^\pm.
\end{equation}
The constants $\CE_{\max}^\pm$, $\sfC^\pm$ where defined in \eqref{eps_0-}, \eqref{eps_0+} and \eqref{sfC} respectively. $P^\pm(u)$ is the projection introduced in \ref{hyp_unbalanced}.
\end{lemma}
\begin{proof}
Assume that \eqref{r_hyp} holds for $v \in \scrF^\pm_{r_0^\pm}$. Then, invoking \ref{hyp_projections_H} we have that $v \in \scrF_{\scrH,\hat{r}^\pm}$, so in particular the projection $P^\pm_\scrH(u)$ is well defined. Fix $\lambda \in [0,1]$. Since $v \in \scrF^\pm_{r_0^\pm}$, the projection $P^\pm(v)$ is well defined by \ref{hyp_unbalanced}. Using \eqref{H_difference_projections} we obtain
\begin{align}\label{ineq_lambda}
\lVert \lambda v + (1-\lambda)P^\pm(v)- P^\pm(v) \rVert_{\scrH}& =\lambda \lVert v-P^\pm(v) \rVert_{\scrH}\\ & \leq (C^\pm+1)\lVert v-P^\pm_{\scrH}(v) \rVert_{\scrH} \leq (C^\pm+1)\hat{r}^\pm
\end{align}
so that $\lambda v + (1-\lambda)P^\pm(v) \in \scrF^\pm_{\scrH,r_0^\pm}$ by the definition of $\hat{r}^\pm$ in \eqref{r+}, \eqref{r-}. Using now again \eqref{H_difference_projections} along with the estimate \eqref{H_local_est} in \ref{hyp_projections_H}, we get
\begin{equation}
\lVert P^\pm(v)-P^\pm_{\scrH}(v) \rVert_{\scrH}^2 \leq (C^\pm)^2(\CE(v)- \min\{ \pm(-a),0\})
\end{equation}
which, plugging into \eqref{ineq_lambda}, gives
\begin{align}
\lVert \lambda v + (1-\lambda)P^\pm(v)- P_\scrH^\pm(v) \rVert_{\scrH}^2 \leq \frac{1}{2}((C^\pm)^2+(C^\pm+1)^2) (\CE(v)- \min\{ \pm(-a),0\}),
\end{align}
that, using again \eqref{H_local_est}, implies exactly \eqref{lambda_general}. Assuming now that \eqref{r_hyp_eps} holds, we have by \eqref{eps_0-}, \eqref{eps_0+} that in particular \eqref{r_hyp} holds. Therefore, \eqref{lambda_energy} follows from \eqref{lambda_general}.
\end{proof}

Next, we have the following property:
\begin{lemma}\label{LEMMA_comparison_preliminary}
Assume that  \ref{hyp_projections_H} and \ref{hyp_levelsets} hold. Let $c>0$ and $T \geq 1$. Assume that $U \in X_T$ is such that $\bE_c(U) \leq 0$. Then the quantities $t^-(U,\CE_{\max}^-)$ and $t^+(U,\CE_{\max}^+)$ defined in \eqref{comparison_t-} and \eqref{comparison_t+}, respectively, are well defined as real numbers. Moreover, it holds that
\begin{equation}\label{ineq_t-}
\CE(U(t^-(U,\CE_{\max}^-)) \leq a+\CE_{\max}^-, \hspace{1mm} \dist_{\scrL}(U(t^-(U,\CE_{\max}^-)),\scrF^-) \leq \frac{r_0^-}{2}
\end{equation}
and
\begin{equation}\label{ineq_t+}
\CE(U(t^+(U,\CE_{\max}^+)) \leq \CE_{\max}^+, \hspace{1mm} \dist_{\scrL}(U(t^+(U,\CE_{\max}^+)),\scrF^+) \leq \frac{r_0^+}{2}.
\end{equation}
\end{lemma}
\begin{proof}
Using that $\bE_c(U) \leq 0$ and the fact that $\{ t \in \R: \CE(U(t)) >0 \}$ is nonempty since $U \in X_T$, we must have that
\begin{equation}
\{ t \in \R: \CE(U(t)) <0 \} \not = \emptyset
\end{equation}
and if $v \in \scrL$ is such that $\CE(v) <0$, then we have $\dist_{\scrL}(v,\scrF^-) \leq r_0^-/2$ by \eqref{negative_levelset} in \ref{hyp_levelsets} and $\CE(v) < a+\CE_{\max}^-$ since we assume $a+\CE_{\max}^->0$. Therefore, $t^-(U,\CE_{\max}^-)$ is well defined, as we have shown that
\begin{equation}
\left\{ t \in \R: \CE(U(t))\leq a+\CE^-_{\max} \mbox{ and } \dist_{\scrL}(U(t),\scrF^-) \leq \frac{r_0^-}{2} 	\right\}\not = \emptyset
\end{equation}
and such set is bounded above by $T$, because $\scrF^-_{r_0^-/2} \cap \scrF^+_{r_0^+/2}=\emptyset$. Using Lemma \ref{LEMMA_limit+}, we have that $t^+(U,\CE_{\max}^+)$ is well defined. Finally, inequalities \eqref{ineq_t-} and \eqref{ineq_t+} follow because $t \in \R \to \CE(U(t)) \in \R$ is lower semicontinuous by \ref{hyp_unbalanced} and $t  \to \dist_{\scrL}(U(t),\scrF^\pm)$ is continuous whenever $U(t) \in \scrF^\pm_{r_0^\pm/2}$ by \ref{hyp_unbalanced} (recall that $t \in \R \to U(t) \in \scrL$ is continuous because $U \in H^1_{\loc}(\R,\scrL)$).
\end{proof}

The main work is done by the following result:
\begin{proposition}\label{PROPOSITION_comparison}
Assume that \ref{hyp_projections_H} and \ref{hyp_levelsets} hold. Let $c>0$ and $T \geq 1$. Consider $U \in X_T$ with $\bE_c(U) \leq 0$. Let $t^\pm:=t^\pm(U,\CE_{\max}^\pm,\eta_0^\pm)$ be as in \eqref{comparison_t-} and \eqref{comparison_t+}. Then, $t^\pm$ are well defined by Lemma \ref{LEMMA_comparison_preliminary}. Moreover, if there exists $\tilde{t}^- < t^-$ such that
\begin{equation}\label{bad_-}
r_0^- \geq \dist_{\scrL}(U(\tilde{t}^-),\scrF^-)\geq \frac{r_0^-}{2}
\end{equation}
then, we can find $\tilde{U}^- \in X_T$ such that
\begin{equation}\label{comparison_result_-}
\forall t \leq t^-, \hspace{2mm} \dist_{\scrL}(\tilde{U}^-(t),\scrF^-) < \frac{r_0^-}{2},
\end{equation}
and
\begin{equation}\label{comparison_energy_-}
\bE_c(\tilde{U}^-) < \bE_c(U).
\end{equation}
Analogously, if there exists $\tilde{t}^+ > t^+$ such that
\begin{equation}\label{bad_+}
r_0^+ \geq \dist_{\scrL}(U(\tilde{t}^+),\scrF^+) \geq \frac{r_0^+}{2},
\end{equation}
then we can find $\tilde{U}^+ \in X_T$ such that
\begin{equation}\label{comparison_result_+}
\forall t \geq t^+, \hspace{2mm} \dist_{\scrL}(\tilde{U}^+(t),\scrF^+) < \frac{r_0^+}{2}
\end{equation}
and
\begin{equation}\label{comparison_energy_+}
\bE_c(\tilde{U}^+) < \bE_c(U).
\end{equation}
Furthermore, we have that
\begin{equation}\label{sfT_prop}
0<t^+-t^- \leq \sfT_\star(c),
\end{equation}
where
\begin{equation}\label{sfT_definition}
\sfT_\star(c):= \frac{1}{c}\ln\left( \frac{-a}{\alpha_\star}+1 \right),
\end{equation}
with $\alpha_\star>0$ a constant which is independent from $c$, $T$ and $U$.
\end{proposition}
The idea of the proof of Proposition \ref{PROPOSITION_comparison} is pictured in Figure \ref{Figure_proof}.
\begin{proof}[Proof of Proposition \ref{PROPOSITION_comparison}]
We begin by proving the first part of the result for $\scrF^-$. Recall that Lemma \ref{LEMMA_comparison_preliminary} gives
\begin{equation}\label{E_t-}
\CE(U(t^-)) \leq a+\CE_{\max}^-.
\end{equation}
and $U(t^-) \in \scrF^-_{r_0^-/2}$. Since $a+\CE_{\max}^- \leq \beta(\CE_{\max}^-)$ by the definition of $\CE_{\max}^-$, \eqref{eps_0-}, we have by \ref{hyp_projections_H} that
\begin{equation}\label{dist_t-}
\mathrm{dist}_\scrL(U(t^-),\scrF^-) \leq \eta_0^-.
\end{equation}

Assume that there exists $\tilde{t}^- < t^-$ such that \eqref{bad_-} is satisfied. Moreover, we assume, as we can, that
\begin{equation}\label{tilde_t_sup}
\tilde{t}^-:= \max\left\{ t \leq t^-: \dist_{\scrL}(U(t),\scrF^-) \geq \frac{r_0^-}{2}\right\}
\end{equation}
(the $\sup$ can be replaced by a $\max$ by continuity). Define
\begin{equation}\label{t-0}
t^-_0:= \inf \{ t \in [\tilde{t}^-,t^-]: \CE(U(t)) \leq a+\CE_{\max}^- \mbox{ and } \dist_{\scrL}(U(t)) \leq \eta_0^-\}.
\end{equation}
Let $\bv^- := P^-(U(t^-_0)) \in \scrF^-$, with $P^-$ as in \ref{hyp_unbalanced}. Notice that due to \eqref{tilde_t_sup}, we have
\begin{equation}\label{t0_and_t-}
\forall t \in [t^-_0,t^-], \hspace{2mm} \dist(U(t),\scrF^-) < \frac{r_0^-}{2}.
\end{equation}

The proof now bifurcates according to two possible cases: 
\begin{itemize}
\item $t^-_0 \leq \tilde{t}^-+1$. In that case, set
\begin{equation}
\tilde{U}^-(t):= \begin{cases}
\bv^- &\mbox{ if } t \leq t_0^--1,\\
(t_0^--t)\bv^-+(t-t_0^-+1)U(t_0^-) &\mbox{ if } t^-_0-1 \leq t \leq t_0^-,\\
U(t) &\mbox{ if } t \geq t^-_0,
\end{cases}
\end{equation}
which belongs to $X_T$. Due to the definition of $\tilde{U}^-$ and \eqref{t0_and_t-}, we have that $\tilde{U}^-$ satisfies \eqref{comparison_result_-}. It remains to check \eqref{comparison_energy_-}. We have
\begin{equation}\label{case1_est_1_Utilde}
\int_{t_0^--1}^{t_0^-}\be_c(\tilde{U}^-(t))dt \leq \int_{t_0^--1}^{t_0^-}\left[ \frac{\lVert U(t_0^-) - \bv^- \rVert_{\scrL}^2}{2}+\CE(\tilde{U}^-(t))\right]e^{ct}dt.
\end{equation}
Fix $t \in [t^-_0-1,t^-_0]$. Choosing $\lambda=t-t_0^-+1\in [0,1]$ and applying \eqref{lambda_energy} in Lemma \ref{LEMMA_r} and \eqref{E_t-}, we have that
\begin{equation}
\CE(\tilde{U}^-(t)) \leq a+\sfC^-\CE_{\max}^-.
\end{equation}
The previous fact combined with \eqref{dist_t-}, \eqref{E_t-} and \ref{case1_est_1_Utilde}, gives
\begin{equation}\label{case1_est_2_Utilde}
\int_{t_0^--1}^{t_0^-} \be_c(\tilde{U}^-(t))dt \leq \left(\frac{(\eta_0^-)^2}{2} + \sfC^-\CE_{\max}^-\right) e^{ct^-_0} +\frac{a(e^{ct^-_0}-e^{c(t_0^--1)})}{c}.
\end{equation}
The continuity of $U$ and \eqref{bad_-} implies that there exists $\tilde{t}^-_2 \in (\tilde{t}^-,t^-_0)$ such that
\begin{equation}\label{t_tilde_2}
\dist(U(\tilde{t}_2^-),\scrF^-)=\frac{r_0^-}{4} \mbox{ and }\forall t \in [\tilde{t}^-,\tilde{t}^-_2], \hspace{2mm} \dist_{\scrL}(U(t),\scrF^-) \geq \frac{r_0^-}{4}.
\end{equation}
Using \eqref{t_tilde_2}, we get
\begin{equation}\label{case1_est_01_U}
\int_{\tilde{t}^-}^{\tilde{t}_2^-} \lVert U'(t) \rVert_{\scrL}e^{ct}dt \geq \frac{r_0^-e^{c\tilde{t}^-}}{4}
\end{equation}
and \eqref{t_tilde_2} also implies
\begin{equation}\label{case1_est_02_U}
\forall t \in [\tilde{t}^-,\tilde{t}_2^-], \hspace{2mm} \CE(U(t)) \geq \kappa^-_{r_0^-/4}.
\end{equation}
Inequalities \eqref{case1_est_01_U} and \eqref{case1_est_02_U} along with the definition of $\eta_0^-$ in \eqref{eta_0-} and Young's inequality give
\begin{align}\label{case1_est_1_U}
\int_{\tilde{t}^-}^{\tilde{t}_2^-} \be_c(U(t))dt & \geq \frac{r_0^-e^{c\tilde{t}^-_0}}{4}\sqrt{2(\kappa^-_{r_0^-/4}-a)}+a\frac{e^{c\tilde{t}_2^-}-e^{c\tilde{t}^-}}{c} \\ &= e (\eta_0^-)^2e^{c\tilde{t}^-}+a\frac{e^{c\tilde{t}_2^-}-e^{c\tilde{t}^-}}{c} ,
\end{align}
which gets to, using also that $\tilde{t}^- \geq t^-_0-1$,
\begin{align}\label{case1_est_2_U}
\int_{-\infty}^{t^{-}_0}\be_c(U(t))dt &= \int_{(-\infty,t^-_0] \setminus [\tilde{t}^-,\tilde{t}_2^-]}\be_c(U(t))dt + \int_{\tilde{t}^-}^{\tilde{t}_2^-} \be_c(U(t))dt \\
&\geq e(\eta_0^-)^2e^{c\tilde{t}^-}+ a\frac{e^{ct^-_0}}{c} \geq (\eta_0^-)^2e^{ct_0^-}+ a\frac{e^{ct^-_0}}{c} .\\
\end{align}
Using now \eqref{case1_est_2_Utilde} we get
\begin{align}\label{case1_est_3_Utilde}
\int_{-\infty}^{t^-_0} \be_c(\tilde{U}^-(t))dt &= \int_{-\infty}^{t^-_0-1}\be_c(\tilde{U}^-(t))dt + \int_{t^-_0-1}^{t^-_0}\be_c(\tilde{U}^-(t))dt \\
& \leq \left(\frac{(\eta_0^-)^2}{2} + \sfC^-\CE_{\max}^-\right) e^{ct^-_0} +\frac{ae^{ct^-_0}}{c}.
\end{align}
Therefore, subtracting \eqref{case1_est_3_Utilde} from \eqref{case1_est_2_U}, we get
\begin{align}
&\int_{-\infty}^{t^{-}_0}\be_c(U(t))dt-\int_{-\infty}^{t^-_0} \be_c(\tilde{U}^-(t))dt  \\&\geq \left( \frac{(\eta_0^-)^2}{2}+\sfC^-\CE_{\max}^-\right) e^{ct^-_0}, \\
\end{align}
which is positive because \eqref{eps_0-} implies
\begin{equation}
\sfC^-\CE_{\max}^- \leq \frac{(\eta_0^-)^2}{4}.
\end{equation}
Since $\tilde{U}^-$ and $U$ coincide in $[t^-_0,+\infty)$, the proof of the first case is concluded.
\item $t_0^- > \tilde{t}+1$. In such a case, set
\begin{equation}
\tilde{U}^-(t):= \begin{cases}
\bv^- &\mbox{ if } t \leq \tilde{t}^-, \\
(t-\tilde{t}^-)U(t_0^-)+(\tilde{t}^-+1-t)\bv^- &\mbox{ if } \tilde{t}^- \leq t \leq \tilde{t}^-+1,\\
U(t^-_0) &\mbox{ if } \tilde{t}^-+1 \leq t \leq t^-_0, \\
U(t) &\mbox{ if } t^-_0 \leq t,
\end{cases}
\end{equation}
which clearly belongs to $X_T$ and for all $t \leq t^-$, $U(t) \in \scrF^-_{r_0^-/2}$ by \eqref{t0_and_t-}. We have that $\tilde{U}^-$ is constant in $[\tilde{t}^-+1,t_0^-]$, therefore
\begin{equation}
\int_{\tilde{t}^-+1}^{t^-_0}\be_c(\tilde{U}^-(t))dt \leq (a+\CE_{\max}^-) \frac{e^{ct_0^-}-e^{c\tilde{t}^-+1}}{c}
\end{equation}
and, due to the definitions of $\CE_{\max}^-$ in \eqref{eps_0-} and $t_0^-$ in \eqref{t-0}
\begin{equation}
\int_{\tilde{t}^-+1}^{t_0^-} \be_c(U(t))dt \geq \min\{ a+\CE_{\max}^-,\kappa_{\eta_0^-} \}\frac{e^{ct_0^-}-e^{c\tilde{t}^-+1}}{c} \geq \int_{\tilde{t}^-+1}^{t^-_0}\be_c(\tilde{U}^-(t))dt,
\end{equation}
because $\CE_{\max}^-+a >0$ by \ref{hyp_levelsets} and $t_0 \geq \tilde{t}^-+1$ by assumption. Hence
\begin{equation}
\int_{\tilde{t}^-+1}^{+\infty}\be_c(\tilde{U}(t))dt \leq \int_{\tilde{t}^-+1}^{\infty} \be_c(U(t))dt.
\end{equation}
Arguing as in the first case scenario, we can prove that
\begin{equation}
\int_{-\infty}^{\tilde{t}^-+1} \be_c(\tilde{U}(t))dt  < \int_{-\infty}^{\tilde{t}^-+1} \be_c(U(t))dt,
\end{equation}
which concludes the proof of the second case.
\end{itemize}
To sum up, we have shown that if \eqref{bad_-} is satisfied, then there exists $\tilde{U}^-$ such that \eqref{comparison_result_-} and \eqref{comparison_energy_-} hold, as we wanted.

Assume now that there exists $\tilde{t}^+ > t^+$ such that \eqref{bad_+} holds. As before, Lemma \ref{LEMMA_comparison_preliminary} and the definition of $\CE_{\max}^+$, \eqref{eps_0+}, imply that $t^+:= t^+(U,\CE_{\max}^+)$ is such that
\begin{equation}\label{dist_t+}
\mathrm{dist}_\scrL(U(t^+),\scrF^+) \leq \eta_0^+
\end{equation}
and
\begin{equation}\label{E_t+}
\CE(U(t^+)) \leq \CE_{\max}^+.
\end{equation}
We claim that we can assume without loss of generality that
\begin{equation}\label{comparison_positivity}
\forall t \in [t^+,+\infty), \hspace{2mm} \CE(U(t)) \geq 0.
\end{equation}
Indeed, if we can find $t_0 \in (t^+,+\infty)$ such that $\CE(U(t))<0$, then by \ref{hyp_levelsets} we have have that $\CE(U(t_0)) \leq a	+\CE_{\max}^-$ and by \eqref{negative_levelset} in \ref{hyp_levelsets} we also have that $\dist_{\scrL}(U(t_0),\scrF^-) \leq r_0^-/2$. Therefore, we have by the definitions \eqref{comparison_t-} and \eqref{comparison_t+} that $t^- \geq t_0$ and $t^+ > t^-$, a contradiction since we assume $t_0 > t^+$. 

For the positive case, the proof is simpler as it suffices to define $\bv^+:= P^+(U(t^+))$ and
\begin{equation}\label{Utilde+_def}
\tilde{U}^+(t):=\begin{cases}
\bv^+ &\mbox{ if } t \geq t^++1 \\
(t-t^+)\bv^++(t^++1-t)U(t^+) &\mbox{ if } t^++1 \geq t \geq t^+,\\
U(t) &\mbox{ if } t^+ \geq t,
\end{cases}
\end{equation}
which is such that $U \in X_T$. Moreover, it holds that for all $t \geq t^+$, we have $\tilde{U}^+(t) \in \scrF^+_{r_0^+/2}$. Therefore, the requirements \eqref{comparison_result_+} and \eqref{comparison_energy_+} hold for $\tilde{U}^+$. It remains to check that \eqref{comparison_energy_+} is also fulfilled. We have that
\begin{equation}\label{case2_est_Utilde_1}
\int_{t^+}^{t^++1}\be_c(\tilde{U}^+(t))dt = \int_{t^+}^{t^++1} \left[\frac{\left\lVert U(t^+)-\bv^+ \right\rVert_{\scrL}^2}{2}+\CE(\tilde{U}^+(t)) \right] e^{ct}dt.
\end{equation}
Using \eqref{lambda_energy} in Lemma \ref{LEMMA_r} and \eqref{E_t+}, we get
\begin{equation}\label{case2_est_Utilde_11}
\int_{t^+}^{t^++1} \CE(\tilde{U}^+(t))e^{ct}dt  \leq \sfC^+\CE_{\max}^+ e^{c(t^++1)}.
\end{equation}
Using now \eqref{dist_t+}, we get
\begin{equation}\label{case2_est_Utilde_12}
\int_{t^+}^{t^++1} \frac{\left\lVert U(t^+)-\bv^+ \right\rVert_{\scrL}^2}{2}  e^{ct}dt \leq \frac{(\eta_0^+)^2}{2}e^{c(t^++1)}.
\end{equation}
Plugging \eqref{case2_est_Utilde_11} and \eqref{case2_est_Utilde_12} into \eqref{case2_est_Utilde_1}, we get
\begin{equation}\label{case2_est_Utilde}
\int_{t^+}^{t^++1}\be_c(\tilde{U}^+(t))dt \leq \left(\frac{(\eta_0^+)^2}{2}+\sfC^+\CE_{\max}^+\right) e^{ct^++1}.
\end{equation}
Since for all $t \geq t^++1$ we have $\tilde{U}^+(t)=\bv^+$, we obtain from \eqref{case2_est_Utilde}
\begin{equation}\label{case2_est_Utilde_def}
\int_{t^+}^{+\infty}\be_c(\tilde{U}^+(t))dt \leq \left(\frac{(\eta_0^+)^2}{2}+\sfC^+\CE_{\max}^+\right) e^{ct^++1}.
\end{equation}
Next, notice that by continuity we can find $\tilde{t}_2^+ \in (t^+,\tilde{t}^+)$ such that
\begin{equation}\label{t_tilde_2+}
\dist(U(\tilde{t}_2^+),\scrF^+)=\frac{r_0^+}{4} \mbox{ and }\forall t \in [\tilde{t}^+,\tilde{t}_2^+],  \hspace{2mm} \frac{r_0^+}{2} \geq \dist_{\scrL}(U(t),\scrF^+) \geq \frac{r_0^+}{4}.
\end{equation}
Therefore, using \eqref{bad_+} and \eqref{t_tilde_2+}, we obtain
\begin{equation}\label{case2_est_01_U}
\int_{\tilde{t_2}^+}^{\tilde{t}^+} \lVert U'(t) \rVert_{\scrL} e^{ct} dt \geq \frac{r_0^+}{4}e^{ct^++1}e^{-1}
\end{equation}
and \eqref{t_tilde_2+}, \eqref{coercivityL} in \ref{hyp_unbalanced} imply
\begin{equation}\label{case2_est_02_U}
\forall t \in [\tilde{t}^+,\tilde{t}_2^+],  \hspace{2mm} \CE(U(t)) \geq \kappa_{r_0^+/4}^+.
\end{equation}
Inequalities \eqref{case2_est_01_U}, \eqref{case2_est_02_U} yield by Young's inequality
\begin{equation}
\int_{\tilde{t}_2^+}^{\tilde{t}^+}\be_c(U(t)) dt \geq \frac{r_0^+}{4}e^{ct^++1}e^{-1}\sqrt{2 \kappa_{r_0^+/4}^+}=(\eta_0^+)^2 e^{t^++1},
\end{equation}
where the last equality is due to the definition of $\eta_0^+$ in \eqref{eta_0+}. Combining with \eqref{comparison_positivity}, we get
\begin{equation}
\int_{t^+}^{+\infty} \be_c(U(t)) dt \geq  (\eta_0^+)^2 e^{t^++1}.
\end{equation}
The definition of $\CE_{\max}^+$ in \eqref{eps_0+} and \eqref{case2_est_Utilde_def} imply then that
\begin{equation}
\int_{t^+}^{+\infty} \be_c(U(t)) dt > \int_{t^+}^{+\infty} \be_c(\tilde{U}^-(t)) dt,
\end{equation}
which establishes \eqref{comparison_energy_+}.

We now show the last part of the proof: we show that \eqref{sfT_prop} holds with the constant $T_\star(c)$ defined in \eqref{sfT_definition}. The argument is the same as in \cite{alikakos-fusco-smyrnelis}, \textit{Lemma 2.10}. Assume by contradiction that there exists $t \in (t^-,+\infty)$ such that $\CE(U(t))<0$. Then, arguing as above we must have $t<t^-$ by the definition of $t^-$ in \eqref{comparison_t-}, a contradiction. Therefore, we can write
\begin{align}\label{eq1_t-t+}
\bE_c(U)&= \frac{1}{2}\int_\R \lVert U'(t) \rVert_{\scrL}^2e^{ct}dt-\int_{-\infty}^{t^-}\CE^-(U(t))e^{ct}dt\\ &+ \int_{\R} \CE^+(U(t))e^{ct}dt,
\end{align}
where $\CE^-$ and $\CE^+$ stand for the positive and the negative part of $\CE$, respectively. We have that
\begin{equation}\label{ineq1_t-t+}
\int_{-\infty}^{t^-} \CE^-(U(t))e^{ct}dt \leq \frac{-a}{c}e^{ct^-}.
\end{equation}
Set $\alpha_\star:= \min\{ \CE_{\max}^+,\CE_{\max}^-+a\}>0$, which is independent on $U$, $c$ and $T$. Notice that for all $t \in (t^-,t^+)$ we have that $\CE(U(t)) \geq \alpha_\star$. Indeed, if $\CE(U(t)) < \alpha_\star$, then by the definition of $t^-$ and $t^+$ in \eqref{comparison_t-} and \eqref{comparison_t+} we get
\begin{equation}
\dist_{\scrL}(U(t),\scrF^\pm) \geq \frac{r_0^\pm}{2},
\end{equation}
which implies that
\begin{equation}
\CE(U(t)) \geq \min\{ \kappa^+_{r_0^+/2},\kappa^-_{r_0^-/2}+a\} \geq \alpha_\star,
\end{equation}
by \eqref{eps_0-} and \eqref{eps_0+}, a contradiction. Therefore, 
\begin{equation}\label{ineq2_t-t+}
\int_{\R}\CE^+(U(t))e^{ct}dt \geq \int_{t^-}^{t^+}\CE^+(U(t))e^{ct}dt \geq \frac{\alpha_\star}{c}\left(e^{ct^+}-e^{ct^-} \right).
\end{equation}
Plugging \eqref{ineq1_t-t+} and \eqref{ineq2_t-t+} into \eqref{eq1_t-t+} and using that $\bE_c(U) \leq 0$, we obtain
\begin{equation}
0 \geq \frac{a}{c}e^{ct^-}+\frac{\alpha_\star}{c}\left(e^{ct^+}-e^{ct^-} \right) \geq \left( \frac{a}{c}+\frac{\alpha_\star}{c}\left(e^{c(t^+-t^-)}-1\right)\right)e^{ct^-},
\end{equation}
that is,
\begin{equation}
0 \geq  -\left(\frac{-a}{\alpha_\star}+1\right)+e^{c(t^+-t^-)},
\end{equation}
which implies
\begin{equation}
0 < t^+-t^- \leq \frac{1}{c}\ln\left( \frac{-a}{\alpha_\star}+1 \right)=\sfT_\star(c),
\end{equation}
which is exactly \eqref{sfT_prop} according to the definition \eqref{sfT_definition}.
\end{proof}

\begin{figure}
\centering
\includegraphics[scale=0.3]{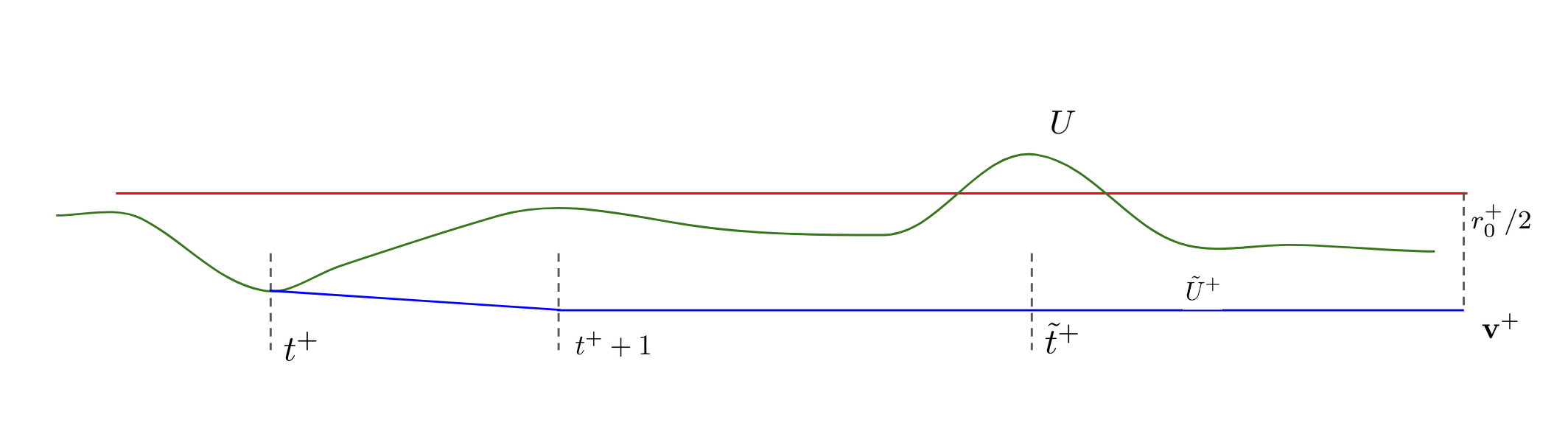}
\includegraphics[scale=0.3]{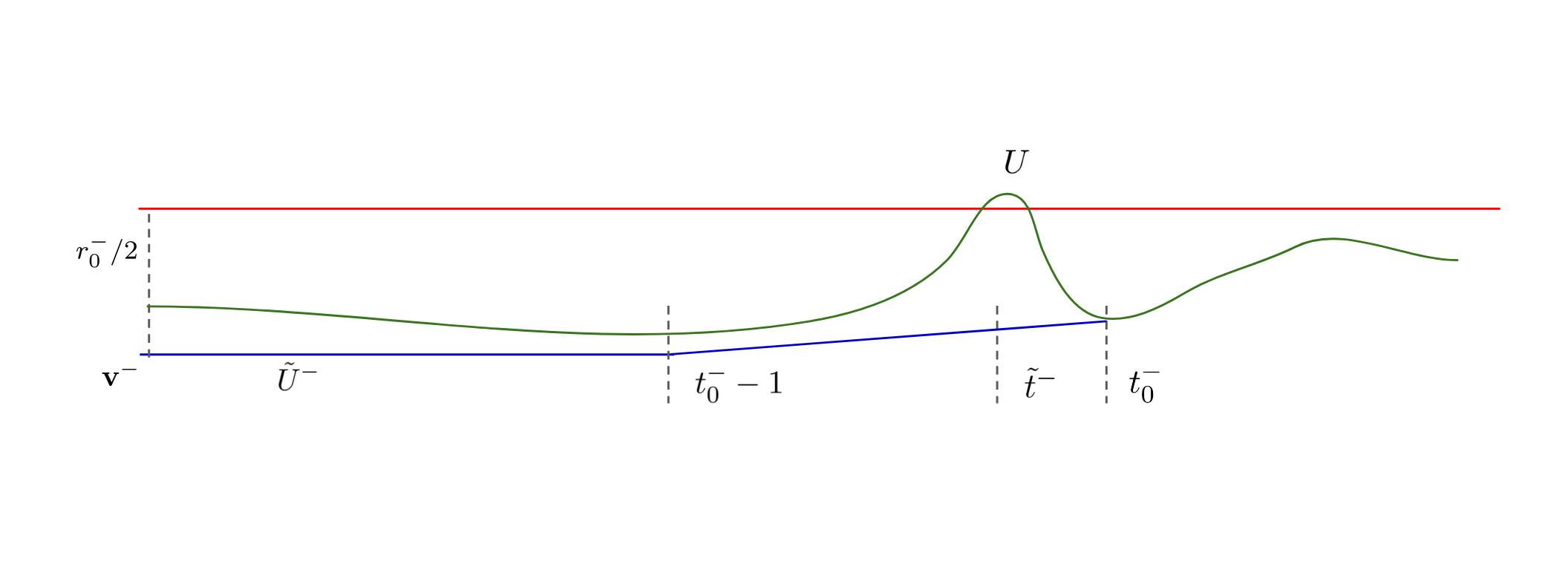}
\includegraphics[scale=0.3]{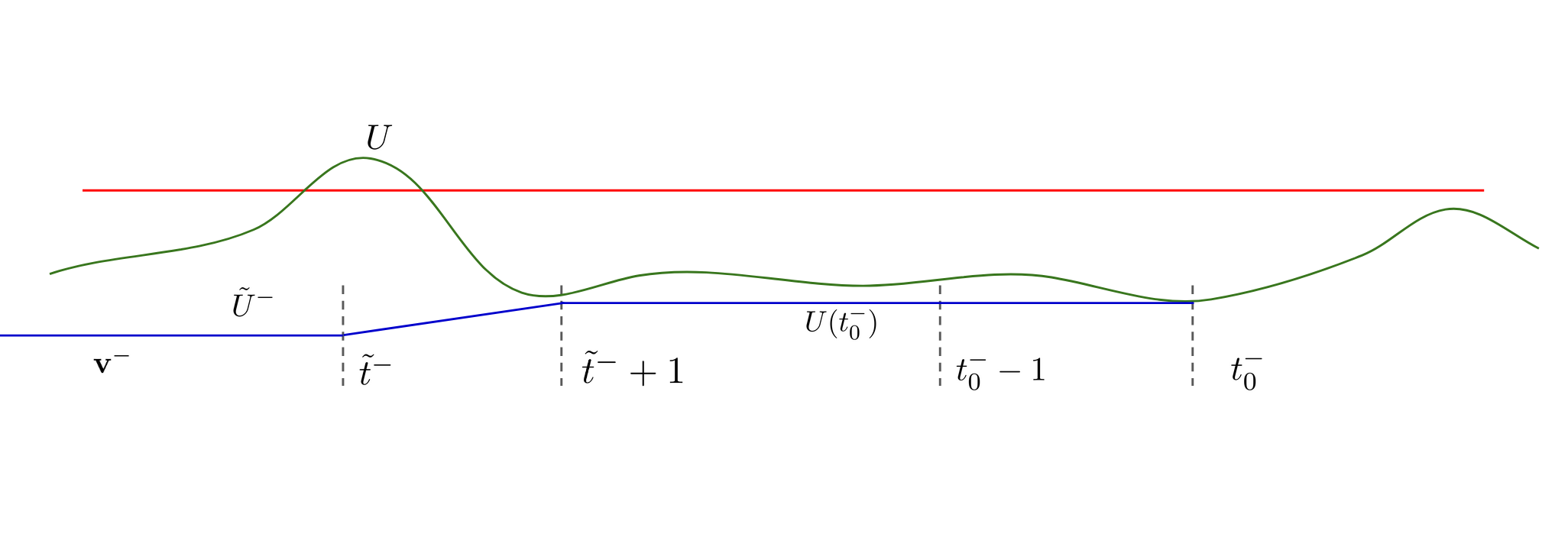}
\caption{As it has been shown, the proof of Proposition \ref{PROPOSITION_comparison} consists on showing that if the function $U$ gets too far from $\scrF^\pm$ after getting too close, then we can find a suitable competitor with strictly less energy. Above, we see a design for the positive case (the competitor $\tilde{U}^+$ is represented in blue). The second and third picture correspond to the two possible scenarios for the negative case (the competitor $\tilde{U}^-$ is represented in blue).}\label{Figure_proof}
\end{figure}

The importance of Proposition \ref{PROPOSITION_comparison} is summarized by the following result, which gives important information on the behavior of the constrained minimizers:
\begin{corollary}\label{COROLLARY_constrained_behavior}
Assume that \ref{hyp_compactness}, \ref{hyp_projections_inverse}, \ref{hyp_projections_H} and \ref{hyp_levelsets} hold. Let $c>0$ and $T \geq 1$. Let $\bU_{c,T}$ be an associated minimizer of $\bE_c$ in $X_T$ given by Lemma \ref{LEMMA_mcT}. Then, if $t^\pm:= t^\pm(\bU_{c,T},\CE_{\max}^\pm)$ are as in \eqref{comparison_t-}, \eqref{comparison_t+} it holds that
\begin{equation}\label{behavior-}
\forall t \leq t^-, \hspace{2mm} \dist_{\scrL}(\bU_{c,T}(t),\scrF^-) < \frac{r_0^-}{2}
\end{equation}
and
\begin{equation}\label{behavior+}
\forall t \geq t^+, \hspace{2mm} \dist_{\scrL}(\bU_{c,T}(t),\scrF^+) < \frac{r_0^+}{2},
\end{equation}
Moreover, we have
\begin{equation}\label{behavior_positivity}
\forall t \geq t^-, \hspace{2mm} \CE(\bU_{c,T}(t)) \geq 0.
\end{equation}
Finally, we have that if $\bE_c(\bU_{c,T})\leq 0$, then
\begin{equation}\label{sfT}
0<t^+-t^- \leq \sfT_\star(c),
\end{equation}
where $\sfT_\star(c)$ is as in \eqref{sfT_definition}. In particular, the function
\begin{equation}
c \in (0,+\infty) \to \sfT_\star(c)
\end{equation}
is continuous.
\end{corollary}
\begin{proof}
If we assume by contradiction that \eqref{behavior+} does not hold, then we necessarily have that there exists $\tilde{t}^- < t^-$ such that \eqref{bad_-} holds. Proposition \ref{PROPOSITION_comparison} implies then the existence of $\tilde{U} \in X_T$ such that $\bE_c(\tilde{U})< \bE_c(\bU_{c,T})=\bm_{c,T}$, a contradiction. Therefore, \eqref{behavior+} holds. Similarly, we can show that \eqref{behavior-} also holds. Finally, in order to establish \eqref{behavior_positivity}, we argue as in the proof Proposition \ref{PROPOSITION_comparison}. Indeed, due to the definition of $t^-$, we have that for $t \geq t^-$ it holds that either
\begin{equation}
\CE(\bU_{c,T}(t)) \geq a+\CE_{\max}^->0
\end{equation}
(which is \ref{hyp_levelsets}) or
\begin{equation}
\dist_{\scrL}(\bU_{c,T}(t),\scrF^-) \geq \frac{r_0^-}{2}
\end{equation}
which by \eqref{negative_levelset} in \ref{hyp_levelsets} implies that $\CE(\bU_{c,T}(t)) \geq 0$. Therefore, \eqref{behavior_positivity} holds and the proof is concluded.
\end{proof}

Moreover, Lemma \ref{LEMMA_mcT_solution} applies to $\bV_{c,T}$ as follows:
\begin{corollary}\label{COROLLARY_mcT_solution}
Assume that \ref{hyp_compactness}, \ref{hyp_projections_inverse}, \ref{hyp_projections_H} and \ref{hyp_levelsets} hold. Let $c>0$ and $T \geq 1$. Let $\bU_{c,T}$ be an associated minimizer of $\bE_c$ in $X_T$ given by Lemma \ref{LEMMA_mcT}. Then, if $t^\pm:= t^\pm(\bU_{c,T},\CE_{\max}^\pm)$ are as in \eqref{comparison_t-}, \eqref{comparison_t+} it holds that there exists $\delta_{c,T}>0$ such that the set
\begin{equation}\label{ScT}
S_{c,T}:=(-\infty,t^-+\delta_{c,T}) \cup (-T,T) \cup (t^+-\delta_{c,T},+\infty)
\end{equation}
is such that $\bU_{c,T} \in \CA(S_{c,T})$ (see \eqref{CA_def}) and
\begin{equation}\label{corollary_mcT_solution_equation}
\bU_{c,T}'' - D_{\scrL}\CE(\bU_{c,T})=-c\bU_{c,T}' \mbox{ in } S_{c,T}.
\end{equation}
\end{corollary}
The proof of \ref{COROLLARY_mcT_solution} is obtained in a straightforward manner by combining Lemma \ref{LEMMA_mcT_solution} with the informations given by Corollary \ref{COROLLARY_constrained_behavior}. Notice that \ref{COROLLARY_mcT_solution} implies that constrained solutions are picewise solutions and, in particular, they solve the equation for times with large absolute value.

\subsection{Existence of the unconstrained solutions}

We now establish the existence of the unconstrained solutions making use of the previous comparison results. As in \cite{alikakos-fusco-smyrnelis} and  \cite{alikakos-katzourakis}, we define the set
\begin{equation}\label{set_CC}
\CC:= \{ c>0: \exists T \geq 1 \mbox{ and } U \in X_T \mbox{ such that } \bE_c(U)<0\}.
\end{equation}
We first prove some important properties for $\CC$ which are the same  to \textit{Lemma 2.12} in \cite{alikakos-fusco-smyrnelis} and \textit{Lemma 27} in \cite{alikakos-katzourakis}:
\begin{lemma}\label{LEMMA_set_CC}
Assume that \ref{hyp_compactness}, \ref{hyp_projections_inverse} and \ref{hyp_projections_H} hold. Let $\CC$ be the set defined in \eqref{set_CC}.  Then, $\CC$ is open and non-empty. Moreover, if we assume that \ref{hyp_levelsets} holds, then $\CC$ is also bounded with
\begin{equation}\label{sup_C_bound}
\sup \CC \leq \frac{\sqrt{-2a}}{d_0},
\end{equation}
where $d_0$ was defined in \eqref{d0}.
\end{lemma}

\begin{proof}
Firstly, we show that $\CC \not = \emptyset$. For that purpose, consider the function $\Psi$ introduced in \eqref{Psi}. Consider the function
\begin{equation}
f: c \in (0,+\infty) \to e^{-c}\left(\frac{a}{c}+e^{2c}\int_{-1}^1\left( \frac{\lVert \Psi'(t) \rVert^2_{\scrL}}{2}+\CE(\Psi(t)) \right)dt \right) \in \R,
\end{equation}
which is well defined by Lemma \ref{LEMMA_PSI}. We have that for all $c>0$
\begin{equation}\label{energy_Psi_ineq}
\bE_c(\Psi)= \frac{-a}{c}e^{-c}+\int_{-1}^1\left(\frac{\lVert \Psi'(t) \rVert^2_{\scrL}}{2}+\CE(\Psi(t)) \right)e^{ct}dt \leq f(c)
\end{equation}
and $f$ is a continuous function such that $\lim_{c \to 0}f(c)=-\infty$ because $a<0$. Moreover, we have that for all $c>0$,
\begin{equation}
f'(c) = -e^{-c}a+ce^{2c}\int_{-1}^1\left( \frac{\lVert \Psi'(t) \rVert^2_{\scrL}}{2}+\CE(\Psi(t)) \right)dt  > 0
\end{equation}
and $\lim_{c \to +\infty}f(c)=+\infty$. Therefore, there exists a unique $c_\Psi >0$ such that $f(c_\Psi)=0$ and for all $c < c_\Psi$ we have $\bE_c(\Psi)<0$ by \eqref{energy_Psi_ineq}. Therefore, $(0,c_\Psi) \subset \CC$, meaning that $\CC \not = \emptyset$ as we wanted to show.

We next prove that $\CC$ is open. Let $c \in \CC$, we have that $\bE_c(\bU_{c,T})<0$, where $\bU_{c,T}$ is a minimizer of $\bE_c$ in $X_T$ given by Lemma \ref{LEMMA_mcT}. By \eqref{limit+_energy0_weak} in Lemma \ref{LEMMA_limit+}, there exists a sequence $(t_n)_{n \in \N}$ in $[T,+\infty)$ such that $t_n \to +\infty$ and
\begin{equation}\label{CC_lim}
\lim_{n \to \infty}\CE(\bU_{c,T}(t_n))=0.
\end{equation}
Up to subsequences, we have that for all $n \in \N$, $\bU_{c,T}(t_n) \in \scrF^+_{r_0^+}$. Hence, we can define
\begin{equation}
\bU_{c,T}^n(s):= \begin{cases}
\bU_{c,T}(s) &\mbox{ if } s \leq t,\\
(1+t_n-s)\bU_{c,T}(t_n)+(s-t_n)P^+(\bU_{c,T}(t_n))&\mbox{ if } t_n \leq s \leq t_n+1,\\
P^+(\bU_{c,T}(t_n)) &\mbox{ if } t_n+1 \leq s.
\end{cases}
\end{equation}
We have that for all $n \in \N$,
\begin{align}\label{ineq_Uctn}
\bE_c(\bU^n_{c,T}(s))&=\int_{-\infty}^{t_n} \be_c(\bU_{c,T}(s))ds+\frac{\lVert \bU_{c,T}(t_n)-P^+(\bU_{c,T}(t_n)) \rVert_{\scrL}^2 }{2}\\&+\int_{t_n}^{t_n+1} \CE(\bU^n_{c,T}(s))ds \\
& \leq \bE_c(\bU_{c,T})+\frac{\lVert \bU_{c,T}(t_n)-P^+(\bU_{c,T}(t_n)) \rVert_{\scrL}^2 }{2}\\ &+\int_{t_n}^{t_n+1}\CE((1+t_n-s)\bU_{c,T}(t_n)+(s-t_n)P^+(\bU_{c,T}(t_n)))ds,
\end{align}
where we have used that $t_n \geq T$ in order to obtain the inequality. Let $\beta^+(\hat{r}^+)$ be as in Lemma \ref{LEMMA_r}. Up to a subsequence, we have that for all $n \in \N$ it holds $\CE(\bU_{c,T}(t_n)) \leq \beta^+(\hat{r}^+)$. Therefore, by Lemma \ref{LEMMA_r} we have that for all $\lambda \in [0,1]$ and $n \in \N$ it holds
\begin{equation}
\CE(\lambda \bU_{c,T}(t_n)+(1-\lambda)P^+(\bU_{c,T}(t_n))) \leq \sfC^+ \CE(\bU_{c,T}(t_n)),
\end{equation}
where $\sfC^+>0$ is independent on $n$ (see \eqref{sfC}). Plugging into \eqref{ineq_Uctn}, we obtain that for all $n \in \N$ it holds
\begin{equation}\label{ineq_Uctn_def}
\bE_c(\bU^n_{c,T}(s)) \leq \bE_c(\bU_{c,T})+\frac{\lVert \bU_{c,T}(t_n)-P^+(\bU_{c,T}(t_n)) \rVert_{\scrL}^2 }{2}+\sfC^+ \CE(\bU_{c,T}(t_n)).
\end{equation}
Recall that we assume  Notice also that \eqref{limit+_function0_weak} implies in particular that
\begin{equation}
\lim_{n \to +\infty}\lVert \bU_{c,T}(t_n)-P^+(\bU_{c,T}(t_n)) \rVert_{\scrL}^2  =0
\end{equation}
which, in combination with inequalities \eqref{CC_lim} and \eqref{ineq_Uctn_def} together with the fact that $\bE_c(\bU_{c,T})<0$, gives that there exists $N \in \N$ such that $\bE_c(\bU_{c,T}^N)<0$. Since $\bU_{c,T}^N$ is constant in $[t_N+1,+\infty)$, we have that for all $\tilde{c} >0$, $\bE_{\tilde{c}}(\bU_{c,T}^N)<+\infty$. Therefore, by Lemma \ref{LEMMA_ATU} we have that
\begin{equation}
\tilde{c} \in (0,+\infty) \to \bE_{\tilde{c}}(\bU_{c,T}^N) \in \R
\end{equation}
is well defined and continuous. Therefore, we can find some $\delta>0$ such that for all $\tilde{c} \in (c-\delta,c+\delta)$, $\bE_{\tilde{c}}(\bU_{c,T}^N)<0$. As a consequence, we have that $(c-\delta,c+\delta) \subset \CC$, which shows that $\CC$ is open.

We now assume that \ref{hyp_levelsets} holds and we use it to establish the bound \eqref{sup_C_bound}. In particular, we can apply Proposition \ref{PROPOSITION_comparison}. Let $c>0$ and $T \geq 1$ be such that $\bE_c(\bU_{c,T})<0$ with $\bU_{c,T} \in X_T$ a minimizing solution given by Lemma \ref{LEMMA_mcT}. Let $t^\pm:=t^\pm(\bU_{c,T},\CE_{\max}^\pm)$ be as in \eqref{comparison_t-}, \eqref{comparison_t+}. Inequality \eqref{sfT} in Proposition \ref{PROPOSITION_comparison} implies that $t^- < t^+$. Recall the definition of $d_0$ in \eqref{d0} and the fact that $\bU_{c,T}(t^\pm) \in \scrF^\pm_{r_0^\pm/2}$. Those facts imply
\begin{equation}\label{d0_ineq}
d_0 \leq \lVert \bU_{c,T}(t^+)-\bU_{c,T}(t^-) \rVert_{\scrL}.
\end{equation}
Since \ref{hyp_levelsets} holds, we can use \eqref{behavior_positivity} in Corollary \ref{COROLLARY_constrained_behavior} to obtain
\begin{align}
\lVert \bU_{c,T}(t^+)-\bU_{c,T}(t^-) \rVert_{\scrL}^2 &\leq 2\int_{\R} \frac{\lVert \bU_{c,T}'(t) \rVert^2}{2}e^{ct}dt \left( \frac{e^{-ct^-}-e^{-ct^+}}{c} \right) 
\\ & \leq 2\left( \bE_{c}(\bU_{c,T}) - \frac{a}{c	}e^{ct^-} \right)\left( \frac{e^{-ct^-}-e^{-ct^+}}{c} \right) .
\end{align}
Using now that $\bE_c(\bU_{c,T}) \leq 0$, the fact that $t^- < t^+$ and \eqref{d0_ineq}, the inequality above becomes
\begin{equation}
d_0^2 \leq -2a \frac{1-e^{c(t^--t^+)}}{c^2} \leq \frac{-2a}{c^2},
\end{equation}
so that \eqref{sup_C_bound} follows.
\end{proof}

We now have all the ingredients for establishing the existence of the unconstrained solutions:
\begin{proposition}\label{PROPOSITION_existence}
Assume that \ref{hyp_compactness}, \ref{hyp_projections_inverse}, \ref{hyp_projections_H}, \ref{hyp_regularity} and \ref{hyp_levelsets} hold.  Let $\overline{c} \in \partial(\CC) \cap (0,+\infty)$, where $\partial(\CC)$ stands for the boundary of the set $\CC$ defined in \eqref{set_CC}. Then, there exists $\overline{T} \geq 1$ such that $\bm_{\overline{c},\overline{T}}=0$ ($\bm_{\overline{c},\overline{T}}$ as in \eqref{mcT}) and $\overline{\bU} \in X_{\overline{T}}$ an associated minimizer of $\bE_{\overline{c}}$ in $X_{\overline{T}}$ which does not saturate the constraints, i. e.
\begin{equation}\label{unconstrained+}
\forall t \geq \overline{T}, \hspace{2mm} \dist_{\scrL}(\overline{\bU}(t),\scrF^+) < \frac{r_0^+}{2}
\end{equation}
and
\begin{equation}\label{unconstrained-}
\forall t \leq -\overline{T}, \hspace{2mm} \dist_{\scrL}(\overline{\bU}(t),\scrF^-) < \frac{r_0^-}{2}.
\end{equation}
Moreover, $\overline{\bU} \in \CA(\R)$ and the pair $(\overline{c},\overline{\bU})$ solves \eqref{abstract_equation}.
\end{proposition}
\begin{remark}
Notice that Lemma \ref{LEMMA_set_CC} implies that (under the necessary assumptions) the set $\CC$ is bounded, meaning that $\partial(\CC) \cap (0,+\infty) \not = \emptyset$. Such a fact, in combination with Proposition \ref{PROPOSITION_existence} shows the existence of the unconstrained solutions.
\end{remark}

\begin{proof}[Proof of Proposition \ref{PROPOSITION_existence}]
By Lemma \ref{LEMMA_set_CC}, we have that $\CC \not = \emptyset$ is open, which implies that $\partial(\CC) \subset \R \setminus \CC$. Therefore, we have $\overline{c} \not \in\CC$. Recall that due to the definition of $\CC$ in \eqref{set_CC}, we have that
\begin{equation}\label{c^star_nonnegative}
\forall T \geq 1, \hspace{2mm} \bm_{\overline{c},T} \geq 0.
\end{equation}
The definition of the boundary allows to consider a sequence $(c_n)_{n \in \N}$ contained in $\CC$ such that $c_n \to \overline{c}$. Then, for each $n \in \N$ there exists $T_n \geq 1$ such that $\bE_{c_n}(\bU_{c_n,T_n})<0$, where, for each $n \in \N$, $\bU_{c_n,T_n}$ is a minimizer of $\bE_{c_n}$ in $X_{T_n}$. For each $n \in \N$, set $t_n^\pm:= t^+(\bU_{c_n,T_n},\CE_{\max}^\pm)$ as in \eqref{comparison_t-}, \eqref{comparison_t+}. Using \eqref{sfT} in Corollary \ref{COROLLARY_constrained_behavior} we have that that
\begin{equation}
\forall n \in \N, \hspace{2mm}0 < t^+_n-t^-_n \leq \sfT_\star(c_n),
\end{equation}
and the function
\begin{equation}
c \in (0,+\infty) \to \sfT_\star(c) \in (0,+\infty)
\end{equation}
is continuous. Since the sequence $(c_n)_{n \in \N}$ is bounded, we have that
\begin{equation}
\sfT_\star:= \max\left\{1,\sup_{n \in \N} \sfT_\star(c_n)\right\}<+\infty
\end{equation}
and
\begin{equation}\label{unconstrainted_sfT}
\forall n \in \N, \hspace{2mm}0 < t^+_n-t^-_n \leq \sfT_\star,
\end{equation}
so that we have a bound on $(t^+_n-t^-_n)_{n \in \N}$. Moreover, \eqref{behavior-} and \eqref{behavior+} in Corollary \ref{COROLLARY_constrained_behavior} imply
\begin{equation}\label{unconstrained_t+}
\forall n \in \N, \forall t \geq t^+_n, \hspace{2mm} \dist_{\scrL}(\bU_{c_n,T_n}(t),\scrF^+) < \frac{r_0^+}{2}
\end{equation}
and
\begin{equation}\label{unconstrained_t-}
\forall n \in \N, \forall t \leq t^-_n, \hspace{2mm} \dist_{\scrL}(\bU_{c_n,T_n}(t),\scrF^-) < \frac{r_0^-}{2}.
\end{equation}
For each $n \in \N$, define the function $\bU_{c_n,T_n}^{t_n^+}:=\bU_{c_n,T_n}(\cdot + t_n^+)$. Then, \eqref{unconstrainted_sfT} implies that \eqref{unconstrained_t+} and \eqref{unconstrained_t-} write as
\begin{equation}\label{translate_unconstrained_+}
\forall n \in \N, \forall t \geq 0, \hspace{2mm} \dist_{\scrL}(\bU_{c_n,T_n}^{t_n^+}(t),\scrF^-) < \frac{r_0^-}{2}.
\end{equation}
and
\begin{equation}\label{translate_unconstrained_t-}
\forall n \in \N, \forall t \leq -\sfT_\star, \hspace{2mm} \dist_{\scrL}(\bU_{c_n,T_n}(t),\scrF^-) < \frac{r_0^-}{2}.
\end{equation}
so that for all $n \in \N$ we have $\bU_{c_n,T_n}^{t_n^+} \in X_{\sfT_\star}$. Moreover, a computation shows
\begin{equation}
\forall n \in \N, \hspace{2mm} \bE_{c_n}(\bU_{c_n,T_n}^{t_n^+})=e^{-c_nt_n^+}\bE_{c_n}(\bU_{c_n,T_n})<0.
\end{equation}
Therefore, if we apply Lemma \ref{LEMMA_LSC} with sequence of speeds $(c_n)_{n \in \N}$ and the sequence $(\bU_{c_n,T_n}^{t_n^+})_{n \in \N}$ in $X_{\sfT_\star}$, we obtain $\overline{\bU} \in X_{\sfT_\star}$ such that
\begin{equation}
\bE_{\overline{c}}(\overline{\bU}) \leq \liminf_{n \to \infty } \bE_{c_n}(\bU_{c_n,T_n}^{t_n^+}) \leq 0,
\end{equation}
which in combination with \eqref{c^star_nonnegative} implies that $\bm_{\overline{c},\sfT_\star}=0$. Therefore, we have $\bE_{\overline{c}}(\overline{\bU})=0$, so that $\overline{\bU}$ is a minimizer of $\bE_{\overline{c}}$ in $X_{\sfT_\star}$. Set $t^\pm_{\star}:= t^\pm(\overline{\bU},\CE_{\max}^\pm)$ as in \eqref{comparison_t-}, \eqref{comparison_t+}. Invoking \eqref{c^star_nonnegative} and Corollary \ref{COROLLARY_mcT_solution}, we obtain that for all $T \geq 1$ such that $\bm_{\overline{c},T}=0$ and $\overline{\bU} \in X_T$, we have $\overline{\bU} \in \CA(S_{\overline{c},T})$ with
\begin{equation}\label{Scstar}
S_{\overline{c},T}:= (-\infty,t^-_\star+\delta_\star(T)) \cup (-T,T) \cup (t^+_\star-\delta_\star(T),+\infty)
\end{equation}
for some $\delta_\star(T)>0$ and
\begin{equation}\label{Scstar_eq}
\overline{\bU}''-D_{\scrL}\CE(\overline{\bU})=-\overline{c} \overline{\bU}' \mbox{ in } S_{\overline{c},T}.	
\end{equation}
Moreover, using \eqref{behavior-} and \eqref{behavior+} in Corollary \ref{COROLLARY_constrained_behavior}, we obtain as before that
\begin{equation}\label{t_star+}
\forall t \geq t^+_\star, \hspace{2mm} \dist_{\scrL}(\overline{\bU}(t),\scrF^+) < \frac{r_0^+}{2}
\end{equation} 
and
\begin{equation}\label{t_star-}
\forall t \leq t^-_\star, \hspace{2mm} \dist_{\scrL}(\overline{\bU}(t),\scrF^-) < \frac{r_0^-}{2}.
\end{equation}
Therefore, if we set $\overline{T}=\max\{1,t^+_\star,-t^-_\star\}$, then \eqref{t_star+} and \eqref{t_star-} imply that $\overline{\bU} \in X_{\overline{T}}$ and that \eqref{unconstrained+}, \eqref{unconstrained-} hold. Moreover, we have that $\bE_{\overline{c}}(\overline{\bU})=0$, so that $\overline{\bU}$ is a minimizer of $\bE_{\overline{c}}$ in $X_{\overline{T}}$ by \eqref{c^star_nonnegative}. Therefore, we obtain that $\overline{\bU} \in \CA(S_{\overline{c},\overline{T}})$ and
\begin{equation}
\overline{\bU}''-D_{\scrL}\CE(\overline{\bU})=-\overline{c} \overline{\bU}' \mbox{ in } S_{\overline{c},\overline{T}},
\end{equation}
with $S_{\overline{c},\overline{T}}$ as in \eqref{Scstar}. The choice of $\overline{T}$ implies that $S_{\overline{c},\overline{T}}=\R$. Therefore, $\overline{\bU} \in \CA(\R)$ and $(\overline{c},\overline{\bU})$ solves \eqref{abstract_equation}, which finishes the proof.
\end{proof}
Notice that our Proposition \ref{PROPOSITION_existence} follows very similar lines than the analogous results in \cite{alikakos-fusco-smyrnelis} and \cite{alikakos-katzourakis}. 

\subsection{Uniqueness of the speed}

The precise statement of the uniqueness result is as follows:
\begin{proposition}\label{PROPOSITION-uniqueness}
Assume that \ref{hyp_regularity} holds. Let $X$ be the set defined in \eqref{X}. Let $(c_1,c_2) \in (0,+\infty)^2$ be such that there exist $\bU_1$ and $\bU_2$ in $X\cap \CA(\R)$ such that $(c_1,\bU_1)$ and $(c_2,\bU_2)$ solve \eqref{abstract_equation} and for each $i \in \{1,2\}$, $\bE_{c_i}(\bU_i) <+\infty$. Assume moreover that
\begin{equation}\label{uniqueness_condition}
\forall i \in \{1,2\},\forall j \in \{1,2\} \setminus \{i\}, \hspace{2mm} \bE_{c_i}(\bU_j) \geq 0.
\end{equation}
 Then, we have $c_1=c_2$.
\end{proposition}
\begin{proof}
We prove the result by contradiction. Hence, we can assume without loss of generality that $c_1<c_2$. A direct computation shows that for every $(c,U) \in (0,+\infty) \times (X \cap \CA(\R))$ a solution to \eqref{abstract_equation}, we have
\begin{equation}\label{identity_uniqueness}
\forall t \in \R, \hspace{2mm} \frac{\lVert U'(t) \rVert_{\scrL}^2}{2}+\CE(U(t))= e^{-ct}\left( \frac{e^{ct}}{c}\left( \CE(U(t))-\frac{\lVert U'(t) \rVert_{\scrL}^2}{2} \right)\right)'.
\end{equation}
Replacing $(c_2,U_2)$ in \eqref{identity_uniqueness} and multiplying for each $t \in \R$ by $e^{c_1t}$, computations show that
\begin{align}\label{identity_uniqueness_2}
\forall t_1 < t_2, \hspace{2mm} c_1 \bE_{c_1}(\bU_2;(t_1,t_2))&=(c_1-c_2)\int_{t_1}^{t_2}\lVert \bU_2'(t) \rVert_{\scrL}^2 e^{c_1t}dt \\ &+\left[e^{c_1t}\left( \CE(\bU_2(t))-\frac{\lVert \bU_2'(t) \rVert_{\scrL}^2}{2} \right)\right]^{t_2}_{t_1}.
\end{align}
Notice now that the definition of $X$ in \eqref{X} implies that
\begin{equation}
X=\bigcup_{T \geq 1}X_T,
\end{equation}
which means that there exists $T \geq 1$ such that $\bU_2 \in X_T$. Combining then Lemma \ref{LEMMA_well_defined} and the fact that $\bE_{c_2}(\bU_2)<+\infty$, we get that $\be_{c_2}(\bU_2(\cdot)) \in L^1(\R)$. Therefore, we can find two sequences $(t_n^+)_{n  \in \N}$ and $(t_n^-)_{n \in \N}$ such that $t_n^\pm \to \pm\infty$ and
\begin{equation}\label{lim_uniqueness}
\lim_{n \to \infty} \be_{c_2}(\bU_2(t_n^\pm))=0.
\end{equation}
Since we have $c_1<c_2$, it holds
\begin{equation}
\forall t \in \R, \hspace{2mm} e^{c_1t}\left\lvert \CE(\bU_2(t))-\frac{\lVert \bU_2'(t) \rVert_{\scrL}^2}{2} \right\rvert \leq \lvert \be_{c_2}(\bU_2(t)) \rvert
\end{equation}
which in combination with \eqref{lim_uniqueness} implies
\begin{equation}\label{lim_uniqueness_def}
\lim_{n \to \infty} e^{c_1t_n^\pm}\left( \CE(\bU_2(t_n^\pm))-\frac{\lVert \bU_2'(t_n^\pm) \rVert_{\scrL}^2}{2} \right)=0.
\end{equation}
Therefore, if we replace for each $n \in \N$ in \eqref{identity_uniqueness_2} with $a=t_n^-$ and $b=t_n^+$, we can then pass to the limit \eqref{lim_uniqueness_def} and obtain
\begin{equation}
c_1\bE_{c_1}(\bU_2) =(c_1-c_2) \int_{\R}\lVert \bU_2'(t) \rVert_{\scrL}^2e^{c_1t}dt <0
\end{equation}
because we assume $c_1 < c_2$. However, by \eqref{uniqueness_condition} we have $\bE_{c_1}(\bU_2)  \geq 0$, which is a contradiction.
\end{proof}

\begin{remark}
Again, the proof of Proposition \ref{PROPOSITION-uniqueness} is essentially a direct adaptation of that given in \cite{alikakos-fusco-smyrnelis} and \cite{alikakos-katzourakis}. Our hypothesis are slightly weaker, since we only assume that the solutions have finite energies and \eqref{uniqueness_condition}, while in \cite{alikakos-fusco-smyrnelis} and \cite{alikakos-katzourakis} it is assumed that the solutions are global minimizers of the corresponding energy functionals. Notice also that \ref{hyp_levelsets} is not needed for proving Proposition \ref{PROPOSITION-uniqueness}, which holds in a more general setting.
\end{remark}

Proposition \ref{PROPOSITION-uniqueness} along with Proposition \ref{PROPOSITION_existence} allows to show that the set $\CC$  defined in \eqref{set_CC} is in fact an open interval:
\begin{corollary}\label{COROLLARY_CC}
Assume that \ref{hyp_compactness}, \ref{hyp_projections_inverse}, \ref{hyp_projections_H}, \ref{hyp_regularity} and \ref{hyp_levelsets} hold. Let
\begin{equation}\label{c(CC)}
c(\CC):= \sup \CC.
\end{equation}
Then, we have $\CC=(0,c(\CC))$. 
\end{corollary}

\begin{proof}
The statement of the result is equivalent to show that 
\begin{equation}
\partial(\CC)\cap(0,+\infty)=\{ c(\CC)\}
\end{equation}
The quantity $c(\CC)$ is well defined in $(0,+\infty)$ because $\CC$ is non-empty and bounded by Lemma \ref{LEMMA_set_CC}. Therefore, we have $c(\CC) \in \partial(\CC) \cap (0,+\infty)$ because $\CC$ is open, so it does not contain its limit points. By Proposition \ref{PROPOSITION_existence}, we find $\bU^\CC \in X$ such that $(c(\CC),\bU^\CC)$ solves \eqref{abstract_equation}. Let now $\overline{c}\in \partial(\CC) \cap (0,+\infty)$. If we show that $\overline{c}=c(\CC)$, the proof will be finished. Applying Proposition \ref{PROPOSITION_existence} with $\overline{c}$, we find $\overline{\bU} \in X$ such that $(\overline{c},\overline{\bU})$ solves \eqref{abstract_equation}. Proposition \ref{PROPOSITION_existence}, along with the fact that $\overline{c}$ and $c(\CC)$ do not belong to $\CC$, also implies that
\begin{equation}
\inf_{U \in X}\bE_{\overline{c}}(U)=\bE_{\overline{c}}(\overline{\bU})=0=\bE_{c(\CC)}(\bU^\CC)=\inf_{U \in X}\bE_{c(\CC)}(U)
\end{equation}
so that
\begin{equation}
\bE_{c(\CC)}(\overline{\bU}) \geq 0 \mbox{ and } \bE_{\overline{c}}(\bU^\CC) \geq 0
\end{equation}
meaning that we can apply Proposition \ref{PROPOSITION-uniqueness} to $(c(\CC),\bU^\CC)$, $(\overline{c},\overline{\bU})$. As a consequence, we have $\overline{c}=c(\CC)$, which concludes the proof.
\end{proof}
\subsection{Proof of Theorem \ref{THEOREM-ABSTRACT} completed}

All the elements of the proof of Theorem \ref{THEOREM-ABSTRACT} are already present in previous result. Indeed, Proposition \ref{PROPOSITION_existence} along with Corollary \ref{COROLLARY_CC} implies the existence of $(c^\star,\bU) \in (0,+\infty) \times X_{T^\star}$ a solution to \eqref{abstract_equation} with $c^\star=c(\CC)$. Conditions \eqref{abstract_bc_weak-} and \eqref{abstract_bc_weak+} are satisfied due to the fact that $\bU \in X_{T^\star}$. The statement regarding the uniqueness of the speed $c^\star$ follows from Proposition \ref{PROPOSITION-uniqueness}. Finally, we have that \eqref{limit+_function0_weak} is exactly the exponential rate of convergence \eqref{exponential_abstract}, which completes the proof.

\qed 

\subsection{Asymptotic behavior of the constrained solutions at $-\infty$}\label{subs_asymptotic_abstract}
As it has been pointed out before, almost nothing can be said about the behavior of arbitrary function in $X_T$ at $-\infty$. However, it turns out that constrained minimizers converge exponentially at $-\infty$ with respect to the $\scrL$-norm provided that the speed fulfills an explicit upper bound, see Proposition \ref{PROPOSITION_conv_sol_-}. Such an upper bound also allows to establish some other properties. Once Proposition \ref{PROPOSITION_conv_sol_-} will have been established, we will be able to complete the proofs of Theorems \ref{THEOREM_abstract_bc} and \ref{THEOREM_abstract_speed}. The results of this subsection are obtained by combining ideas from Smyrnelis \cite{smyrnelis}, Alikakos and Katzourakis \cite{alikakos-katzourakis}, Alikakos, Fusco and Smyrnelis \cite{alikakos-fusco-smyrnelis}. It is worth to point out that the arguments we present here strongly rely on the fact that the solutions considered are minimizers and that we do not expect them to hold for more general critical points.

We begin by showing the following preliminary result, which follows by a direct computation:
\begin{lemma}\label{LEMMA_equipartition}
Assume that \ref{hyp_regularity} holds. Let $c>0$, $t_1 < t_2$ and $U \in \CA((t_1,t_2))$ such that
\begin{equation}\label{equation_equipartition}
U''-D_{\scrL}\CE(U)=-cU' \mbox{ in }(t_1,t_2).
\end{equation}
Then, we have the formula
\begin{equation}\label{equipartition_identity}
\forall t \in (t_1,t_2), \hspace{2mm} \frac{d}{dt}\left( \CE(U(t))-\frac{\lVert U'(t) \rVert_{\scrL}^2}{2} \right) = c\lVert U'(t) \rVert^2_{\scrL}.
\end{equation}
\end{lemma}
Lemma \ref{LEMMA_equipartition} gives the following pointwise bounds for constrained solutions:
\begin{lemma}\label{LEMMA_equipartition_pointwise}
Assume that \ref{hyp_compactness}, \ref{hyp_projections_inverse}, \ref{hyp_projections_H}, \ref{hyp_regularity} and \ref{hyp_levelsets} hold. Let $\bU_{c,T}$ be a constrained solution given by Lemma \ref{LEMMA_mcT} and $t^-:= t^-(\bU_{c,T},\CE_{\max}^-)$ be as in \eqref{comparison_t-}. Then for all $t<t^-$ we have the inequality
\begin{equation}\label{equipartition-_inequality}
\frac{\lVert \bU_{c,T}'(t) \rVert_{\scrL}^2}{2} \leq \CE(\bU_{c,T}(t))-a.
\end{equation}
Similarly, it holds that for all $t>t^+$
\begin{equation}\label{equipartition+_inequality}
\CE(\bU_{c,T}(t)) \leq \frac{\lVert \bU_{c,T}'(t) \rVert_{\scrL}^2}{2},
\end{equation}
where $t^+:=t^+(\bU_{c,T},\CE_{\max}^+)$ is as in \eqref{comparison_t+}.
\end{lemma}
\begin{proof}
Notice that \eqref{corollary_mcT_solution_equation} in Corollary \ref{COROLLARY_mcT_solution} implies that $\bU_{c,T}$ solves
\begin{equation}
\bU_{c,T}''-D_\scrL\CE(\bU_{c,T})=-c\bU_{c,T}' \mbox{ in } (-\infty,t^-).
\end{equation}
Therefore, the function
\begin{equation}
f_{c,T}: t \in (-\infty,t^-] \to e^{ct}\left( \CE(\bU_{c,T}(t))-a-\frac{\lVert \bU_{c,T}'(t) \rVert_{\scrL}^2}{2}  \right),
\end{equation}
is $C^1$ and we clearly have that $f_{c,T} \in L^1((-\infty,t^-])$. By \eqref{equipartition_identity} in Lemma \ref{LEMMA_equipartition}, we have
\begin{equation}\label{ineq_fcT}
\forall t \in (-\infty,t^-), \hspace{2mm} f_{c,T}'(t) = cf_{c,T}(t)+ce^{ct}\lVert \bU'_{c,T}(t) \rVert_{\scrL}^2 \geq 0,
\end{equation}
and we also have $f_{c,T}' \in L^1((-\infty,t^-))$. Therefore, it holds that
\begin{equation}\label{lim_fcT}
\lim_{t \to -\infty} f_{c,T}(t)=0.
\end{equation}
Fix $t_1 < t_2 \leq t^-$. Integrating \eqref{ineq_fcT} in $[t_1,t_2]$ we get
\begin{equation}
f_{c,T}(t_2) \geq f_{c,T}(t_1),
\end{equation}
which in combination with \eqref{lim_fcT} gives
\begin{equation}
\forall t < t^-, \hspace{2mm} f_{c,T}(t) \geq 0,
\end{equation}
which is \eqref{equipartition-_inequality}. Inequality \eqref{equipartition+_inequality} is obtained in an identical fashion.
\end{proof}

We conclude this subsection by proving the exponential convergence result, which is inspired by the ideas in \textit{Proof of (28)} in Smyrnelis \cite{smyrnelis}.
\begin{proposition}\label{PROPOSITION_conv_sol_-}
Assume that \ref{hyp_compactness}, \ref{hyp_projections_inverse}, \ref{hyp_projections_H}, \ref{hyp_regularity} and \ref{hyp_levelsets} hold. Let $c>0$ and $T \geq 1$. Assume moreover that $c<\gamma^-$, where $\gamma^-$ is defined in \eqref{gamma-}. Let $\bU_{c,T}$ be a constrained solution given by Lemma \ref{LEMMA_mcT}.  Then, there exists $\overline{M}^->0$  such that for all $\eps \in (0,\gamma^--c)$ and $t \in \R$ it holds
\begin{equation}\label{L1_CE_conv-}
\int_{-\infty}^{t}(\CE(\bU_{c,T}(s ))-a)e^{-\eps s}ds \leq \overline{M}^- e^{(\gamma^--c-\eps)t}.
\end{equation}
Furthermore, there exist $M^->0$ and $\bv_{c,T}^-\in \scrF^-$ such that for all $t \in \R$
\begin{equation}\label{solutions_limit-_norm}
\lVert \bU_{c,T}(t)-\bv_{c,T}^- \rVert_{\scrL}^2 \leq M^- e^{(\gamma^--c)t}.
\end{equation}
\end{proposition}
\begin{proof}
Let $t^-:= t^-(\bU_{c,T},\CE_{\max}^-)$ be as in \eqref{comparison_t-}. By applying \eqref{behavior-} in Corollary \ref{COROLLARY_constrained_behavior}, we obtain that for all $t \leq t^-$, $\bU_{c,T}(t) \in \scrF^-_{r_0^-/2}$. For all $t \leq t^-$, define $\bv^-(t):= P^-(\bU_{c,T}(t))$. Consider the function
\begin{equation}
\tilde{U}^-_t(s):=\begin{cases}
\bv^-(t) &\mbox{ if } s \leq t-1, \\
(t-s)\bv^-(t)+(s-t+1)\bU_{c,T}(t) &\mbox{ if } t-1 \leq s \leq t, \\
\bU_{c,T}(s) &\mbox{ if } t \leq s,
\end{cases}
\end{equation}
which belongs to $X_T$. Therefore,
\begin{equation}
\bE_c(\bU_{c,T}) \leq \bE_c(\tilde{U}_t^-)
\end{equation}
and, equivalently
\begin{equation}\label{ineq_first_limit-}
\int_{-\infty}^t \be_c(\bU_{c,T}(s)) ds \leq \frac{a}{c}e^{ct}+\int_{t-1}^t \left( \frac{\lVert \bU_{c,T}(t)-\bv^-(t) \rVert_{\scrL}^2}{2}+(\CE(\tilde{U}^-_t(s))-a) \right)e^{cs}ds.
\end{equation}
Using Lemma \ref{LEMMA_r} and \eqref{coercivityL} in \ref{hyp_unbalanced}, \eqref{ineq_first_limit-} becomes
\begin{equation}
\int_{-\infty}^t \be_c(\bU_{c,T}(s)) ds \leq \frac{a}{c}e^{ct}+(C^-+\sfC^-) (\CE(\bU_{c,T}(t)-a)e^{ct},
\end{equation}
which gives
\begin{equation}\label{ineq_second_limit-}
\int_{-\infty}^t (\CE(\bU_{c,T}(s)-a)e^{cs}ds \leq \frac{1}{\gamma^-}(\CE(\bU_{c,T}(t)-a)e^{ct},
\end{equation}
where $\gamma^-$ was defined in \eqref{gamma-}. Define the function
\begin{equation}\label{theta-_def}
\theta^-_{c,T}: t \in (-\infty,t^-] \to \int_{-\infty}^t (\CE(\bU_{c,T}(s)-a)e^{cs}ds	\in \R.
\end{equation}
By \ref{hyp_regularity}, the function $\theta_{c,T}^-$ defined in \eqref{theta-_def} verifies that for all $t \in (-\infty,t^-)$
\begin{equation}
(\theta^-_{c,T})'(t)= (\CE(\bU_{c,T}(t))-a)e^{ct}
\end{equation}
which, by \eqref{ineq_second_limit-} implies
\begin{equation}
\forall t \leq t^-, \hspace{2mm} \gamma^-\theta_{c,T}^-(t) \leq  (\theta^-_{c,T})'(t).
\end{equation}
Fix now $t \in (-\infty,t^-)$ and assume that $\theta^-_{c,T}(t)>0$. The previous inequality is equivalent to
\begin{equation}
\gamma^- \leq (\ln (\theta^-_{c,T}(t)))'.
\end{equation}
which, by integrating in $[t,t^-]$ becomes
\begin{equation}
\gamma^-(t^--t) \leq \ln (\theta^-_{c,T}(t^-))-\ln (\theta^-_{c,T}(t)),
\end{equation}
hence
\begin{equation}
e^{\gamma^-(t^--t)} \leq  \frac{\theta^-_{c,T}(t^-)}{\theta^-_{c,T}(t)},
\end{equation}
that is
\begin{equation}
\theta^-_{c,T}(t) e^{\gamma^-(t^--t)} \leq \theta^-_{c,T}(t^-),
\end{equation}
which clearly also holds if we drop the assumption $\theta^-_{c,T}(t)>0$, as $\theta^-_{c,T}$ is a non-negative function. Thus, we have shown that
\begin{equation}\label{theta-_est}
\forall t \leq t^-, \hspace{2mm} \theta^-_{c,T}(t) \leq \theta^-_{c,T}(t^-) e^{-\gamma^-(t^--t)}.
\end{equation}
Now, we have that using \eqref{theta-_est} we get for any fixed $t \leq t^--1$, $\eps>0$ and $i \in \N$
\begin{align}\label{est_epsilon}
\int_{t-i-1}^{t-i} (\CE(\bU_{c,T}(s))-a)e^{-\eps s} ds &\leq e^{-(c+\eps)(t-i-1)}\int_{t-i-1}^{t-i}(\CE(\bU_{c,T}(s))-a)e^{cs}ds\\ & \leq e^{-(c+\eps)(t-i-1)}\theta^-_{c,T}(t^-) e^{-\gamma^-(t^--t+i)} \\ &=e^{(c+\eps)(1-t^-)}\theta^-_{c,T}(t^-) e^{(\gamma^--c-\eps)(t-t^-)}e^{(c+\eps-\gamma^-)i}.
\end{align}
Since we assume that $c<\gamma^-$, we have that by choosing any $\eps \in (0,\gamma^--c)$ it holds
\begin{equation}
\sum_{i \in \N}e^{(c+\eps-\gamma^-)i} = \frac{1}{1-e^{(c+\eps-\gamma^-)}},
\end{equation}
which, in combination with \eqref{est_epsilon} gives \eqref{L1_CE_conv-} (notice that the case $t > t^--1$ presents no problem, as $e^{(\gamma^--c-\eps)t}$ is then large). Therefore, by \eqref{equipartition-_inequality} in Lemma \ref{LEMMA_equipartition_pointwise} we have that for all $\eps \in (0,\gamma^--c)$ and $t \in \R$
\begin{equation}
\int_{-\infty}^{t}\frac{\lVert U'(s) \rVert_{\scrL}^2}{2}e^{-\eps s}ds \leq\overline{M}^- e^{(\gamma^--c-\eps)t},
\end{equation}
which, by Cauchy-Schwarz inequality, means that
\begin{equation}\label{inequality_TV_L}
\int_{-\infty}^{t} \lVert \bU_{c,T}'(s) \rVert_{\scrL} ds \leq \left( \frac{e^{\eps t}}{\eps }\int_{-\infty}^{t} \lVert \bU_{c,T}'(s) \rVert_{\scrL}^2e^{-\eps s} ds \right)^{\frac{1}{2}} \leq \frac{2\overline{M}^-}{\eps}e^{(\gamma^--c)t},
\end{equation}
where we have used that $\lim_{s \to -\infty}e^{\eps s}=0$, because $\eps >0$. Since $c< \gamma^-$, in particular inequality \eqref{inequality_TV_L} implies the existence of some $\tilde{v}^- \in \scrL$ such that
\begin{equation}\label{tilde_v}
\lim_{t \to -\infty}\lVert \bU_{c,T}(t)-\tilde{v}^-\rVert_{\scrL}=0.
\end{equation}
Inequality \eqref{inequality_TV_L} also implies that for all $\tilde{t} < t \in \R$ we have
\begin{equation}
\lVert \bU_{c,T}(t)-\bU_{c,T}(\tilde{t}^-) \rVert_{\scrL} \leq \frac{2\overline{M}^-}{\eps}e^{(\gamma^--c)t},
\end{equation}
which by taking the limit $\tilde{t} \to -\infty$ and using \eqref{tilde_v} gives \eqref{solutions_limit-_norm}, by choosing for instance $\eps=(\gamma^--c)/2 \in (0,\gamma^--c)$ and $M^-=\frac{2\overline{M}^-}{\eps}>0$.
\end{proof}
\begin{remark}\label{REMARK_integrability_derivative}
Notice that combining \eqref{L1_CE_conv-} in Proposition \ref{PROPOSITION_conv_sol_-} with \eqref{equipartition-_inequality} in Lemma \ref{LEMMA_equipartition_pointwise}, we obtain in particular that $\bU_{c,T}' \in L^2(\R,\scrL)$ provided that $c<\gamma^-$ (see the statements of the results for the notations).
\end{remark}

\subsection{Proof of Theorem \ref{THEOREM_abstract_bc} completed}

Assume first that \ref{hyp_levelsets} holds. Let $(c^\star,\bU)$ be the solution to \eqref{abstract_equation} with conditions at infinity \eqref{abstract_bc_weak-} and \eqref{exponential_abstract} given by Proposition \ref{PROPOSITION_existence} and Lemma \ref{LEMMA_limit+}. Since we took $c^\star=\sup\CC$ with $\CC$ as in \eqref{set_CC}, inequality \eqref{sup_C_bound} in Lemma \ref{LEMMA_set_CC} implies that
\begin{equation}
c^\star \leq \frac{\sqrt{-2a}}{d_0}
\end{equation}
which by \eqref{a_convergence} in \ref{hyp_convergence} implies that
\begin{equation}
c^\star < \gamma^-
\end{equation}
so that we can apply \eqref{solutions_limit-_norm} in Proposition \ref{PROPOSITION_conv_sol_-} to $\bU$, as it is a minimizer of $\bE_{c^\star}$ in $X_{T^\star}$ for some $T^\star \geq 1$. Therefore, \eqref{abstract_bc_strong} holds for $\bU$, which completes the proof.

\qed

\subsection{Proof of Theorem \ref{THEOREM_abstract_speed} completed}

Since we assume that \ref{hyp_convergence} holds and $\tilde{\bU}$ is such that $\tilde{\bU} \in X_T$ for some $T \geq 1$ and $\bE_{c^\star}(\tilde{\bU})=0$, then  by Proposition \ref{PROPOSITION_existence}, we can apply Proposition \ref{PROPOSITION_conv_sol_-} to $\bU$. We recall that by Remark \ref{REMARK_integrability_derivative} we have that $\bU' \in L^2(\R,\scrL)$ and by \eqref{L1_CE_conv-} in Proposition \ref{PROPOSITION_conv_sol_-} we have that $\CE \circ \bU \in L^1((-\infty,t])$ for all $t \in \R$. Therefore, we can find a sequence $(t_n^-)_{n \in \N}$ in $\R$ such that

\begin{equation}\label{abstract_speed_lim1}
\lim_{n \to \infty} t_n^-=-\infty
\end{equation}
and
\begin{equation}\label{abstract_speed_lim2}
\lim_{n \to \infty}\left(\CE(\bU(t_n^-))-a- \frac{\lVert \bU'(t_n^-) \rVert_{\scrL}^2}{2} \right) = 0.
\end{equation}
Similarly, since $\bE_c(\bU)=0<+\infty$, we have that $\CE \circ \bU \in L^1([t,+\infty))$ for all $t \in \R$, which means that we can find $(t_n^+)_{n \in \N}$ a sequence of real numbers such that
\begin{equation}\label{abstract_speed_lim3}
\lim_{n \to \infty} t_n^+=+\infty
\end{equation}
and
\begin{equation}\label{abstract_speed_lim4}
\lim_{n \to \infty}\left( \frac{\lVert \bU'(t_n^+) \rVert_{\scrL}^2}{2}-\CE(\bU(t_n^+)) \right)=0.
\end{equation}
Taking the scalar product in $\scrL$ between equation \eqref{abstract_equation} and $\bU'$, we obtain
\begin{equation}
\forall t \in \R, \hspace{2mm} \langle \bU''(t),\bU'(t) \rangle_{\scrL} -\langle D_\scrL\CE(\bU(t)),\bU'(t) \rangle_{\scrL}=-c\lVert \bU'(t) \rVert_{\scrL}^2
\end{equation}
so that
\begin{equation}
\forall t \in \R, \hspace{2mm} \langle \bU''(t),\bU'(t) \rangle_{\scrL} -(\CE(\bU(t)))'=-c^\star\lVert \bU'(t) \rVert_{\scrL}^2.
\end{equation}
Fix $n \in \N$. Integrating above in $[t_n^-,t_n^+]$ (which is non-empty up to an extraction) we obtain
\begin{equation}\label{abstract_speed_equation}
\int_{t_n^-}^{t_n^+} \langle \bU''(t),\bU'(t) \rangle_\scrL dt-\CE(\bU(t_n^+))+\CE(\bU(t_n^-))=-c^\star\int_{t_n^-}^{t_n^+} \lVert \bU'(t) \rVert_{\scrL}^2dt.
\end{equation}
Integrating by parts we obtain
\begin{equation}
\int_{t_n^-}^{t_n^+} \langle \bU''(t),\bU'(t) \rangle_\scrL dt= \lVert \bU'(t_n^+) \rVert_{\scrL}^2-\lVert \bU'(t_n^-) \rVert_{\scrL}^2-\int_{t_n^-}^{t_n^+} \langle \bU'(t),\bU''(t) \rangle_{\scrL}dt,
\end{equation}
which means
\begin{equation}
\int_{t_n^-}^{t_n^+} \langle \bU''(t),\bU'(t) \rangle_\scrL dt=\frac{1}{2}(\lVert \bU'(t_n^+) \rVert_{\scrL}^2-\lVert \bU'(t_n^-) \rVert_{\scrL}^2).
\end{equation}
Plugging into \eqref{abstract_speed_equation} we obtain
\begin{equation}
a+\left(\CE(\bU(t_n^-))-a- \frac{\lVert \bU'(t_n^-) \rVert_{\scrL}^2}{2} \right) +\left( \frac{\lVert \bU'(t_n^+) \rVert_{\scrL}^2}{2}-\CE(\bU(t_n^+))\right) =-c^\star\int_{t_n^-}^{t_n^+} \lVert \bU'(t) \rVert_{\scrL}^2dt.
\end{equation}
Using \eqref{abstract_speed_lim1}, \eqref{abstract_speed_lim2}, \eqref{abstract_speed_lim3} and \eqref{abstract_speed_lim4}, along with the fact that $\bU' \in L^2(\R,\scrL)$, we can pass to the limit $n \to \infty$ and we get that
\begin{equation}
a=-c^\star\int_\R \lVert \bU'(t) \rVert_{\scrL}^2dt,
\end{equation}
which shows \eqref{abstract_speed_formula}. We now show that \eqref{abstract_speed_variational} holds. Inspecting again the proof of Theorem \ref{THEOREM-ABSTRACT}, we have that $c^\star$ is equal to $c(\CC)$ as in Corollary \ref{COROLLARY_CC}. Take $c<c^\star$, then by Corollary \ref{COROLLARY_CC} we have that $c \in \CC$. The definitions of $\CC$ in \eqref{set_CC} implies then that
\begin{equation}
\exists T \geq 1, \hspace{2mm} \inf_{U \in X_T} \bE_c(U)<0
\end{equation}
which, by considering $\tilde{U} \in X_T$ such that $\bE_c(\tilde{U})<0$ and then the sequence $(\tilde{U}(\cdot+n))_{n \in \N}$ which is contained in $X$, implies that $\inf_{U \in X}\bE_c(U)=-\infty$. If we now take $c > c^\star$, we have again by Corollary \ref{COROLLARY_CC} that
\begin{equation}
\forall T \geq 1, \hspace{2mm} \inf_{U \in X_T} \bE_c(U) \geq 0
\end{equation}
which means
\begin{equation}
\inf_{U \in X} \bE_c(U)=0.
\end{equation}
Therefore, \eqref{abstract_speed_variational} follows. Finally, we have that \eqref{abstract_speed_bound} is exactly \eqref{sup_C_bound} in Lemma \ref{LEMMA_set_CC}.
\qed
\section{Proofs of the main results completed}\label{section_proofs_final}
Once we have proved the abstract results, we are ready to prove the main ones. In order to do such a thing, we need to show that the main problem can be put into the abstract framework. This is shown in Lemma \ref{LEMMA_implication} which is in subsection \ref{subs_implication}. The next subsections are then devoted to the conclusion of the proofs of the main results, which are Theorems \ref{THEOREM_main_TW}, \ref{THEOREM_strong_bc} and \ref{THEOREM_carac_speed}. However, as pointed out before, we do not have a counterpart of Theorem \ref{THEOREM_uniform_convergence} in the abstract setting, which means that we prove it using arguments relative to the main setting.

\subsection{Proving the link between the main setting and the abstract setting}\label{subs_implication}
The following result establishes the link between the main assumptions and the abstract ones. As a consequence, the main results can be deduced from the abstract framework, which we have already established.
\begin{lemma}\label{LEMMA_implication}
Assume that \ref{asu_unbalanced} holds. Set
\begin{equation}\label{choice_spaces}
\scrL:= L^2(\R,\R^k), \hspace{2mm} \scrH:= H^1(\R,\R^k), \hspace{2mm} \tilde{\scrH}:= H^2(\R,\R^k),
\end{equation}
\begin{equation}\label{choice_constants}
r_0^\pm:= \rho_0^\pm
\end{equation}
and
\begin{equation}\label{choice_CE}
\forall v \in \scrL, \hspace{2mm} \CE(v):= \begin{cases}
E(\psi+v)-\Fm^+ &\mbox{ if } v \in \scrH, \\
+\infty &\mbox{ otherwise},
\end{cases}
\end{equation}
where $\Fm^+$ was introduced in \ref{asu_unbalanced}, the constants $\rho_0^\pm$ are those from \eqref{rho_1} and the function $\psi$ is any smooth function in $X(\sigma^-,\sigma^+)$ converging to $\sigma^\pm$ at $\pm \infty$ at an exponential rate and such that $\psi' \in H^2(\R,\R^k)$. Under this choice, assumptions \ref{hyp_unbalanced}, \ref{hyp_smaller_space}, \ref{hyp_compactness}, \ref{hyp_projections_inverse}, \ref{hyp_projections_H} and \ref{hyp_regularity} hold. Moreover, we have:
\begin{itemize}
\item If \ref{asu_perturbation} holds, then \ref{hyp_levelsets} holds.
\item If \ref{asu_convergence} holds, then \ref{hyp_convergence} holds.
\end{itemize}
\end{lemma}
\begin{proof}
The fact that the functional
\begin{equation}
v \in \scrH \to E(\psi+v)
\end{equation}
is well defined and, moreover, is a $C^1$ functional on $(\scrH,\lVert \cdot \rVert_{\scrH})$ is proven by classical arguments. See for instance Bisgard \cite{bisgard}, Montecchiari and Rabinowitz \cite{montecchiari-rabinowitz18}. See also \cite{oliver-bonafoux} for the precise statement in this setting. We obviously have $\scrF^\pm= \CF^\pm-\psi$. We now pass to prove that the assumptions are satisfied.

\textbf{Assumption \ref{hyp_unbalanced} is satisfied}: 

The fact that $\CE$ is weakly lower semicontinuous in $\scrL$ is standard, see \textit{Lemma 3.1} in \cite{smyrnelis}. We already invoked  \textit{Lemma 2.1} in \cite{schatzman} in Schatzmann so that  \eqref{rho_1} and \eqref{rho_2} hold. That is, due to \eqref{choice_constants} we have that if
\begin{equation}\label{inf_translations}
\inf_{\tau \in \R} \lVert v+\psi-\Fq^\pm(\cdot+\tau) \rVert_{\scrL} \leq r_0^\pm
\end{equation}
there is a unique $\tau(v) \in \R$ which attains the infimum in \eqref{inf_translations}. Moreover, the correspondence $v \to \tau(v)$ defined on the subset of $\scrL$ composed of functions that verify \eqref{inf_translations} is of class $C^2$. Therefore, the applications
\begin{equation}\label{projection_tau}
P^\pm: v \in \scrF^-_{r_0^\pm/2} \to \Fq^\pm(\cdot+\tau(v))-\psi \in \scrF^\pm,
\end{equation}
satisfy the properties required. Finally, we have that estimate \eqref{coercivityL} follows by \textit{Lemma 3.2} in \cite{monteil-santambrogio}, up to modifying the choice of the constants $\rho_0^\pm,\beta_0^\pm$.

\textbf{Assumption \ref{hyp_smaller_space} is satisfied}:

By \eqref{choice_spaces}, we have that $\tilde{\scrH} \subset \scrH \subset \scrL$ and the associated norms verify
\begin{equation}
\lVert \cdot \rVert_{\scrL} \leq \lVert \cdot \rVert_{\scrH} \leq \lVert \cdot \rVert_{\tilde{\scrH}}.
\end{equation}
As we pointed out before, $\CE$ restricted to $(\scrH,\lVert \cdot\rVert_{\scrH})$ is a $C^1$ functional. Moreover, as shown in \cite{bisgard,montecchiari-rabinowitz18}, we have that the differential is given by
\begin{equation}\label{DERIVATIVE_FUNCTIONAL}
\forall v \in \scrH, \hspace{2mm} D\CE(v):w \in \scrH \to \int_\R \left( \langle \psi'+v',w' \rangle + \langle \nabla V(\psi+v),w \rangle \right) \in \R.
\end{equation}
Let now $v \in \tilde{\scrH}$, since $\psi$ is smooth with good behavior at infinity we can integrate by parts to get
\begin{align}\label{IBP}
\forall w \in \scrH, \hspace{2mm} D\CE(v)(w) &= \int_\R \langle -(\psi''+v'')+\nabla V(\psi+v), w \rangle\\
&= \langle D_\scrL\CE(v),w \rangle_{\scrL},
\end{align}
where we have set
\begin{equation}
D_{\scrL}\CE: v \in (\tilde{\scrH},\lVert \cdot \rVert_{\tilde{\scrH}}) \to -(\psi''+v'')+\nabla V(\psi+v) \in (\scrL,\lVert \cdot \rVert_{\scrL}),
\end{equation}
which, by standard arguments, can be shown to be continuous. Notice that \eqref{abstract_ipp} in \ref{hyp_smaller_space} is exactly \eqref{IBP} above, which concludes this part of the proof.

\textbf{Assumption \ref{hyp_compactness} is satisfied}:

Let $(\bv_n^-)_{n \in \N}$ be a $\scrL$-bounded sequence in $\scrF^-$. We want to show the existence of a subsequence of $(\bv_n^-)_{n \in \N}$ strongly convergent in $\scrH$. Since 
\begin{equation}
\scrF^-=\CF^--\psi=\{ \Fq^-(\cdot+\tau)-\psi: \tau \in \R\},
\end{equation}
we have that $(\bv_n^-)_{n \in\N}=(\Fq^-(\cdot+\tau_n)-\psi)_{n \in \N}$ with $(\tau_n)_{n \in \N}$ a bounded sequence of real numbers. Since such a sequence is bounded in $\scrL$, we know that, up to an extraction, there exists $\tilde{v} \in \scrL$ such that $\Fq^-(\cdot+\tau_n)-\psi \rightharpoonup \tilde{v}$ weakly in $\scrL$. Due to the weak lower semicontinuity of $\CE$, we have that
\begin{equation}
\CE(\tilde{v}) \leq \liminf_{n \to +\infty}\CE(\Fq^-(\cdot+\tau_n)-\psi) = a,
\end{equation}
which, by minimality, implies that $\CE(v)=0$, that is, $\tilde{v} \in \scrF^-$. We can then write $\tilde{v}=\Fq^-(\cdot+\tau)-\psi$ for some $\tau \in \R$. Now, notice that, by the compactness of minimizing sequences \eqref{MS_COMPACTNESS}, there exists a sequence $(\tau_n')_{n \in \N}$ of real numbers such that, up to an extraction
\begin{equation}\label{conv_taun_taun'}
\Fq^-(\cdot+\tau_n+\tau_n')-\Fq^- \to 0 \mbox{ strongly in } \scrH
\end{equation}
which necessarily implies that
\begin{equation}
\tau_n+\tau_n' \to 0
\end{equation}
and, since $(\tau_n)_{n \in \N}$ is bounded, we have that $(\tau_n')_{n \in \N}$ is a bounded sequence as well. Therefore, we can assume, up to an extraction, that $\tau_n' \to \tau$. Combining this information with \eqref{conv_taun_taun'}, we obtain
\begin{equation}
\Fq^-(\cdot+\tau_n)-\Fq^-(\cdot-\tau^-) \to 0 \mbox{ strongly in } \scrH,
\end{equation}
which establishes the claim.

We need to show the same for $\scrF^+$. The argument is identical to the one above, except for the fact that the compactness of minimizing sequences is replaced by 3. in assumption \ref{asu_unbalanced}, which is in fact stronger, and we use that the elements in $\scrF^+$ are local minimizers (instead of global ones), which does not require any modification of the reasoning.

\textbf{Assumption \ref{hyp_projections_inverse} is satisfied}:

More precisely, we show that 2. in \ref{hyp_projections_inverse} holds. Notice that since the results are local in nature and \ref{hyp_unbalanced} implies that locally the situation does not change between $\scrF^-$ and $\scrF^+$, we can treat both cases together. Let $(v,\bv^\pm) \in \scrF^\pm_{r_0^\pm}$. Let $\tau(v)$ be given by the projection map defined in \eqref{projection_tau}. We have that $\bv^\pm=\Fq^\pm(\cdot+\tau)-\psi$ for some $\tau \in \R$. Define

\begin{equation}
\hat{P}^\pm_{(v,\bv^\pm)}: w \in \scrL \to w(\cdot-\tau(v)+\tau)-\psi+\psi(\cdot-\tau(v)+\tau) \in \scrL.
\end{equation}
Clearly, using the definition of the projection in \eqref{projection_tau} and $\tau(v)$
\begin{align}
\lVert \hat{P}^\pm_{(v,\bv^\pm)}(v)-\bv^\pm \rVert_{\scrL}&= \lVert v(\cdot-\tau(v)+\tau)-(\Fq^\pm(\cdot+\tau)-\psi(\cdot-\tau(v)+\tau)) \rVert_{\scrL}\\&= \lVert v-(\Fq^\pm(\cdot+\tau(v))-\psi) \rVert_{\scrL}=\inf_{\tilde{\tau} \in \R}\lVert v-(\Fq^\pm(\cdot+\tilde{\tau})-\psi) \rVert_{\scrL},
\end{align}
meaning that
\begin{equation}
P^\pm(\hat{P}^\pm_{(v,\bv^\pm)}(v))=\bv^\pm
\end{equation}
and
\begin{equation}
\dist_{\scrL}(\hat{P}^\pm_{(v,\bv^\pm)}(v),\scrF^\pm)=\dist_{\scrL}(v,\scrF^\pm)
\end{equation}
which are \eqref{hatP} and \eqref{hatP_dist} respectively. Next, notice that for $(w_1,w_2) \in \scrL^2$ and $h \in \R$ we have
\begin{equation}
\hat{P}^\pm_{(v,\bv^\pm)}(w_1+hw_2)= \hat{P}^\pm_{(v,\bv^\pm)}(w_1)+hw_2(\cdot-\tau(v)+\tau)
\end{equation}
so that $\hat{P}^\pm_{(v,\bv^\pm)}$ is differentiable and
\begin{equation}
\forall (w_1,w_2) \in \scrL^2, \hspace{2mm} D(\hat{P}^\pm_{(v,\bv^\pm)})(w_1,w_2)=w_2(\cdot-\tau(v)+\tau)
\end{equation}
so that
\begin{equation}
\forall (w_1,w_2) \in \scrL^2, \hspace{2mm} \lVert D(\hat{P}^\pm_{(v,\bv^\pm)})(w_1,w_2) \rVert_{\scrL}=\lVert w_2 \rVert_{\scrL},
\end{equation}
which is \eqref{D_inverse}. Finally, notice that $v  \in \scrH$ if and only if $\hat{P}^\pm_{(v,\bv^\pm)} \in \scrH$. Assuming that $v \in \scrH$ we have
\begin{align}
\CE(\hat{P}^\pm_{(v,\bv^\pm)}(v))&=\CE( v(\cdot-\tau(v)+\tau)-\psi+\psi(\cdot-\tau(v)+\tau) )\\&= E(v(\cdot-\tau(v)+\tau)+\psi(\cdot-\tau(v)+\tau))=E(\psi+v)=\CE(v)
\end{align}
and if $v \in \scrL \setminus \scrH$, we have $\CE(\hat{P}^\pm_{(v,\bv^\pm)}(v))=+\infty=\CE(v)$. Therefore, \eqref{hatP_CE} holds. We have then showed that \ref{hyp_projections_inverse} holds.

\textbf{Assumption \ref{hyp_projections_H} is satisfied}: 

Lemma \textit{2.1} in Schatzmann states that for $v \in \scrF_{\scrH,r_0^\pm}^\pm$  the problem
\begin{equation}
\inf_{\tau \in \R}\lVert v+\psi-\Fq^\pm(\cdot+\tau) \rVert_{\scrH}
\end{equation}
has a unique solution $\tau_\scrH(v) \in \R$ and the projection map
\begin{equation}
P^\pm_\scrH: v \in \scrF^\pm_{\scrH,r_0^\pm} \to \Fq^\pm(\cdot+\tau_\scrH(v)) \in \scrF^\pm
\end{equation}
is $C^1$ with respect to the $\scrH$-norm. Next, we have that \eqref{H_difference_projections} is \textit{Corollary 2.3} in \cite{schatzman}. Finally, the fact that \eqref{H_local_bound} implies \eqref{H_local_est} for the constants $C^\pm$ (up to possibly increasing) is a consequence of the compactness of the minimizing sequences. See for example \textit{Corollary 3.2} in \cite{schatzman}.

\textbf{Assumption \ref{hyp_regularity} is satisfied}:

We show the existence of the map $\FP$. We follow \textit{Lemma 3.3} in \cite{schatzman}. Let $R_0>0$ be the constant from \ref{asu-infinity}. For $R\geq R_0$ define in $\R^k$
\begin{equation}
f_{R}(u):= \begin{cases}
u &\mbox{ if } \lvert u \rvert \leq R,\\
R \frac{u}{\lvert u \rvert} &\mbox{ otherwise},
\end{cases}
\end{equation}
where $R_0$ is the constant from \ref{asu-infinity}. For $u \in R^k$ such that $\lvert u \rvert \leq R$, we have $f_{R}(u)=u$. Assume that $u \in \R^k$ is such that $\lvert u \rvert > R$. In that case, there exists $\xi \in \left(\frac{R}{\lvert u \rvert},1\right)$ such that
\begin{equation}
V(u)=V(f_{R}(u))+\langle\nabla_uV(\xi u),u-f_{R}(u) \rangle = V(f_{R}(u))+\frac{1}{\xi}\left(1-\frac{R}{\lvert u \rvert}\right)\langle\nabla_uV(\xi u), \xi u \rangle
\end{equation}
which, by \ref{asu-infinity} implies
\begin{equation}\label{fR_0.5}
\forall R \geq R_0, \forall u \in \R^k: \lvert u \rvert > R, \hspace{2mm} V(u) \geq V(f_{R}(u))+\frac{1}{\xi}\left(1-\frac{R}{\lvert u \rvert}\right)\nu_0 \lvert \xi u \rvert^2 > V(f_{R}(u)).
\end{equation}
In particular, we have shown
\begin{equation}\label{fR_1}
\forall R \geq R_0, \forall u \in \R^k, \hspace{2mm} V(u) \geq V(f_{R}(u)).
\end{equation}
Next, let $J \subset \R$ be a compact interval and $v \in H^1(J,\R^k)$. For $R \geq R_0$, consider the function $v_R:= f_R \circ v$. Since we clearly have that for all $u \in \R^k$, $\lvert f_R(u) \rvert \leq \lvert u \rvert$, it holds that $v_R \in L^2(J,\R^k)$. Next, we have that $f_R$ is the projection onto the closed ball of center 0 and radius $R$, so that it is non-expansive. As a consequence, we have
\begin{equation}\label{fR_1.5}
\forall R \geq R_0, \forall u \in \R^k, \hspace{2mm} \lvert Df_R(u) \rvert \leq 1.
\end{equation}
 Therefore, applying the chain rule we obtain
\begin{equation}
\mbox{for a. e. } t \in J, \hspace{2mm} \lvert v_R'(t) \rvert \leq \lvert v'(t) \rvert,
\end{equation}
which means that $v_R \in H^1(J,\R^k)$ and, combining with \eqref{fR_1} we obtain
\begin{equation}\label{fR_2}
E(v_R;J) \leq E(v;J)
\end{equation}
and, by \eqref{fR_0.5} the equality above holds if and only if $v_R=v$. Let now
\begin{equation}
R_{\max}:=2 \max\{ R_0, \lVert \Fq^- \rVert_{L^\infty(\R,\R^k)},\lVert \Fq^+ \rVert_{L^\infty(\R,\R^k)}\}.
\end{equation}
Consider now the application
\begin{equation}\label{FP_def_link}
\FP: v \in \scrL \to f_{R_{\max}}\circ (v+\psi)-\psi \in \scrL
\end{equation}
which is well-defined due to the previous considerations. Moreover, the choice of $R_{\max}$ implies that $\FP$ equals the identity on $\{\Fq^-(\cdot+\tau)-\psi: \tau \in \R\}$ and $\{ \Fq^+(\cdot+\tau)-\psi:\tau \in \R \}$, which is exactly \eqref{FP_3}. Inequality \eqref{fR_2} gives \eqref{FP_1}. Finally, using \eqref{fR_1.5} we have that
\begin{align}
\forall (v_1,v_2) \in \scrL^2, \hspace{2mm} \lVert \FP(v_1)-\FP(v_2) \rVert_{\scrL}^2&= \int_\R \lvert f_{R_{\max}} \circ (v_1+\psi) - f_{R_{\max}}\circ (v_2+\psi) \rvert^2\\
& \leq \int_\R \sup_{u \in \R^k}\lvert Df_{R_{\max}}(u) \rvert^2 \lvert v_1-v_2 \rvert^2\\
& \leq \int_\R \lvert v_1-v_2 \rvert^2=\lVert v_1-v_2 \rVert_{\scrL}^2,
\end{align}
which is \eqref{FP_2}. Therefore, our map $\FP$ satisfies the required properties. Next, let $\bW$ be a local minimizer of $\bE_c$. We show that $\bW$ satisfies the desired regularity properties, that is, $\bW \in \CA(I)$ with $\CA(I)$ as in \eqref{CA_def}. Write $\overline{\bW}:=\bW+\psi$. We assume that for all $t \in I$, $\bW(t)=\FP(\bW)$. The definition of $\FP$ in \eqref{FP_def_link} implies that
 \begin{equation}
 \forall (x_1,x_2) \in I \times \R, \hspace{2mm} \overline{\bW}(x_1,x_2)=f_{R_{\max}}(\overline{\bW}(x_1,x_2))
\end{equation}
so that
\begin{equation}\label{Linfinity_W}
\lVert \overline{\bW} \rVert_{L^\infty(I\times \R,\R^k)} \leq R_{\max}.
\end{equation}
Therefore, by classical elliptic regularity arguments, we have that, with the obvious identifications, $\overline{\bW}$ solves
\begin{equation}\label{eq_U}
-c\partial_{x_1}\overline{\bW}-\Delta \overline{\bW} =-\nabla_uV(\overline{\bW}) \mbox{ in } I \times \R
\end{equation}
 and for all $\alpha \in (0,1)$  we have that $\overline{\bW} \in \CC^{3,\alpha}(I_C \times \R,\R^k)$ for any compact $I_C \subset I$. It is then clear that 
\begin{equation}
\bW \in \CC^2(I_C,L^2(\R,\R^k)) \cap \CC^1(I_C,H^1(\R,\R^k) \cap \CC^0(I_C,H^2(\R,\R^k))
\end{equation}
for any $I_C \subset I$ compact, which means that $\bW \in \CA(I)$.

\textbf{Assumption \ref{asu_perturbation} implies \ref{hyp_levelsets}}: Immediate.

\textbf{Assumption \ref{asu_convergence} implies \ref{hyp_convergence}}: Immediate.
\end{proof}

Once Lemma \ref{LEMMA_implication} has been established, the main results are easily obtained by rephrasing the abstract ones.
\subsection{Proof of Theorem \ref{THEOREM_main_TW} completed}
Assume that \ref{asu_perturbation} holds. Notice that \ref{asu_perturbation} implies that \ref{asu-sigma}, \ref{asu-infinity}, \ref{asu-zeros}, \ref{asu-wells} and \ref{asu_unbalanced} hold. Therefore, applying Lemma \ref{LEMMA_implication} we have that, choosing the objects as in its statement, we get that \ref{hyp_compactness}, \ref{hyp_projections_inverse}, \ref{hyp_projections_H}, \ref{hyp_regularity} and \ref{hyp_levelsets} hold. Those are exactly the assumptions which are needed for Theorem \ref{THEOREM-ABSTRACT} to hold, meaning that we obtain $(c^\star,\bU)$ with $c^\star>0$ and $\bU \in \CA(\R) \cap X$, with $\CA(\R)$ as in \eqref{CA_def} and $X$ as in \eqref{X}, which solves
\begin{equation}\label{abstract_equation_proof}
\bU''-D_{\scrL}\CE(\bU)=-c\bU' \mbox{ in } \R
\end{equation}
and satisfies the conditions at infinity
\begin{equation}\label{abstract_bc_weak_proof}
\exists T^- \leq 0: \forall t \leq T^-, \hspace{2mm} \bU(t) \in \scrF^-_{r_0^-/2} \hspace{1mm}\mbox{ and }\hspace{1mm} \exists \bv^+(\bU) \in \scrF^+: \lim_{t \to +\infty} \lVert \bU(t)-\bv^+(\bU) \rVert_{\scrH}=0.
\end{equation}
We now pass to prove each of the three statements of Theorem \ref{THEOREM_main_TW} separately:

\begin{enumerate}
\item \textbf{Existence}. Recall that for all $t \in \R$ we have $\bU(t) \in \scrL=L^2(\R,\R^k)$. Let us then define
\begin{equation}\label{identification}
\FU: (x_1,x_2) \in \R^2 \to \bU(x_1)(x_2) \in \R^k.
\end{equation}
It is clear then that since $\bU \in \CA(\R)$ we have that $\FU \in \CC^2_{\loc}(\R,\R^k)$ and, moreover for all $(x_1,x_2)$ and any pair of index $(i,j) \in \{0,1,2\}^2$ such that $i+j \leq 2$ we have
\begin{equation}\label{derivatives_FU}
\partial_{x_1}^i\partial_{x_2}^j\FU(x_1,x_2)=(\bU^{(i)}(x_1))^{(j)}(x_2)
\end{equation}
where for a curve $f$ taking values in a Hilbert space we denote by $f^{(i)}$ its $i$-th derivative, $i \in\N$. As a consequence of \eqref{abstract_equation_proof}, \eqref{derivatives_FU} and the formula for $D\CE$ when we make $\CE=E-\Fm^+$ (see \eqref{DERIVATIVE_FUNCTIONAL}) we obtain that
\begin{equation}\label{elliptic_system_proog}
-c\partial_{x_1}\FU-\Delta \FU = -\nabla_u V(\FU) \mbox{ in } \R^2
\end{equation}
and by \eqref{abstract_bc_weak_proof} we obtain that for some $L \in \R$ we have for some $x_1 \leq L$ that $\FU(x_1,\cdot) \in \CF^-_{\rho^-/2}$, since we choose $r_0^\pm=\rho^\pm$, so that $\CF^\pm_{\rho^\pm/2}=\scrF^\pm_{r_0^\pm/2}$.  The variational characterization \eqref{FU_min} follows directly from Theorem \ref{THEOREM_main_TW}, using the fact that we have $X=S$ and $\bE_c=E_{2,c}$ for all $c>0$ (again we implicitly identify $\FU$ with $\bU$ via \eqref{identification}. Finally, we have that for all $t \in \R$, $\bU(t)=\FP(\bU(t))$. According to the choice of $\FP$ made in Lemma \ref{LEMMA_implication}, this implies that $\lVert \FU \rVert_{L^\infty(\R,\R^k)}<+\infty$, which by classical Schauder theory and the smoothness properties of $V$ implies that for all $\alpha \in (0,1)$, $\FU \in \CC^{2,\alpha}(\R^2,\R^k)$. The proof of the existence part of Theorem \ref{THEOREM_main_TW} is hence completed.
\item \textbf{Uniqueness of the speed}. Again, we have that $X=S$ and $\bE_c=E_{2,c}$ for all $c>0$, meaning that the proof of this statement follows from the analogous one in Theorem \ref{THEOREM-ABSTRACT}.
\item \textbf{Exponential convergence}. Using the exponential rate of convergence of $\bU$ given by Theorem \ref{THEOREM-ABSTRACT}, which is \eqref{exponential_abstract}, we obtain that for some $b>c^\star/2$ it holds
\begin{equation}
\lim_{x_1 \to +\infty}\lVert \FU(x_1,\cdot)-\Fq^+(\cdot+\tau^+) \rVert_{H^1(\R,\R^k)}e^{bx_1}=0
\end{equation}
for some $\tau^+ \in \R$. This concludes the proof of the statement.
\end{enumerate}
The proof of Theorem \ref{THEOREM_main_TW} is concluded.
\qed
\subsection{Proof of Theorem \ref{THEOREM_strong_bc} completed}
Assume that \ref{asu_perturbation} and \ref{asu_convergence} hold. Arguing as in the proof of Theorem \ref{THEOREM_main_TW}, we have that the assumptions of Theorem \ref{THEOREM_abstract_bc} are fulfilled if we choose as in Lemma \ref{LEMMA_implication}. Therefore, Theorem \ref{THEOREM_strong_bc} is readily obtained by a straightforward rewriting of Theorem \ref{THEOREM_abstract_bc}.
\qed
\subsection{Proof of Theorem \ref{THEOREM_uniform_convergence}}
We now provide the proof of Theorem \ref{THEOREM_uniform_convergence}, which is a consequence of the following results, which are more general as required by Theorem \ref{THEOREM_uniform_convergence} and might be of independent interest:
\begin{lemma}\label{LEMMA_uniform_convergence+}
Assume that \ref{asu-sigma}, \ref{asu-infinity} and \ref{asu-zeros} hold. Let $(\hat{\sigma}_-,\hat{\sigma}_+) \in \Sigma^2$ (possibly equal) and $q \in X(\hat{\sigma}_-,\hat{\sigma}_+)$. Assume moreover that there exists $L^+ \in \R$ and $U \in H^1_{\loc}([L^+,+\infty) \times \R,\R^k)$ uniformly continuous and such that
\begin{equation}\label{uniform_convergence+hyp1}
\int_{L^+}^{+\infty}\lvert E(U(x_1,\cdot))-E(q)\rvert dx_1<+\infty,
\end{equation}
\begin{equation}\label{uniform_convergence+hyp2}
\lim_{x_1 \to +\infty}\lVert U(x_1,\cdot)-q \rVert_{L^2(\R,\R^k)}=0.
\end{equation}
Then, it holds that
\begin{equation}\label{uniform_convergence+res1}
\lim_{x_1 \to +\infty} \lVert U(x_1,\cdot)-q \rVert_{L^\infty(\R,\R^k)}=0
\end{equation}
and
\begin{equation}\label{uniform_convergence+res2}
\lim_{x_2 \to \pm \infty} \lVert U(\cdot,x_2)-\hat{\sigma}_\pm \rVert_{L^\infty([L^+,+\infty),\R^k)}=0.
\end{equation}
\end{lemma}
Similarly, we have the following
\begin{lemma}\label{LEMMA_uniform_convergence-}
Assume that \ref{asu-sigma}, \ref{asu-infinity} and \ref{asu-zeros} hold. Let $(\hat{\sigma}_-,\hat{\sigma}_+) \in \Sigma^2$ (possibly equal) and $q \in X(\hat{\sigma}_-,\hat{\sigma}_+)$. Assume moreover that there exists $L^- \in \R$ and $U \in H^1_{\loc}((-\infty,L^-]\times \R,\R^k)$ uniformly continuous and such that
\begin{equation}\label{uniform_convergence-hyp1}
\int_{-\infty}^{L^-}lvert E(U(x_1,\cdot))-E(q)\rvert dx_1<+\infty,
\end{equation}
\begin{equation}\label{uniform_convergence-hyp2}
\lim_{x_1 \to -\infty}\lVert U(x_1,\cdot)-q \rVert_{L^2(\R,\R^k)}=0.
\end{equation}
Then, it holds that
\begin{equation}\label{uniform_convergence-res1}
\lim_{x_1 \to -\infty} \lVert U(x_1,\cdot)-q \rVert_{L^\infty(\R,\R^k)}=0
\end{equation}
and
\begin{equation}\label{uniform_convergence-res2}
\lim_{x_2 \to \pm \infty} \lVert U(\cdot,x_2)-\hat{\sigma}_\pm \rVert_{L^\infty((-\infty,L^-],\R^k)}=0.
\end{equation}
\end{lemma}
In the proof of Lemmas \ref{LEMMA_uniform_convergence+} and \ref{LEMMA_uniform_convergence-}, we will need to use the following fact:
\begin{lemma}\label{LEMMA_1D_H1conv}
Assume that \ref{asu-sigma}, \ref{asu-infinity} and \ref{asu-zeros} hold. Let $(\hat{\sigma}_-,\hat{\sigma}_+) \in \Sigma^2$ (possibly equal) and $q \in X(\hat{\sigma}_-,\hat{\sigma}_+)$. Assume that $(q_n)_{n \in \N}$ is a sequence in $X(\hat{\sigma}_-,\hat{\sigma}_+)$ such that
\begin{equation}\label{H1conv_hyp1}
\lim_{n \to \infty} \lVert q_n-q \rVert_{L^2(\R,\R^k)}=0
\end{equation}
and
\begin{equation}\label{H1conv_hyp2}
\lim_{n \to \infty} E(q_n) = E(q),
\end{equation}
then, it holds that
\begin{equation}\label{H1conv_res}
\lim_{n \to \infty}\lVert q_n-q \rVert_{H^1(\R,\R^k)}=0.
\end{equation}
\end{lemma}
\begin{proof}
First, notice that
\begin{equation}\label{H1_conv_uniform_bound_Linf}
\sup_{n \in \N}\lVert q_n \rVert_{L^\infty(\R,\R^k)}<+\infty.
\end{equation}
Indeed, \eqref{H1conv_hyp2} implies that $(q_n')_{n \in \N}$ is bounded in $L^2(\R,\R^k)$ which, in combination with \eqref{H1conv_hyp1}) means that $(q_n)_{n \in \N}$ is bounded in $H^1(\R,\R^k)$, hence in $L^\infty(\R,\R^k)$. We also have that
\begin{equation}\label{H1_conv_bound_gradient}
\nabla V(q) \in L^2(\R,\R^k),
\end{equation}
which follows easily from the fact that $V$ is smooth and quadratic near the wells. For all $n \in \N$, we write the following expansion
\begin{equation}\label{H1_conv_expansion}
V(q_n)=V(q)+\langle \nabla V(q),q_n-q \rangle +\int_{0}^1 D^2V(q+\lambda(q_n-q))(q_n-q)(q_n-q) d\lambda,
\end{equation}
which holds pointwise in $\R$. Therefore, Cauchy-Schwartz inequality implies
\begin{equation}
\left( \int_{\R} \lvert V(q_n)-V(q) \rvert \right)^2 \leq \left( \int_{\R} \lvert \nabla V(q) \rvert^2+\sup_{\substack{ u \in \R^k \\ \lvert u \rvert \leq \lVert q_n-q \rVert_{L^\infty}}}\lvert D^2V(u)\rvert \right) \lVert q_n-q \rVert_{L^2(\R,\R^k)}^2,
\end{equation}
hence, by \eqref{H1_conv_uniform_bound_Linf} and \eqref{H1_conv_bound_gradient} we find a constant $C>0$ such that for all $n \in \N$
\begin{equation}
 \int_{\R} \lvert V(q_n)-V(q) \rvert \leq C \lVert q_n-q \rVert_{L^2(\R,\R^k)}
\end{equation}
which by \eqref{H1conv_hyp1} means that $V(q_n)-V(q) \to 0$ in $L^1(\R,\R^k)$. As a consequence, \eqref{H1conv_hyp2} implies that
\begin{equation}\label{H1_conv_norm_derivative}
\lim_{n \to \infty} \lVert q_n' \rVert_{L^2(\R,\R^k)}=\lVert q' \rVert_{L^2(\R,\R^k)}.
\end{equation}
Suppose now by contradiction that \eqref{H1conv_res} does not hold. Then, we can find a subsequence $(q_{n_m})_{m \in \N}$ and $\hat{\delta}>0$ such that for all $m \in \N$
\begin{equation}\label{H1conv_contr}
\lVert q_{n_m}-q \rVert_{H^1(\R,\R^k)} \geq \hat{\delta}.
\end{equation}
Since $(q_{n_m}')_{m \in \N}$ is bounded in $L^2(\R,\R^k)$, it converges weakly in $L^2$ up to an extraction, and the limit is $q'$ by uniqueness of the limit in the sense of distributions. By \eqref{H1_conv_norm_derivative}, we have that such a subsequence also converges strongly in $L^2(\R,\R^k)$, which combining with \eqref{H1conv_hyp1} contradicts \eqref{H1conv_contr}.
\end{proof}
We now prove Lemma \ref{LEMMA_uniform_convergence+}. The proof of Lemma \ref{LEMMA_uniform_convergence-} being analogous, we skip it.
\begin{proof}[Proof of Lemma \ref{LEMMA_uniform_convergence+}]
Assume by contradiction that \eqref{uniform_convergence+res1} does not hold. Then, we can find a sequence $(x_{1,n})_{n \in \N}$ in $[L^+,+\infty) \times \R$ such that $x_{1,n} \to +\infty$ as $n \to \infty$ as well as $\hat{\delta}>0$ such that for all $n \in \N$
\begin{equation}
\lVert U(x_{1,n},\cdot)-q \rVert_{L^\infty(\R,\R^k)} \geq \hat{\delta}.
\end{equation}
By uniform continuity, there exists $\nu>0$ such that for all $n \in \N$ we have
\begin{equation}\label{uniform_convergence_contr1}
\max_{x_1 \in [x_{1,n}-\nu,x_{1,n}+\nu]}\lVert U(x_1,\cdot)-q \rVert_{L^\infty(\R,\R^k)} \geq \frac{\hat{\delta}}{2}
\end{equation}
Let $A:=\cup_{n \in \N}[x_{1,n}-\nu,x_{1,n}+\nu] $. By \eqref{uniform_convergence+hyp1} we have that
\begin{equation}
\int_A (E(U(x_1,\cdot))-E(q)) dx_1<+\infty,
\end{equation}
and since $A$ has positive measure and it is unbounded by above, we find a sequence $(y_{1,n})_{n \in \N}$ in $A$ such that $y_{1,n}\to +\infty$ as $n \to \infty$ and $\lim_{n \to \infty} E(U(y_{1,n},\cdot))=E(q)$. Combining this fact with \eqref{uniform_convergence+hyp2}, we have that assumptions \eqref{H1conv_hyp1} and \eqref{H1conv_hyp2} in Lemma \ref{LEMMA_1D_H1conv} hold, which means that
\begin{equation}
\lim_{n \to \infty} \lVert U(y_{1,n},\cdot) - q \rVert_{H^1(\R,\R^k)}=0
\end{equation}
which contradicts \eqref{uniform_convergence_contr1}. Therefore, we have shown that \eqref{uniform_convergence+res1} holds. In order to prove \eqref{uniform_convergence+res2}, we first show that there exists $\underline{L}^+ \leq L^+$ such that
\begin{equation}\label{uniform_convergence+res2_prel}
\lim_{x_2 \to \pm \infty} \lVert U(\cdot,x_2)-\hat{\sigma}_\pm \rVert_{L^\infty([\underline{L}^+,+\infty),\R^k)}=0.
\end{equation}
We prove \eqref{uniform_convergence+res2_prel} by contradiction. The other case being handled in an analogous fashion, assume that there exists a sequence $(x_{2,n})_{n \in \N}$ in $\R$ such that $x_{2,n}\to +\infty$ as $n \to \infty$, a sequence $(x_{1,n})_{n \in \N}$ in $[L^+,+\infty)$ tending to $+\infty$ and $\hat{\delta}>0$ such that for all $n \in \N$
\begin{equation}\label{uniform_convergence_contr2}
\lvert U(x_{1,n},x_{2,n})-\hat{\sigma}_+ \rvert \geq \hat{\delta}.
\end{equation}
Since we already proved that \eqref{uniform_convergence+res1} holds, there exists $N_1 \in \N$ such that for all $n \geq \N$ we have
\begin{equation}\label{uniform_convergence_N1}
\lVert U(x_{1,n},\cdot)-q \rVert_{L^\infty(\R,\R^k)} \leq \frac{\hat{\delta}}{4}
\end{equation}
and, since $q \in X(\hat{\sigma}_-,\hat{\sigma}_+)$,  there exists $\hat{t} \in \R$ such that for all $t \geq \hat{t}$ we have
\begin{equation}\label{uniform_convergence_N2}
\lvert q(t)-\hat{\sigma}_+ \rvert \leq \frac{\hat{\delta}}{4}.
\end{equation}
Let $N_2 \in \N$ be such that for all $n \in \N$, $x_{2,n} \geq \hat{t}$. Taking any $n \geq \max\{N_1,N_2\}$, we obtain by \eqref{uniform_convergence_N1} and \eqref{uniform_convergence_N2} that
\begin{equation}
\lvert U(x_{1,n},x_{2,n})-\hat{\sigma}_+ \rvert \leq \frac{\delta}{2},
\end{equation}
which contradicts \eqref{uniform_convergence_contr2} and establishes \eqref{uniform_convergence+res2_prel}. In order to establish \eqref{uniform_convergence+res2}, we handle the limit $x_2 \to +\infty$, as the other one is treated identically. Let $\rho_{\Sigma}^+:= \dist(\sigma^+,\Sigma \setminus \{\sigma^+\})>0$. 
We claim that for every $\tilde{L} \geq L^+$ we have that if
\begin{equation}\label{tilde_L_hyp}
\lim_{x_2 \to \pm \infty} \lVert U(\cdot,x_2)-\hat{\sigma}_\pm \rVert_{L^\infty([\tilde{L},+\infty),\R^k)}=0,
\end{equation}
then
\begin{equation}\label{tilde_L_res}
\lim_{x_2 \to \pm \infty} \lVert U(\cdot,x_2)-\hat{\sigma}_\pm \rVert_{L^\infty([\tilde{L}-\eta_{\Sigma}^+,+\infty),\R^k)}=0,
\end{equation}
where
\begin{equation}\label{eta_SIGMA}
\eta_{\Sigma}^+:= \min\left\{\tilde{L}-L^+,\frac{\rho_{\Sigma}}{4\lVert D\FU \rVert_{L^\infty(\R^2,\R^k)} }\right\}.
\end{equation}
Such a claim allows to easily complete the proof of \eqref{uniform_convergence+res2} by a finite induction process, due to the fact that \eqref{uniform_convergence+res2_prel} holds.
\end{proof}
It remains to establish one claim in the proof of Lemma \ref{LEMMA_uniform_convergence+}.
\begin{proof}[Proof that \eqref{tilde_L_hyp} implies \eqref{tilde_L_res}]
Assume that \eqref{tilde_L_hyp} holds. We show that for every $\eps \in (0,\eta_\Sigma^+)$ we have
\begin{equation}\label{tilde_L_res_eps}
\lim_{x_2 \to \pm \infty} \lVert U(\cdot,x_2)-\hat{\sigma}_\pm \rVert_{L^\infty([\tilde{L}-\eta_{\Sigma}^++\eps,+\infty),\R^k)}=0,
\end{equation}
which clearly implies \eqref{tilde_L_res} by uniform continuity. Fix then $\eps \in  (0,\eta_\Sigma^+)$. By assumption, there exists $\overline{x}_2^+ \in \R$ such that for all $x_2 \geq \overline{x}_2^+$ we have
\begin{equation}
\lvert U(\tilde{L},x_2)-\sigma^+ \rvert \leq \frac{\rho_{\Sigma}^+}{4},
\end{equation}
which, by \eqref{eta_SIGMA} implies that for all $(x_1,x_2) \in [\tilde{L}-\eta^+_{\Sigma}+\eps,\tilde{L}] \times [\overline{x}_2,+\infty)$, it holds
\begin{equation}\label{rho_sigma_contr1}
\lvert U(x_1,x_2)-\sigma^+ \rvert \leq \frac{\rho_{\Sigma}^+}{2}
\end{equation}
and the definition of $\rho_{\Sigma}^+$ gives in turn that for all such $(x_1,x_2)$ and $\sigma\in \Sigma \setminus \{\sigma^+\}$ we have
\begin{equation}\label{rho_sigma_contr2}
\lvert U(x_1,x_2)-\sigma \rvert \geq \frac{\rho_{\Sigma}^+}{2}.
\end{equation}
Assume now that \eqref{tilde_L_res_eps} does not hold. Then, inequalities \eqref{rho_sigma_contr1} and \eqref{rho_sigma_contr2} imply that we can find a sequence $(x_{1,n},x_{2,n})_{n \in \N}$ contained in $ [\tilde{L}-\eta^+_{\Sigma}+	\eps,\tilde{L}] \times [\overline{x}_2,+\infty)$, such that $x_{2,n} \to +\infty$ as $n \to \infty$ and $\hat{\delta}>0$ such that for all $n \in \N$ and $\sigma \in \Sigma$
\begin{equation}
\lvert U(x_{1,n},x_{2,n})-\sigma \rvert \geq \hat{\delta}.
\end{equation}
By uniform continuity, we can find $\nu \in (0,\eps)$ such that for all $n \in \N$ and
\begin{equation}
(x_1,x_2) \in B((x_{1,n},x_{2,n}),\nu) \subset  [\tilde{L}-\eta^+_{\Sigma},\tilde{L}] \times [\overline{x}_2,+\infty)
\end{equation}
we have for all $\sigma \in \Sigma$
\begin{equation}
\lvert U(x_1,x_2)-\sigma \rvert \geq \frac{\hat{\delta}}{2}
\end{equation}
or, equivalently
\begin{equation}\label{uniform_convergence_N3}
V(U(x_1,x_2)) \geq V_{\hat{\delta}/2}:=\min\left\{V(u): u \in \R^k, \dist(u,\Sigma) \geq \frac{\hat{\delta}}{2} \right\}
\end{equation}
which is positive by \ref{asu-sigma} and \ref{asu-zeros}. Up to an extraction and since $x_{2,n} \to +\infty$ as $n \to \infty$, we can assume that whenever $n \not = m$ we have
\begin{equation}
B((x_{1,n},x_{2,n}),\nu) \cap  B((x_{1,m},x_{2,m}),\nu) = \emptyset,
\end{equation}
which, due to the definition of $\eta_{\Sigma}^+$ in \eqref{eta_SIGMA} and \eqref{uniform_convergence_N3} implies that
\begin{align}
\int_{L^+}^{+\infty}\lvert E(U(x_1))-E(q)\rvert dx_1 &\geq \int_{\tilde{L}-\eta^+_{\Sigma}}^{\tilde{L}}E(U(x_1))dx_1-\eta_{\Sigma}^+E(q)\\& \geq \sum_{n \in \N}\left(\int_{B((x_{1,n},x_{2,n}),\nu)}V(U(x_1,x_2)) dx_1dx_2\right)- \eta_{\Sigma}^+E(q)\\
& \geq \sum_{n \in \N} (\pi \nu^2 V_{\hat{\delta}/2}) -\eta_{\Sigma}^+E(q)=+\infty,
\end{align}
which enters in contradiction with \eqref{uniform_convergence+hyp1}. Therefore, the claim has been proven.
\end{proof}
We have now all the necessary ingredients for completing the proof of Theorem \ref{THEOREM_uniform_convergence}:
\begin{proof}[Proof of Theorem \ref{THEOREM_uniform_convergence} completed]
Let $(c^\star,\FU)$ be the solution given by Theorem \ref{THEOREM_main_TW}, interpreted via the choices made in Lemma \ref{LEMMA_implication}. We will invoke Lemma \ref{LEMMA_uniform_convergence+}. The $L^2$ exponential convergence \eqref{exp_convergence_main} given by Theorem \ref{THEOREM_main_TW} implies in particular that assumption \eqref{uniform_convergence+hyp2} in Lemma \ref{LEMMA_uniform_convergence+} holds with $U=\FU$, $q=\Fq^+(\cdot+\tau^+)$. Moreover, since $E_{2,c^\star}(\FU)=0<+\infty$, assumption \eqref{uniform_convergence+hyp1} in Lemma \ref{LEMMA_uniform_convergence+} holds for all $L \in \R$ in view of the definition of $E_{2,c^\star}$ (recall that $c^\star>0$). Finally, we have by Theorem \ref{THEOREM_main_TW} that $\FU \in \CC^{2,\alpha}(\R^2,\R^k)$, $\alpha \in (0,1)$, so that $\FU$ is uniformly continuous. As a consequence, Lemma \ref{LEMMA_uniform_convergence+} applies and we have \eqref{uniform_convergence+res1} and \eqref{uniform_convergence+res2} for all $L \in \R$, and this is exactly \eqref{uniform_convergence_1} and \eqref{uniform_convergence_2} for all $L \in \R$. 

Assume now that \ref{asu_convergence} holds, so that Theorem \ref{THEOREM_strong_bc} applies. We will show that we can invoke Lemma \ref{LEMMA_uniform_convergence-}. We have that \eqref{exp_convergence_-} in Theorem \ref{THEOREM_strong_bc} implies that \eqref{uniform_convergence-hyp2} in Lemma \ref{LEMMA_uniform_convergence-} holds with $U=\FU$ and $q=\Fq^-(\cdot+\tau^-)$. Moreover, the abstract result Proposition \ref{PROPOSITION_conv_sol_-} in combination with Lemma \ref{LEMMA_implication} implies in particular that for all $L \in \R$, \eqref{uniform_convergence-hyp2} in Lemma \ref{LEMMA_uniform_convergence-} holds. Since $\FU$ is uniformly continuous, Lemma \ref{LEMMA_uniform_convergence-} applies, which means that \eqref{uniform_convergence-res1} holds, so that we have proven \eqref{uniform_convergence_3} in Theorem \ref{THEOREM_uniform_convergence}. Moreover, for all $L \in \R$ we have that \eqref{uniform_convergence-res2} holds, which combined with \eqref{uniform_convergence_2} (which also holds for all $L \in \R$) gives \eqref{uniform_convergence_4} and completes the proof.
\end{proof}
\subsection{Proof of Theorem \ref{THEOREM_carac_speed} completed}
Assume that \ref{asu_perturbation} and \ref{asu_convergence} hold. Arguing as in the proof of Theorem \ref{THEOREM_main_TW}, we have that the assumptions of Theorem \ref{THEOREM_abstract_speed} are fulfilled if we choose as in Lemma \ref{LEMMA_implication}. Notice that Theorem \ref{THEOREM_carac_speed} is exactly Theorem \ref{THEOREM_abstract_speed} if we choose the abstract objects as in Lemma \ref{LEMMA_implication}. Therefore, Theorem \ref{THEOREM_carac_speed} is established.
\qed
\printbibliography

\end{document}